\documentclass [12pt]{amsart}
\usepackage[utf8]{inputenc}
\pdfoutput=1
\usepackage{amsmath,amssymb,amsthm,amsfonts}
\usepackage{mathtools}%

\usepackage[english]{babel}%
\usepackage{comment}%
\usepackage{hyperref}%
\usepackage{bm}%
\usepackage{bbm}%
\usepackage{mathrsfs}%

\usepackage{tikz,graphicx,color}
\usepackage{tikz-cd}%
\usepackage{tikz-3dplot}%
\usetikzlibrary{calc}%
\usetikzlibrary{arrows}%
\usetikzlibrary{shapes}%
\usetikzlibrary{patterns}%
\usetikzlibrary{positioning}%
\usetikzlibrary{arrows.meta}
\usetikzlibrary{decorations.markings}
\usetikzlibrary{knots}

\usepackage{epstopdf}%

\usepackage[arrow]{xy}%
\usepackage{diagbox}%
\usepackage{subfig}%
\usepackage{arcs}%
\usepackage{xcolor}%
\usepackage{xspace}

\usepackage[margin=1.0in]{geometry}%

\usepackage{enumitem}
\usepackage{letltxmacro}
\usepackage{thmtools,etoolbox}

\def\myarabic#1{\normalfont(\roman{#1})}
\newlist{theoremlist}{enumerate}{1}
\setlist[theoremlist]{label=\myarabic{theoremlisti},ref={\myarabic{theoremlisti}},itemindent=0pt,labelindent=0pt,
  leftmargin=*,noitemsep}

\makeatletter
\renewcommand{\p@theoremlisti}{\perh@ps{\thetheorem}}
\protected\def\perh@ps#1#2{\textup{#1#2}}
\newcommand{\itemrefperh@ps}[2]{\textup{#2}}
\newcommand{\itemref}[1]{\begingroup\let\perh@ps\itemrefperh@ps\ref{#1}\endgroup}
\makeatother

\usepackage{nameref,hyperref}
\usepackage[capitalize]{cleveref}

\newtheorem{theorem}{Theorem}[section]

\newtheorem{lemma}[theorem]{Lemma}
\newtheorem{proposition}[theorem]{Proposition}
\newtheorem{corollary}[theorem]{Corollary}
\theoremstyle{definition}
\newtheorem{notation}[theorem]{Notation}
\theoremstyle{definition}
\newtheorem{remark}[theorem]{Remark}
\theoremstyle{definition}
\newtheorem{definition}[theorem]{Definition}
\newtheorem{conjecture}[theorem]{Conjecture}
\newtheorem{question}[theorem]{Question}
\theoremstyle{definition}
\newtheorem{problem}[theorem]{Problem}
\theoremstyle{definition}
\newtheorem{example}[theorem]{Example}

\crefname{figure}{Figure}{Figures}

\def\figref#1(#2){Figure~\hyperref[#1]{\ref*{#1}(#2)}}

\addtotheorempostheadhook[theorem]{\crefalias{theoremlisti}{theorem}}
\addtotheorempostheadhook[lemma]{\crefalias{theoremlisti}{lemma}}
\addtotheorempostheadhook[proposition]{\crefalias{theoremlisti}{proposition}}
\addtotheorempostheadhook[corollary]{\crefalias{theoremlisti}{corollary}}

\def\Acal{\mathcal{A}}\def\Bcal{\mathcal{B}}\def\Fcal{\mathcal{F}}\def\Ical{\mathcal{I}}\def\Mcal{\mathcal{M}}\def\Xcal{\mathcal{X}}

\def\C{{\mathbb{C}}}
\def\R{{\mathbb{R}}}
\def\N{{\mathbb{N}}}
\def\Z{{\mathbb{Z}}}

\newcommand\parr[1]{{({#1})}}
\def\<{{\langle}}
\def\>{{\rangle}}
\def\eps{{\epsilon}}
\def\id{\operatorname{id}}

\def\Span{ \operatorname{Span}}

\def\wt{\operatorname{wt}}

\def\xing{{\operatorname{xing}}}
\def\supp{\operatorname{supp}}

\def\Cast{\C^\ast}
\def\xrasim{\xrightarrow{\sim}}

\def\GL{\operatorname{GL}}

\def\Gr{\operatorname{Gr}}
\def\Grtnn{\Gr_{\ge 0}}
\def\Grtp{\Gr_{>0}}
\def\alt{\operatorname{alt}}

\def\Pio{\Pi^\circ}
\def\Povtp_#1{\Pi_{#1}^{>0}}
\def\Povtnn_#1{\Pi_{#1}^{\geq0}}
\def\BND{\Bcal}
\def\Bound{\BND}%
\def\Boundkn{\BND(k,n)}

\newcounter{todocnt} %
\newcounter{todoex} %

\newcounter{todofigure}

\numberwithin{equation}{section}

\def\tikzscl{1}

\def\M{M}

\def\Reg{R}

\def\curve{\bm{\gamma}}
\def\bth{{\bm{\theta}}}
\def\btht{{\tilde{\bm{\theta}}}}

\def\g{\gamma}
\def\J{J}

\def\Tiling{{\mathbb{T}}}

\crefname{figure}{Figure}{Figures}

\def\Jt{{\tilde\J}}

\def\th{\theta}
\def\tht{{\tilde\theta}}

\def\eps{\varepsilon}

\def\ggg{\Gamma}
\def\GGG{{\bm{\Gamma}}}
\def\gggt{\tilde\ggg}
\def\gg{\gamma^\C}
\def\curvec{\curve^\C}
\def\curvecd{{\widehat{\curve}}^\C}
\def\ggd{\widehat{\gamma}^\C}

\def\gt{\tilde\g}
\def\ggt{{\tilde\gamma}^\C}
\def\ggp_#1{\gamma^{\C}_{f',#1}}

\def\u#1{{u_\ft^\parr{#1}}}

\def\Ising{{\operatorname{Ising}}}
\def\elec{{\operatorname{elec}}}
\def\La{\Lambda_\elec}
\def\la{\Lambda}
\def\M{M_{\Ising}}
\def\match{\Acal}
\def\pmatch{I_\Acal}

\def\RowSpan{\operatorname{RowSpan}}

\def\Ptp_#1{\Pi^{>0}_{#1}}
\def\Ptnn_#1{\Pi^{\geq0}_{#1}}

\def\fb{{\bar f}}
\def\t{u}

\def\Measop{\operatorname{Meas}}

\def\Meas(#1,#2){\Measop_{#1}(#2)}
\def\Measp(#1,#2){\Measop'_{#1}(#2)}

\def\ft{f}
\def\Bkn{\Boundkn}

\def\m#1{b_{#1}^{-}}
\def\p#1{b_{#1}^{+}}
\def\md#1{d_{#1}^{-}}
\def\pd#1{d_{#1}^{+}}
\def\Crit{\operatorname{Crit}}
\def\Ctp{\Crit^{>0}}
\def\Ctpkn{\Crit^{>0}_{k,n}}
\makeatletter
\newcommand{\raisemath}[1]{\mathpalette{\raisem@th{#1}}}
\newcommand{\raisem@th}[3]{\raisebox{#1}{$#2#3$}}
\makeatother
\def\Ctpd{\Critd^{\raisemath{-3.5pt}{>0}}}

\def\J{J}

\def\eps{\epsilon}

\def\Ctnn{\Crit^{\geq0}}

\def\N{N}
\def\fel{f^\elec_\Reg}
\def\fis{f^\Ising_\Reg}
\def\felt{f^\elec_\tau}
\def\fist{f^\Ising_\tau}
\def\feltr{f^\elec_{\taur}}
\def\fistr{f^\Ising_{\taur}}
\def\permel{\bar{f}^\elec_\Reg}
\def\permis{\bar{f}^\Ising_\Reg}
\def\pel{\phi^\elec}
\def\pis{\phi^\Ising}

\def\br[#1]{[\![#1]\!]}
\def\brx[#1]{(#1)}
\def\dual#1{\widehat{#1}}
\def\wtd{\dual\wt}
\def\fd{\dual f}

\def\altp{\alt^\perp}
\def\xrasim{\xrightarrow{\sim}}

\def\bt{{\mathbf{t}}}
\def\t{t}

\def\Tspace{\Theta^\circ}

\def\CCrit{\Crit}
\def\Cio{\CCrit^{\circ}}

\def\ft{{f,\bt}}
\def\v{v}
\def\Mis{\operatorname{Mis}}

\def\Misud{\Mis^\updown}
\def\Misdu{\Mis^\downup}

\def\MisDdu{\MisD^\downup}
\newcommand{\cev}[1]{\reflectbox{\ensuremath{\vec{\reflectbox{\ensuremath{#1}}}}}}

\def\H#1#2{\br[t_{#1},t_{#2}]}
\def\HT#1{\br[\t,t_{#1}]}
\def\Hun#1#2{\br[t_{#1},t_{#2}]_+}

\def\Critkn{\Crit_{k,n}}

\def\CioknR{\Cio_{k,n}(\R)}
\def\CioR_#1{\Cio_{#1}(\R)}

\def\TspaceR{\Theta^{\R}}
\def\TspaceRkn{\Theta^{\R}_{k,n}}

\def\fkn{{f_{k,n}}}

\def\GX{G^{\times}}
\def\GXV{\vec G^{\times}}
\def\conn_#1{c_{#1}}
\def\dimf{d_f}

\def\D_#1{D_{#1}}

\def\THtp{\Theta^{>0}}

\def\THtpd{\dual{\Theta}^{>0}}

\def\Hyp_#1{\Delta_{#1}}
\def\Cyct_#1{\overline{\operatorname{Cyc}}_{#1}}

\def\Cl(#1){#1^\boxtimes}

\let\ge\geqslant
\let\geq\geqslant

\let\leq\leqslant

\def\Meascl(#1,#2){{\overline{\Measop}}_{#1}(#2)}

\def\dsh#1{#1^\downarrow}
\def\Rtp{\R_{>0}}
\def\Gau{\operatorname{Gauge}}
\def\bGau{\Gau^{\hspace{0.01in}\raisebox{-2pt}{\tribl}}}
\def\wGau{\Gau^{\hspace{0.01in}\raisebox{-2pt}{\triwh}}}

\def\Pmid_#1{\Pi^{>0}_{#1,\dsh{#1}}}

\def\Measf{\Measop_f}

\def\UW{\operatorname{UW}}

\def\aperm{an $f$}
\def\perm{\fb}
\def\tt{\tilde{t}}
\def\btt{\tilde{\bt}}

\def\bthr{\bth^{\operatorname{reg}}}
\def\paragraph#1{\subsubsection*{#1}}

\def\pkn{{\perm_{k,n}}}

\def\pf{f}

\def\Xis{\Xcal^\Ising}
\def\Xel{\Xcal^\elec}
\def\shift{S}

\def\Gd{\widehat{G}}
\def\shcomb{\sigma}
\def\shf{\shcomb^{-1} f\shcomb}
\def\shbtht{\btht\circ \shcomb}
\def\str{S}
\def\cpt{x}
\def\Gtop{\vec{G}^{\operatorname{top}}}
\def\cyc{C}
\def\Conc{C}
\def\pme{\varepsilon}
\def\npm{[n]_{\pme}}
\def\Cpm{\Conc_\pme}
\def\Apm{A_\pme}
\def\Bpm{B_\pme}

\def\It{\tilde I}

\def\ctau{\cev\tau}
\def\updown{{\uparrow\!\downarrow}}
\def\downup{{\downarrow\!\uparrow}}
\def\Mleft{\overleftarrow{M}}

\def\ncrs{x}
\def\nc{y}
\def\common{R}
\def\bZ{\mathbf{Z}}

\def\Tspaced{\widehat\Theta^\circ}

\def\oddop{\operatorname{odd}}
\def\evop{\operatorname{even}}
\def\nodd{[n]_{\oddop}}
\def\nev{[n]_{\evop}}
\def\taur{\tau_\Reg}

\def\b{B}
\def\Avec{D_\Acal}
\def\Gel{G^{\elec}}
\def\Gis{G^{\Ising}}
\def\wtel{\wt^{\elec}}
\def\wtis{\wt^{\Ising}}

\def\XO#1#2{X_0^\parr{#1,#2}}
\def\XOkn{\XO kn}
\def\Mb{(M2)$^\tribl$\xspace}
\def\Mw{(M2)$^\triwh$\xspace}
\def\Maw{(M1)$^\triwh$\xspace}
\def\Mcb{(M3)$^\tribl$\xspace}

\def\Gall{\mathcal{G}}
\def\Gred{\Gall_{\operatorname{red}}}

\newcommand*\tribl{{\includegraphics[height=.6em]{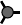}}}
\newcommand*\triwh{{\includegraphics[height=.6em]{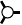}}}
\def\WTV{\Gred^\triwh}
\def\BTV{\Gred^\tribl}

\def\Gkn{G_{k,n}}
\def\Gknd{G_{k-1,n}}

\def\Vwh{V_\circ}
\def\Mscl{1}
\def\Mbsmall{\scalebox{\Mscl}{\Mb}}
\def\Mcbsmall{\scalebox{\Mscl}{\Mcb}}
\def\Mwsmall{\scalebox{\Mscl}{\Mw}}
\def\Mawsmall{\scalebox{\Mscl}{\Maw}}

\def\crat(#1,#2;#3,#4){(\v_{#1},\v_{#2};\v_{#3},\v_{#4})}

\def\Pet{\operatorname{Pet}}
\def\parag#1{\subsection{#1}}

\def\green{green!80!black}

\makeatletter
\@namedef{subjclassname@2020}{\textup{2020} Mathematics Subject Classification}
\makeatother

\def\BT{\mathbf{T}}
\def\GH_#1{G^{\Mcal}_{#1}}

\def\Measdop{\widehat{\Measop}}
\def\Measd(#1,#2){\Measdop_{#1}(#2)}
\def\Ciod{\widehat{\Crit}^{\raisemath{-3.5pt}{\circ}}}

\def\Ctpd{\widehat{\Crit}^{\raisemath{-3.5pt}{>0}}}
\def\MisD{\widehat{\Mis}}

\begin{document}
\numberwithin{equation}{section}

\title{Critical varieties in the Grassmannian}
\author{Pavel Galashin}
\address{Department of Mathematics, University of California, Los Angeles, CA 90095, USA}
\email{{\href{mailto:galashin@math.ucla.edu}{galashin@math.ucla.edu}}}
\thanks{P.G.\ was supported by an Alfred P. Sloan Research Fellowship and by the National Science Foundation under Grants No.~DMS-1954121 and No.~DMS-2046915.}
\date{\today}

\subjclass[2020]{
  Primary:
  14M15. %
  Secondary:
  15B48, %
  82B27, %
  05E99. %
}

\keywords{Critical varieties, totally nonnegative Grassmannian, positroids, critical dimer model, Ising model, electrical networks, zonotopal tilings}

\begin{abstract}
We introduce a family of spaces called \emph{critical varieties}. Each critical variety is a subset of one of the positroid varieties in the Grassmannian. The combinatorics of positroid varieties is captured by the dimer model on a planar bipartite graph $G$, and the critical variety is obtained by restricting to Kenyon's critical dimer model associated to a family of isoradial embeddings of $G$. This model is invariant under square/spider moves on $G$, and we give an explicit boundary measurement formula for critical varieties which does not depend on the choice of $G$. This extends our recent results for the critical Ising model, and simultaneously also includes the case of critical electrical networks.

We systematically develop the basic properties of critical varieties. In particular, we study their real and totally positive parts, the combinatorics of the associated strand diagrams, and introduce a shift map motivated by the connection to zonotopal tilings and scattering amplitudes.
\end{abstract}

\maketitle

\hypersetup{bookmarksdepth=2}
\setcounter{tocdepth}{1}
\tableofcontents

\section*{Introduction}
The \emph{totally nonnegative Grassmannian} $\Grtnn(k,n)$ is a remarkable space introduced by Lusztig~\cite{Lus2,LusIntro} and Postnikov~\cite{Pos}, who described a stratification of $\Grtnn(k,n)$ into \emph{positroid cells}. Building on Postnikov's work, Knutson--Lam--Speyer~\cite{KLS} studied \emph{positroid varieties} which are Zariski closures of positroid cells. These objects have been studied extensively in the recent years, making surprising appearances in such fields as 
 the physics of scattering amplitudes~\cite{AHT,abcgpt}, 
 knot theory~\cite{FPST,STWZ,GL2}, and
 statistical mechanics~\cite{CoWi,Lam,GP}. In fact, this paper is directly influenced by ideas from statistical mechanics: our goal is to apply them to introduce \emph{critical parts} of positroid varieties and study their properties from the point of view of total positivity.

Postnikov gave a parametrization of each positroid cell using a weighted planar bipartite graph $G$ in a disk. His construction was later recast in~\cite{Talaska,PSW} in terms of the dimer model on $G$. Allowing arbitrary positive real edge weights of $G$ parametrizes the whole positroid cell. To obtain a \emph{critical cell} (which is the ``totally positive part'' of the corresponding critical variety), one restricts to the \emph{critical dimer model} on $G$, introduced by Kenyon~\cite{Kenyon}. Special cases of the critical dimer model yield 
 Baxter's critical $Z$-invariant Ising model~\cite{Bax,Bax2} and critical electrical resistor networks as discussed e.g. in~\cite[Section~6]{Kenyon}. Our construction is compatible with the recently discovered embeddings of the planar Ising model~\cite{GP} and electrical networks~\cite{Lam} into $\Grtnn(k,n)$.

Strictly speaking, Kenyon's critical dimer model is attached not just to a planar bipartite graph $G$, but to an isoradial embedding~\cite{Mercat} of $G$. An embedding is called \emph{isoradial} if every interior face of $G$ is inscribed in a circle of radius $1$. The main observation that led to our below results was that the graphs appearing in Postnikov's parametrizations of positroid varieties admit natural isoradial embeddings known as \emph{plabic tilings}, introduced by Oh--Postnikov--Speyer~\cite{OPS}. While the critical dimer model and its connections to the Ising model and electrical networks are well known, the specialization to plabic tilings and the totally nonnegative Grassmannian appears to not have been studied before.

A given positroid cell can be parametrized by many different planar bipartite graphs, all of which are related by  \emph{square moves} (also known as spider moves or urban renewals)~\cite{KPW,Pos}. A crucial feature of the critical dimer model is that its \emph{boundary measurements} are unchanged under square moves. An important consequence for our purposes is that \emph{the critical variety depends only on the ambient positroid variety}, and not on a particular choice of the planar bipartite graph $G$. In fact, we give a simple explicit formula for the boundary measurement map that does not depend on the choice of $G$, generalizing our previous results~\cite{ising_crit} for the critical Ising model. 

We initiate a systematic study of critical varieties, which aims to be parallel to the well-developed theory of positroid varieties. We prove many results in different directions, some of which are highlighted below.

\section{Main results}\label{sec:main-results}
We explain our results and constructions, roughly following the order in which they appear in the main body of the paper.

\subsection{Planar bipartite graphs}\label{sec:intro:plan-bipart-graphs}
We start by giving a brief background on the totally nonnegative Grassmannian. See \cref{sec:background} for further details.

Let $G$ be a planar bipartite graph embedded in a disk. We assume that $G$ has $n$ black boundary vertices, each of degree $1$, labeled $b_1,b_2,\dots,b_n$ in clockwise order. A \emph{strand} (or a \emph{zig-zag path}) in $G$ is a path that makes a sharp right turn at each black vertex and a sharp left turn at each white vertex. Thus $G$ gives rise to a \emph{strand permutation} $\perm_G\in S_n$: for each $1\leq p\leq n$, the strand that starts at $b_p$ terminates at $b_{\perm_G(p)}$. See \figref{fig:plabic}(b). We say that $G$ is \emph{reduced}~\cite{Pos} if it has the minimal number of faces among all graphs with the same strand permutation. It is known that a reduced graph contains no closed strands, thus each strand starts and ends at the boundary of $G$. When $p$ is a fixed point of $\perm_G$ (i.e., $\perm_G(p)=p$), we assume that $b_p$ is incident to an interior white vertex of degree $1$.

\begin{remark}\label{rmk:intro:S_n_loopless}
There is a bijection $\perm\mapsto f$ between permutations $\perm\in S_n$ and \emph{loopless bounded affine permutations} $f$ defined in \cref{sec:BAP}. For a permutation $\perm\in S_n$, the map $f:\Z\to\Z$ is uniquely determined by the conditions $f(p+n)=f(p)+n$ and $p<f(p)\leq p+n$ for all $p\in\Z$ together with $f(p)\equiv\perm(p)\pmod n$ for all $1\leq p\leq n$.  We will use the permutation $\perm$ to construct a critical variety, but we will label it by $f$ in order to match the labeling of positroid varieties.
\end{remark}

Given a reduced planar bipartite graph $G$, one can consider the dimer model on it. Let us assign a positive real weight $\wt(e)$ to each edge $e$ of $G$. An \emph{almost perfect matching} $\match$ of $G$ is a collection of edges of $G$ that uses each interior vertex exactly once, and uses some subset of boundary vertices. We denote by $\{b_p\}_{p\in\pmatch}$ for $\pmatch\subset[n]:=\{1,2,\dots,n\}$ the set of boundary vertices used by $\match$. It is easy to check that there exists an integer $1\leq k\leq n$ such that any almost perfect matching of $G$ satisfies $|\pmatch|=k$. The number $k$ depends only on $\pf_G$.

\begin{figure}

\def\wh#1(#2){\node[draw,circle,scale=0.4,fill=white] (#1) at (#2) {};}
\def\bl#1(#2){\node[draw,circle,scale=0.4,fill=black!40] (#1) at (#2) {};}
\def\blbnd#1(#2){\node[draw,circle,scale=0.25,fill=black!60] (#1) at (#2) {};}
\def\rad{3}
\def\bnd#1#2{
\blbnd{B#2}(#1:\rad)
\node[anchor=180+#1,scale=0.8] (T#2) at (#1:\rad) {$b_{#2}$};
}
\def\lw{1pt}
\def\edgecl{black}
\def\edg#1#2{\draw[line width=\lw,\edgecl] (#1)--(#2);}
\def\RAD{2.4}
\def\wei#1#2(#3)#4{
\node[scale=0.5,blue,anchor=#4,inner sep=2pt] (ZZZ) at (#3) {$\brx[#1#2]$};
}
\def\dang{10}
\def\strandlw{1pt}
\def\tikzscl{0.5}
\scalebox{0.97}{
\makebox[1.0\textwidth]{
\begin{tabular}{cccc}

\begin{tikzpicture}[scale=\tikzscl,baseline=(ZUZU.base)]
\coordinate(ZUZU) at (0,0);
\draw[line width=0.5,black!20,dashed] (0,0) circle (\rad);
\wh A(-1,1)
\bl B(1,1)
\wh C(1,-1)
\bl D(-1,-1)
\bnd{135}{1}
\bnd{45}{2}
\bnd{-45}{3}
\bnd{-135}{4}
\draw[line width=\lw,\edgecl] (B1)--(A)--(B)--(C)--(D)--(A);
\draw[line width=\lw,\edgecl] (B)--(B2);
\draw[line width=\lw,\edgecl] (C)--(B3);
\draw[line width=\lw,\edgecl] (D)--(B4);

\wh X(45:\RAD)
\wh Y(-135:\RAD)
\end{tikzpicture}

 & 
\def\lw{0.5pt}
\def\edgecl{black!50}
\begin{tikzpicture}[scale=\tikzscl,baseline=(ZUZU.base)]
\coordinate(ZUZU) at (0,0);
\draw[line width=0.5,black!20,dashed] (0,0) circle (\rad);
\wh A(-1,1)
\bl B(1,1)
\wh C(1,-1)
\bl D(-1,-1)
\bnd{135}{1}
\bnd{45}{2}
\bnd{-45}{3}
\bnd{-135}{4}
\draw[line width=\lw,\edgecl] (B1)--(A)--(B)--(C)--(D)--(A);
\draw[line width=\lw,\edgecl] (B)--(B2);
\draw[line width=\lw,\edgecl] (C)--(B3);
\draw[line width=\lw,\edgecl] (D)--(B4);

\draw [blue,->,line width=\strandlw] plot [smooth] coordinates {(135-\dang:\rad) (-0.6,1.2) (0.7,0.7) (1.2,-0.6) (-45+\dang:\rad)};
\draw [red,->,line width=\strandlw] plot [smooth] coordinates {(-135-\dang:\rad) (-135-\dang:\RAD) (-0.6,-1.2) (0.7,-0.7) (1.2,0.6) (45+\dang:\RAD) (45+\dang:\rad) };

\begin{scope}[xscale=-1,yscale=-1]
\draw [brown,->,line width=\strandlw] plot [smooth] coordinates {(135-\dang:\rad) (-0.6,1.2) (0.7,0.7) (1.2,-0.6) (-45+\dang:\rad)};
\draw [\green,->,line width=\strandlw] plot [smooth] coordinates {(-135-\dang:\rad) (-135-\dang:\RAD) (-0.6,-1.2) (0.7,-0.7) (1.2,0.6) (45+\dang:\RAD) (45+\dang:\rad) };
\end{scope}

\wh X(45:\RAD)
\wh Y(-135:\RAD)
\end{tikzpicture}

 &

\begin{tikzpicture}[scale=\tikzscl,baseline=(ZUZU.base)]
\coordinate(ZUZU) at (0,0);
\draw[line width=0.5,black!20,dashed] (0,0) circle (\rad);
\wh A(-1,1)
\bl B(1,1)
\wh C(1,-1)
\bl D(-1,-1)
\bnd{135}{1}
\bnd{45}{2}
\bnd{-45}{3}
\bnd{-135}{4}
\draw[line width=\lw,\edgecl] (B1)--(A)--(B)--(C)--(D)--(A);
\draw[line width=\lw,\edgecl] (B)--(B2);
\draw[line width=\lw,\edgecl] (C)--(B3);
\draw[line width=\lw,\edgecl] (D)--(B4);
\wh X(45:\RAD)
\wh Y(-135:\RAD)
\wei 24(-1.35,-1.35){north west}
\wei 12(0,-1){south}
\wei 14(-1,0){east}
\wei 23(1,0){west}
\wei 34(0,1){north}
\wei 24(1.35,1.35){south east}

\end{tikzpicture}
&

\scalebox{0.8}{
\def\rbnd#1(#2){\node[draw=red,circle,scale=0.25,fill=red!60] (#1) at (#2) {};}
\def\bnr#1#2{
\rbnd{B#2}(#1:\rad)
\node[anchor=180+#1,scale=0.8,red] (TT#2) at (#1:\rad) {$#2$};
}
\begin{tikzpicture}[baseline=(A.base)]
\node(A) at (0,0){
\begin{tabular}{l}
$\Delta_{12}=\brx[23]\cdot \brx[24]$ \\
$\Delta_{23}=\brx[34]\cdot \brx[24]$ \\
$\Delta_{34}=\brx[14]\cdot \brx[24]$ \\
$\Delta_{14}=\brx[12]\cdot \brx[24]$ \\
$\Delta_{13}=\brx[24]\cdot \brx[24]$ \\
$\Delta_{24}=\brx[14]\cdot \brx[23]+\brx[12]\cdot \brx[34]\textcolor{red}{=}\brx[13]\cdot \brx[24]$ \\
\end{tabular}
};
\begin{scope}[shift={(1.5,0.7)}]
\def\rad{0.7}
\draw[red,dashed] (0,0) circle (\rad);
\bnr{0}1
\bnr{50}2
\bnr{100}3
\bnr{210}4
\draw[red](B1)--(B2)--(B3)--(B4)--(B1)--(B3);
\draw[red](B2)--(B4);
\end{scope}
\node[red,inner sep=1pt] (B) at (1.75,-0.4) {by Ptolemy's theorem};
\draw[red,line width=0.5pt,->] (B.south) --(1.75,-1.1);
\end{tikzpicture}

}
\\
(a) & (b) & (c) & (d)
\end{tabular}
}
}

  \caption{\label{fig:plabic} (a) A (reduced) planar bipartite graph $G$; (b) strands in $G$; (c) edge weights $\wt_\bth$, where the unmarked edges have weight $1$ and we abbreviate $\brx[pq]:=\sin(\th_q-\th_p)$; (d) the boundary measurements~${\Delta_I(G,\wt_\bth)}$.}
\end{figure}
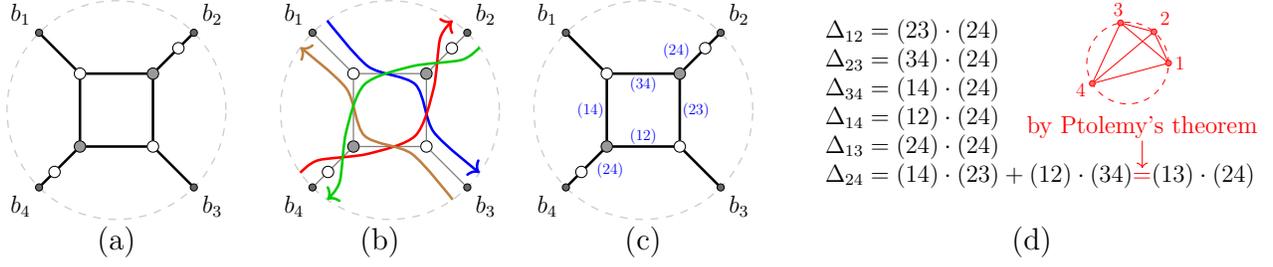

 Denote the set of $k$-element subsets of $[n]$ by ${[n]\choose k}$, and for $I\in{[n]\choose k}$, define
\begin{equation}\label{eq:intro:Delta_I}
 \Delta_I(G,\wt):=\sum_{\match:\;\pmatch=I} \wt(\match),\quad\text{where}\quad \wt(\match):=\prod_{e\in\match} \wt(e).
\end{equation}
Here the summation runs over almost perfect matchings of $G$. We consider the tuple $(\Delta_I(G,\wt))_{I\in{[n]\choose k}}$ to be defined up to multiplication by a common scalar. The \emph{boundary measurements} $\Meas(G,\wt):=(\Delta_I(G,\wt))_{I\in{[n]\choose k}}$ give rise to a point in the \emph{totally nonnegative Grassmannian} $\Grtnn(k,n)$. The \emph{Grassmannian} $\Gr(k,n)$ is the set of all linear $k$-dimensional subspaces of $\C^n$. Each such subspace $V$ is the row span of a full rank $k\times n$ matrix $A$, and the \emph{Pl\"ucker coordinates} of $V$ are by definition the maximal minors of $A$. Pl\"ucker coordinates are defined up to multiplication by a common nonzero scalar, and $\Grtnn(k,n)$ is the subset of $\Gr(k,n)$ where the ratio of any two nonzero Pl\"ucker coordinates is a positive real number. 
We have \emph{positroid stratifications}~\cite{Pos,KLS}
\begin{equation*}%
  \Grtnn(k,n)=\bigsqcup_{f\in\Bkn} \Ptp_f \quad\text{and}\quad \Gr(k,n)=\bigsqcup_{f\in\Bkn} \Pio_f,
\end{equation*}
where $\Bkn$ is the set of \emph{$(k,n)$-bounded affine permutations}; see \cref{dfn:BAP}. The image of the map $\Measop_G$ is precisely the positroid cell $\Ptp_{f_G}$. The positroid stratification contains a unique open dense cell (called the \emph{top cell}) labeled by $\fkn\in\Bkn$. The map $\fkn:\Z\to\Z$ sends $p\mapsto p+k$ for all $p\in \Z$, and the permutation $\pkn\in S_n$ sends $p\mapsto p+k$ modulo $n$ for all $p\in[n]$. An example for $k=2$, $n=4$ is shown in \figref{fig:plabic}(b).

\emph{Square moves} are certain transformations of $(G,\wt)$ which preserve the boundary measurements; see \figref{fig:move}(a) for an example.  
 The weights of the edges are changed appropriately (\cref{fig:sq-mv}), and the resulting weighted graph $(G',\wt')$ satisfies $\Meas(G,\wt)=\Meas(G',\wt')$ and $\pf_G=\pf_{G'}$. Conversely, any two reduced planar bipartite graphs $G$ and $G'$ satisfying $\pf_G=\pf_{G'}$ can be related by a sequence of square moves.

\begin{figure}

\def\wh#1(#2){\node[draw,circle,scale=0.4,fill=white] (#1) at (#2) {};}
\def\bl#1(#2){\node[draw,circle,scale=0.4,fill=black!40] (#1) at (#2) {};}
\def\blbnd#1(#2){\node[draw,circle,scale=0.25,fill=black!60] (#1) at (#2) {};}
\def\rad{3}
\def\bnd#1#2{
\blbnd{B#2}(#1:\rad)
\node[anchor=180+#1,scale=0.8] (T#2) at (#1:\rad) {$b_{#2}$};
}
\def\lw{1pt}
\def\edgecl{black}
\def\edg#1#2{\draw[line width=\lw,\edgecl] (#1)--(#2);}
\def\RAD{2.4}
\def\wei#1#2(#3)#4{
\node[scale=0.5,blue,anchor=#4,inner sep=2pt] (ZZZ) at (#3) {$\brx[#1#2]$};
}
\def\dang{10}
\def\strandlw{1pt}
\def\tikzscl{0.5}
\begin{tabular}{c|ccc}

\begin{tabular}{ccc}

\begin{tikzpicture}[scale=\tikzscl,baseline=(ZUZU.base)]
\coordinate(ZUZU) at (0,0);
\draw[line width=0.5,black!20,dashed] (0,0) circle (\rad);
\wh A(-1,1)
\bl B(1,1)
\wh C(1,-1)
\bl D(-1,-1)
\bnd{135}{1}
\bnd{45}{2}
\bnd{-45}{3}
\bnd{-135}{4}
\draw[line width=\lw,\edgecl] (B1)--(A)--(B)--(C)--(D)--(A);
\draw[line width=\lw,\edgecl] (B)--(B2);
\draw[line width=\lw,\edgecl] (C)--(B3);
\draw[line width=\lw,\edgecl] (D)--(B4);
\wh X(45:\RAD)
\wh Y(-135:\RAD)
\wei 24(-1.35,-1.35){north west}
\wei 12(0,-1){south}
\wei 14(-1,0){east}
\wei 23(1,0){west}
\wei 34(0,1){north}
\wei 24(1.35,1.35){south east}

\end{tikzpicture}
&
                    \scalebox{1.5}{$\leftrightarrow$}
&

\begin{tikzpicture}[scale=\tikzscl,baseline=(ZUZU.base)]
\coordinate(ZUZU) at (0,0);
\draw[line width=0.5,black!20,dashed] (0,0) circle (\rad);
\bl A(-1,1)
\wh B(1,1)
\bl C(1,-1)
\wh D(-1,-1)
\bnd{135}{1}
\bnd{45}{2}
\bnd{-45}{3}
\bnd{-135}{4}
\draw[line width=\lw,\edgecl] (B1)--(A)--(B)--(C)--(D)--(A);
\draw[line width=\lw,\edgecl] (B)--(B2);
\draw[line width=\lw,\edgecl] (C)--(B3);
\draw[line width=\lw,\edgecl] (D)--(B4);
\wh X(135:\RAD)
\wh Y(-45:\RAD)
\wei 13(1.35,-1.35){north east}
\wei 13(-1.35,1.35){south west}
\wei 34(0,-1){south}
\wei 23(-1,0){east}
\wei 14(1,0){west}
\wei 12(0,1){north}
\end{tikzpicture}

\end{tabular}

&

\begin{tabular}{ccc}

\def\lw{0.5pt}
\def\edgecl{black!50}
\begin{tikzpicture}[scale=\tikzscl,baseline=(ZUZU.base)]
\coordinate(ZUZU) at (0,0);
\draw[line width=0.5,black!20,dashed] (0,0) circle (\rad);
\wh A(-1,1)
\bl B(1,1)
\wh C(1,-1)
\bl D(-1,-1)
\bnd{135}{1}
\bnd{45}{2}
\bnd{-45}{3}
\bnd{-135}{4}
\draw[line width=\lw,\edgecl] (B1)--(A)--(B)--(C)--(D)--(A);
\draw[line width=\lw,\edgecl] (B)--(B2);
\draw[line width=\lw,\edgecl] (C)--(B3);
\draw[line width=\lw,\edgecl] (D)--(B4);

\draw [blue,->,line width=\strandlw] plot [smooth] coordinates {(135-\dang:\rad) (-0.6,1.2) (0.7,0.7) (1.2,-0.6) (-45+\dang:\rad)};
\draw [red,->,line width=\strandlw] plot [smooth] coordinates {(-135-\dang:\rad) (-135-\dang:\RAD) (-0.6,-1.2) (0.7,-0.7) (1.2,0.6) (45+\dang:\RAD) (45+\dang:\rad) };

\begin{scope}[xscale=-1,yscale=-1]
\draw [brown,->,line width=\strandlw] plot [smooth] coordinates {(135-\dang:\rad) (-0.6,1.2) (0.7,0.7) (1.2,-0.6) (-45+\dang:\rad)};
\draw [\green,->,line width=\strandlw] plot [smooth] coordinates {(-135-\dang:\rad) (-135-\dang:\RAD) (-0.6,-1.2) (0.7,-0.7) (1.2,0.6) (45+\dang:\RAD) (45+\dang:\rad) };
\end{scope}

\wh X(45:\RAD)
\wh Y(-135:\RAD)
\end{tikzpicture}

&
                    \scalebox{1.5}{$\rightarrow$}
&

\begin{tikzpicture}[scale=\tikzscl,baseline=(ZUZU.base)]
\coordinate(ZUZU) at (0,0);
\draw[line width=0.5,black!20,dashed] (0,0) circle (\rad);
\bnd{135}{1}
\bnd{45}{2}
\bnd{-45}{3}
\bnd{-135}{4}

\draw [blue,->,line width=\strandlw] (135-\dang:\rad) -- (-45+\dang:\rad);
\draw [red,->,line width=\strandlw] (-135-\dang:\rad) -- (45+\dang:\rad);

\begin{scope}[xscale=-1,yscale=-1]
\draw [brown,->,line width=\strandlw] (135-\dang:\rad) -- (-45+\dang:\rad);
\draw [\green,->,line width=\strandlw] (-135-\dang:\rad) -- (45+\dang:\rad);
\end{scope}

\end{tikzpicture}

\end{tabular}

\\
(a)  & (b) 
\end{tabular}

  \caption{\label{fig:move} (a) A square move and its effect on $\wt_\bth$; (b) converting a plabic graph $G$ into a reduced strand diagram of $\pf_G$ from \cref{dfn:intro:strand_diag}.}
\end{figure}
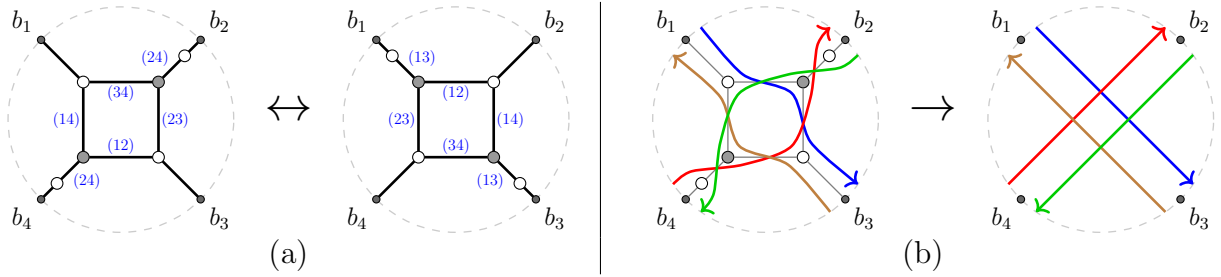

\subsection{Critical dimer model}\label{sec:intro:critical-dimer-model}
Let $G$ be a reduced planar bipartite graph with strand permutation $\perm$.  Choose a tuple $\bth=(\th_1,\th_2,\dots,\th_n)\in\R^n$. For now, we assume that $\th_1<\th_2<\dots<\th_n<\th_1+\pi$; this condition will be weakened in \cref{sec:intro-critical-varieties}. We define a weight function $\wt_\bth$ on the edges of $G$ as follows. Observe that every edge $e$ of $G$ belongs to exactly two strands. Suppose that one strand terminates at $b_p$ and the other strand terminates at $b_q$ for some $1\leq p<q\leq n$. In this case, we say that $e$ is \emph{labeled by $\{p,q\}$}. We set %
\begin{equation}\label{eq:wt_sin}
  \wt_\bth(e):=
  \begin{cases}
   \sin(\th_q-\th_p), &\text{if $e$ is not adjacent to a boundary vertex,}\\
   1,&\text{otherwise.}
 \end{cases}
\end{equation}
\begin{remark}\label{rmk:side_lengths}
Setting $v_r:=\exp(2i\th_r)$ for all $r\in[n]$, we find $\sin(\th_q-\th_p)=\frac12 |v_q-v_p|$. Thus the edge weights record distances between cyclically ordered points on a circle.%
\end{remark}
\begin{remark}
Setting $\wt_\bth(e):=\sin(\th_q-\th_p)$ for \emph{all} edges of $G$ (including boundary edges) gives rise to \emph{dual critical varieties} discussed in \cref{sec:duality}.
\end{remark}

A crucial property of this choice of weights is that the resulting boundary measurement map is invariant under square moves: for any two reduced graphs $G,G'$ with the same strand permutation $\perm$, we have $\Meas(G,\wt_\bth)=\Meas(G',\wt'_\bth)$, where $\wt_\bth$ and $\wt'_\bth$ are defined by~\eqref{eq:wt_sin} on the edges of $G$ and $G'$, respectively. 
 For instance, the two graphs in \figref{fig:move}(a) produce the same boundary measurements (up to a common scalar). 
 Thus $\Meas(G,\wt_\bth)$ depends only on $\pf$ and $\bth$, therefore it  makes sense to denote $\Meas(\pf,\bth):=\Meas(G,\wt_\bth)$. In \cref{sec:intro:bound-meas-form}, we give an explicit simple formula for $\Meas(\pf,\bth)$ which depends only on $\pf$ and $\bth$ and does not involve choosing a reduced graph $G$.

\begin{remark}
As mentioned in the introduction, the formula~\eqref{eq:wt_sin} is obtained by combining the critical dimer model of~\cite{Kenyon} with the plabic tilings of~\cite{OPS}, in which case the construction of~\cite{Kenyon} simplifies considerably. %
It is not clear to us whether the complex edge weights in \cref{sec:intro:critical-varieties} below may also be obtained by specializing Kenyon's construction. We do not pursue this direction further since we will not rely on any known properties of isoradial embeddings in our approach.
\end{remark}

\subsection{Critical cells}\label{sec:intro-critical-varieties}
Unexpectedly, the combinatorics of critical cells and varieties turns out to be described by the associated \emph{reduced strand diagrams}. Let $\perm\in S_n$ be a permutation and $f$ be the corresponding loopless bounded affine permutation. %
\begin{definition}\label{dfn:intro:strand_diag}
Place the points $b_1,b_2,\dots, b_n$ on a circle in clockwise order, and for each $p\in[n]$, let $\m p$  (resp., $\p p$) be a point slightly before (resp., after) $b_p$ in clockwise order. The \emph{reduced strand diagram} of $\pf$ 
is obtained by drawing a straight arrow $\p s \to \m p$ whenever $\perm(s)=p$; see Figures~\hyperref[fig:move]{\ref*{fig:move}(b)} and~\ref{fig:align}. We say that $p\neq q\in[n]$ \emph{form \aperm-crossing} if the arrows $\p s\to \m {p}$ and $\p {t}\to\m{q}$ cross.
\end{definition}
\begin{definition}\label{dfn:theta}
A tuple $\bth=(\th_1,\th_2,\dots,\th_n)\in\R^n$ is called \emph{$\pf$-admissible} if for all $1\leq p<q\leq n$ such that $p$ and $q$ form an $f$-crossing, we have 
\begin{equation}\label{eq:adm_condition}
 \th_p<\th_q<\th_p+\pi.
\end{equation}
\end{definition}
\begin{remark}\label{rmk:intro:edge_wts_>0}
We show in \cref{prop:edges_conn_cpts} that if $\bth$ is $\pf$-admissible then all edge weights $\wt_\bth(e)$ are strictly positive. This is a surprising property since in general, $G$ contains edges of weight $\sin(\th_q-\th_p)$ where $p,q$ need not form \aperm-crossing. For example, the two graphs in \figref{fig:move}(a) contain edges of weight $\brx[24]$ and $\brx[13]$, respectively. 
\end{remark}

\begin{definition}
The \emph{critical cell} $\Ctp_\pf\subset\Ptp_\pf$ is defined by
\begin{equation*}%
  \Ctp_\pf:=\{\Meas(\pf,\bth)\mid \text{$\bth\in\R^n$ is $\pf$-admissible}\}.
\end{equation*}
\end{definition}

\begin{figure}
  \includegraphics{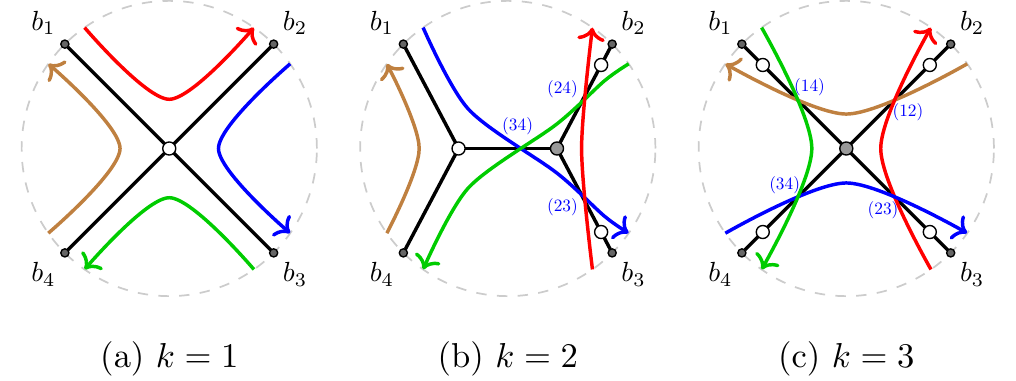}
  \caption{\label{fig:crit-ex} Examples of $\wt_\bth$ for $n=4$ and $f\in\Bkn$ for $k=1,2,3$. Unmarked edges have weight $1$.}
\end{figure}

\begin{example}
We give several examples in \cref{fig:crit-ex}. In the first example, we see that $\Ctp_{f_{1,4}}$ is a single point. In the second example, we see that $\Meas(f,\bth)$ does not depend on $\th_1$. In the third example, (the ratios of) the Pl\"ucker coordinates of $\Meas(f,\bth)\in\Grtnn(3,4)$ record (the ratios of) the side lengths of a convex inscribed quadrilateral with vertices $\v_1,\v_2,\v_3,\v_4$ (cf. \cref{rmk:side_lengths}).
\end{example}

We show in \cref{sec:conn-comp-strand} that the dimension of $\Ctp_\pf$ is at most (and conjecturally equal to) $n-\conn_f$, where $\conn_f$ is the number of connected components of the reduced strand diagram of $f$. For the top cell ($f=\fkn$), we establish the equality in \cref{thm:inj_fkn}. 

\begin{theorem}\label{thm:intro:inj_fkn}
  Let $1<k<n$ and $f=\fkn$. Then 
\begin{equation*}%
  \Ctp_{\fkn}\cong \Rtp^{n-1}.
\end{equation*}
For $k=1$ or $k=n$, $\Ctp_{\fkn}$ is a single point.
\end{theorem}

\begin{remark}
We caution that our notion of connectedness is different from the one studied in e.g.~\cite{OPS,ARW2}. For example, the permutation $f_{1,n}$ is usually considered connected in the literature, while in our case, the reduced strand diagram of $f_{1,n}$ has $n$ connected components (which is why $\Ctp_{f_{1,n}}$ has dimension $0$). We compare the two notions in \cref{prop:connected:refine}.
\end{remark}

\subsection{Critical varieties}\label{sec:intro:critical-varieties}
The \emph{critical variety} $\CCrit_\pf$ is the Zariski closure of $\Ctp_\pf$ inside the complex Grassmannian $\Gr(k,n)$. Our goal is to describe a certain subset $\Cio_\pf\subset\CCrit_\pf$ called an \emph{open critical variety}.

\begin{definition}\label{dfn:intro:f_adm_C}
A tuple $\bt=(t_1,t_2,\dots,t_n)\in(\Cast)^n$ is \emph{$f$-admissible} if $t_p\neq \pm t_q$ whenever $p,q$ form an $f$-crossing. 
\end{definition}

Given a reduced planar bipartite graph $G$ with strand permutation $\perm$ and a tuple $\bt\in(\Cast)^n$, we introduce a weight function $\wt_\bt:E(G)\to\C$ defined by
\begin{equation}\label{eq:wt_T}
  \wt_\bt(e):=
  \begin{cases}
   \H qp, &\text{if $e$ is not adjacent to a boundary vertex,}\\
   1,&\text{otherwise,}
 \end{cases}%
\end{equation}
where $e\in E(G)$ is labeled by $\{p,q\}$ with $1\leq p<q\leq n$ and $\br[x,y]:=\frac xy-\frac yx$ for $x,y\in\Cast$. Setting $t_p:=\exp(i\th_p)$ for all $p\in[n]$, $\wt_\bt$ specializes\footnote{Strictly speaking, we have $\wt_\bth(e)=\frac1{2i}\wt_\bt(e)$ for all non-boundary $e\in E(G)$, but this rescaling does not affect the boundary measurements since $\wt_\bt$ and $\wt_\bth$ are \emph{gauge-equivalent}; see the proof of \cref{lemma:GX_conn}.} to $\wt_\bth$ and the corresponding notions of $f$-admissibility coincide.
\begin{example}\label{ex:t1=t3_t2=t4}
Abbreviating $\H qp$ as $\brx[pq]$, the edge weights $\wt_\bt(e)$ in the case $f=f_{2,4}$ are given in \figref{fig:plabic}(c) and the corresponding boundary measurements are computed in \figref{fig:plabic}(d). Observe that any tuple $\bt=(t_1,t_2,t_3,t_4)\in(\Cast)^4$ satisfying $t_1=t_3$, $t_2=t_4$, and $t_1\neq\pm t_2$ is $f$-admissible. On the other hand, for any reduced graph $G$ with strand permutation $\perm$ (both of which are shown in \figref{fig:move}(a)), such a tuple $\bt$ gives $\wt_\bt(e)=0$ for some interior edge $e$ of $G$. Moreover, the dimer partition functions $\Delta_I(G,\wt_\bt)$ in~\eqref{eq:intro:Delta_I} will be zero for \emph{all} $I\in{[4]\choose 2}$. 
\end{example}

Nevertheless, it turns out that for \emph{any} $f$-admissible $\bt\in(\Cast)^n$, there is a well-defined element $\Meas(f,\bt)\in\Pio_f$ which coincides with the dimer partition functions $(\Delta_I(G,\wt_\bt))_{I\in{[n]\choose k}}$ up to multiplication by a common scalar whenever the latter are not all zero. This is a consequence of a remarkable property of critical varieties which we call \emph{the Laurent phenomenon} (\cref{thm:Laurent_ph}). Just as in the case of cluster algebras~\cite{FZ}, it involves a certain amount of non-trivial cancellation, which in the case of \figref{fig:plabic}(d) manifests itself in that all minors are divisible by $\brx[24]$. In fact, we conjecture that after writing $\Meas(f,\bt)$ in a certain canonical form, all \emph{cluster variables} in the cluster algebra structure~\cite{posit_cluster} on $\Pio_f$ are Laurent polynomials in the $\bt$-variables (\cref{conj:strong_Laurent}).

\begin{definition}\label{dfn:intro:Cio}
The \emph{open critical variety} $\Cio_f\subset\Pio_f$ is defined as
\begin{equation*}%
  \Cio_f:=\{\Meas(f,\bt)\mid \bt\in(\Cast)^n\text{ is $f$-admissible}\}.
\end{equation*}
\end{definition}
The question of whether $\Cio_f$ is actually an open subvariety of $\Crit_f$ remains unanswered (\cref{prob:Cio}). In \cref{sec:real-part-critical}, we give a conjectural description of its set $\Cio_f(\R)$ of real points (assuming $\bt$ is \emph{generic} as in \cref{dfn:intro:generic} below) and prove it in \cref{sec:proofs_top_cell} in the case $f=\fkn$ for $1\leq k\leq n-2$.

\subsection{Boundary measurement formula}\label{sec:intro:bound-meas-form}
Currently, in order to compute $\Meas(\pf,\bth)$, one needs to choose a reduced planar bipartite graph $G$, and then the result does not depend on this choice. It is therefore natural to look for an expression for $\Meas(\pf,\bth)$ purely in terms of $\pf$ and $\bth$. The answer turns out to be an explicitly defined point of $\Grtnn(k,n)$ as we now explain. The results of this section give a natural extension of our previous results~\cite{ising_crit} obtained in the case of the Ising model.

Take the reduced strand diagram of $\pf$ as in \figref{fig:move}(b). For each $r\in[n]$, let %
\begin{equation}\label{eq:intro:J_r}
  \J_r:=\{p\in[n]\mid \text{$b_r$ is to the left of the arrow $\p s\to \m p$}\}.
\end{equation}
Here  we set $s:=\perm^{-1}(p)$. The integer $k$ from  \cref{sec:intro:plan-bipart-graphs} satisfies $|\J_r|=k-1$. We have $r\notin \J_r$, and the collection $(\J_r\sqcup\{r\})_{r\in[n]}$ is known as the \emph{Grassmann necklace}~\cite{Pos} of $\pf$. 

For an index $r\in [n]$, let $\eps_r\in\{\pm1\}$ be given by %
\[\eps_r:=(-1)^{\#\{p\in[n]\mid \perm(p)\leq p<r\}}.\]
\begin{definition}\label{dfn:curve}
Let $\bth=(\th_1,\th_2,\dots,\th_n)$ be a tuple of angles. Define a curve $\curve_{\pf,\bth}:\R\to\R^{n}$ whose coordinates $\curve_{\pf,\bth}(t)=(\g_1(t),\g_2(t),\dots,\g_{n}(t))$ are given by
\begin{equation}\label{eq:curve}
  \g_r(t):= \eps_r\prod_{p\in \J_r} \sin(t-\th_p) \qquad\text{for $r\in[n]$.}
\end{equation}
\end{definition}

In \cref{sec:degen-bound-meas}, we give a boundary measurement formula for an arbitrary $\pf$-admissible tuple $\bth$. For simplicity, here we restrict to the case when $\bth$ satisfies a certain genericity assumption.
\begin{definition}\label{dfn:intro:generic}
A tuple $\bth\in\R^n$ is called \emph{generic} if all angles in $\bth$ are pairwise non-congruent modulo $\pi$. Similarly, $\bt\in(\Cast)^n$ is \emph{generic} if $t_p\neq \pm t_q$ for all $p\neq q\in[n]$.
\end{definition}
\begin{theorem}\label{thm:intro:bound_meas}
Suppose that $\bth$ is a generic $\pf$-admissible tuple. Then the linear span $\Span(\curve_{\pf,\bth})\subset\R^n$ has dimension $k$ and we have 
\begin{equation}\label{eq:*}
  \Meas(\pf,\bth)=\Span(\curve_{\pf,\bth}) \quad\text{inside $\Grtnn(k,n)$.}
\end{equation}
\end{theorem}
\noindent Replacing $\sin(t-\th_p)$ with $\HT p$ in~\eqref{eq:curve}, one obtains a formula for $\Meas(f,\bt)$ for all generic $\bt\in(\Cast)^n$; cf. \cref{thm:bound_meas_C}.

\begin{example}
Let $k=2$, $n=4$, and $\perm=\perm_{k,n}$ be the permutation sending $p\mapsto p+2$ modulo~$4$. The boundary measurement map $\Meas(\pf,\bth)$ was computed in \figref{fig:plabic}(d). Since the Pl\"ucker coordinates are defined up to a common scalar, the term $\brx[24]$ cancels out. The sets $\J_r$ are given by $\J_1=\{2\}$, $\J_2=\{3\}$, $\J_3=\{4\}$, and $\J_4=\{1\}$, so $\curve_{\pf,\bth}$ has coordinates
\[\curve_{\pf,\bth}(t)=(\sin(t-\th_2),\sin(t-\th_3),\sin(t-\th_4),-\sin(t-\th_1)).\]
We can choose a basis of $\Span(\curve_{\pf,\bth})$ consisting of e.g. $\curve_{\pf,\bth}(0)$ and $\curve_{\pf,\bth}(\pi/2)$, which we can write in the rows of the following matrix:
\[A=\begin{pmatrix}
-\sin(\th_2) & -\sin(\th_3) & -\sin(\th_4) & \sin(\th_1)\\
\cos(\th_2) &\cos(\th_3) &\cos(\th_4) & -\cos(\th_1) 
\end{pmatrix}. \]
We see that the maximal minors of $A$ coincide with the values computed in \figref{fig:plabic}(d).
\end{example}

\begin{remark}
For a generic $\bth$, an explicit basis of $\Span(\curve_{\pf,\bth})$ can be chosen by taking any $k$ distinct points on the curve $\curve_{\pf,\bth}$ (\cref{prop:basis}). A more canonical way to produce a basis of $\Span(\curve_{\pf,\bth})$ is to observe that each coordinate $\g_r(t)$ is a trigonometric polynomial of degree $k-1$. Therefore it has precisely $k$ non-trivial Fourier coefficients. The rows of the resulting $k\times n$ matrix of Fourier coefficients form a basis of $\Span(\curve_{\pf,\bth})$ which does not depend on anything besides $\pf$ and the angles $\th_1,\th_2,\dots,\th_{n}$; see \cref{prop:Fourier}.
\end{remark}

\subsection{Applications}\label{sec:intro:applications}

Our boundary measurement formula can be specialized to the cases of the Ising model and electrical resistor networks. For the Ising model, this was done in~\cite{ising_crit}; here we focus on the case of electrical networks.
Critical electrical networks are defined on isoradial graphs. To produce an isoradial graph, take a rhombus tiling $\Tiling$ of a polygonal region $\Reg$, such as the one in~\figref{fig:intro1}(a,b), color its vertices black and white in a bipartite way, and let $G_\Tiling$ be the graph consisting of all diagonals of rhombi that connect their black vertices; see \figref{fig:intro1}(c). The graph $G_\Tiling$ is isoradial, with white vertices of $\Tiling$ being the centers of the corresponding unit circles. 

We consider $G_\Tiling$ as an electrical resistor network, replacing every edge by a resistor. Given a rhombus $ABCD$ of $\Tiling$ with black vertices $A$ and $C$, the edge $AC$ of $G_\Tiling$ is treated as a resistor whose resistance equals the ratio $\frac{|AC|}{|BD|}$ of the lengths of the rhombus diagonals.

Denote by $\b_1,\b_2,\dots,\b_\N$ the vertices of $G_\Tiling$ that belong to the boundary of $\Reg$, listed in clockwise order. Fix $p\in[\N]$. Let us apply the voltage of $1$ to $\b_p$ and the voltage of $0$ to all other boundary vertices. Then the voltages at all interior vertices, as well as the currents through all edges, can be computed from Ohm's and Kirchhoff's laws. For each $q\in[\N]$, denote by $\la^{\Tiling}_{p,q}$ the current that flows into $\b_q$.\footnote{By linearity of Ohm's and Kirchhoff's laws, knowing $\la^{\Tiling}_{p,q}$ for all $1\leq p,q\leq \N$ allows one to solve the more general problem: for any known voltages that are applied to the boundary vertices, one finds the resulting currents flowing through each boundary vertex.} (Thus $\la^{\Tiling}_{p,q}\geq0$ for $p\neq q$.) The matrix $\La^{\Tiling}=(\la^{\Tiling}_{p,q})$ is known as the \emph{response matrix} of the electrical network associated with $G_\Tiling$.

\begin{figure}\makebox[1.0\textwidth]{

\setlength{\tabcolsep}{4pt}
\begin{tabular}{cccc}\\

\scalebox{0.65}{
\begin{tikzpicture}[scale=1.30 ,baseline=(ZUZU.base)]\coordinate(ZUZU) at (0,0); 
\coordinate (X0) at (-1.090,4.131);
\coordinate (X1) at (-1.797,3.424);
\coordinate (X2) at (-1.414,2.500);
\coordinate (X3) at (-0.707,3.207);
\coordinate (X4) at (-2.090,4.131);
\coordinate (X5) at (-2.797,3.424);
\coordinate (X6) at (-1.797,2.424);
\coordinate (X7) at (-1.414,1.500);
\coordinate (X8) at (-3.075,4.305);
\coordinate (X9) at (-3.782,3.598);
\coordinate (X10) at (-0.707,2.207);
\coordinate (X11) at (-2.606,1.836);
\coordinate (X12) at (-2.223,0.912);
\coordinate (X13) at (-2.606,-0.012);
\coordinate (X14) at (-1.631,-0.234);
\coordinate (X15) at (-1.248,0.690);
\coordinate (X16) at (-2.797,2.424);
\coordinate (X17) at (-3.782,2.598);
\coordinate (X18) at (-3.606,1.836);
\coordinate (X19) at (0.000,1.500);
\coordinate (X20) at (0.000,2.500);
\coordinate (X21) at (-0.439,1.277);
\coordinate (X22) at (-0.822,0.354);
\coordinate (X23) at (-2.989,0.912);
\coordinate (X24) at (0.000,2.000);
\coordinate (X25) at (-0.354,2.854);
\coordinate (X26) at (-0.898,3.669);
\coordinate (X27) at (-1.590,4.131);
\coordinate (X28) at (-2.582,4.218);
\coordinate (X29) at (-3.428,3.951);
\coordinate (X30) at (-3.782,3.098);
\coordinate (X31) at (-3.289,2.511);
\coordinate (X32) at (-3.201,2.130);
\coordinate (X33) at (-3.106,1.836);
\coordinate (X34) at (-2.797,1.374);
\coordinate (X35) at (-2.797,0.450);
\coordinate (X36) at (-2.118,-0.123);
\coordinate (X37) at (-1.226,0.060);
\coordinate (X38) at (-0.631,0.816);
\coordinate (X39) at (-0.927,1.389);
\coordinate (X40) at (-1.061,1.854);
\coordinate (X41) at (-0.354,1.854);
\fill [opacity=0.10,blue] (X0)-- (X1)-- (X2)-- (X3) -- cycle;
\draw [line width=1pt,black,opacity=1.00] (X0)-- (X1);
\draw [line width=1pt,black,opacity=1.00] (X1)-- (X2);
\draw [line width=1pt,black,opacity=1.00] (X2)-- (X3);
\draw [line width=1pt,black,opacity=1.00] (X3)-- (X0);
\fill [opacity=0.10,blue] (X4)-- (X5)-- (X1)-- (X0) -- cycle;
\draw [line width=1pt,black,opacity=1.00] (X4)-- (X5);
\draw [line width=1pt,black,opacity=1.00] (X5)-- (X1);
\draw [line width=1pt,black,opacity=1.00] (X1)-- (X0);
\draw [line width=1pt,black,opacity=1.00] (X0)-- (X4);
\fill [opacity=0.10,blue] (X1)-- (X6)-- (X7)-- (X2) -- cycle;
\draw [line width=1pt,black,opacity=1.00] (X1)-- (X6);
\draw [line width=1pt,black,opacity=1.00] (X6)-- (X7);
\draw [line width=1pt,black,opacity=1.00] (X7)-- (X2);
\draw [line width=1pt,black,opacity=1.00] (X2)-- (X1);
\fill [opacity=0.10,blue] (X8)-- (X9)-- (X5)-- (X4) -- cycle;
\draw [line width=1pt,black,opacity=1.00] (X8)-- (X9);
\draw [line width=1pt,black,opacity=1.00] (X9)-- (X5);
\draw [line width=1pt,black,opacity=1.00] (X5)-- (X4);
\draw [line width=1pt,black,opacity=1.00] (X4)-- (X8);
\fill [opacity=0.10,blue] (X7)-- (X10)-- (X3)-- (X2) -- cycle;
\draw [line width=1pt,black,opacity=1.00] (X7)-- (X10);
\draw [line width=1pt,black,opacity=1.00] (X10)-- (X3);
\draw [line width=1pt,black,opacity=1.00] (X3)-- (X2);
\draw [line width=1pt,black,opacity=1.00] (X2)-- (X7);
\fill [opacity=0.10,blue] (X6)-- (X11)-- (X12)-- (X7) -- cycle;
\draw [line width=1pt,black,opacity=1.00] (X6)-- (X11);
\draw [line width=1pt,black,opacity=1.00] (X11)-- (X12);
\draw [line width=1pt,black,opacity=1.00] (X12)-- (X7);
\draw [line width=1pt,black,opacity=1.00] (X7)-- (X6);
\fill [opacity=0.10,blue] (X13)-- (X14)-- (X15)-- (X12) -- cycle;
\draw [line width=1pt,black,opacity=1.00] (X13)-- (X14);
\draw [line width=1pt,black,opacity=1.00] (X14)-- (X15);
\draw [line width=1pt,black,opacity=1.00] (X15)-- (X12);
\draw [line width=1pt,black,opacity=1.00] (X12)-- (X13);
\fill [opacity=0.10,blue] (X5)-- (X16)-- (X6)-- (X1) -- cycle;
\draw [line width=1pt,black,opacity=1.00] (X5)-- (X16);
\draw [line width=1pt,black,opacity=1.00] (X16)-- (X6);
\draw [line width=1pt,black,opacity=1.00] (X6)-- (X1);
\draw [line width=1pt,black,opacity=1.00] (X1)-- (X5);
\fill [opacity=0.10,blue] (X9)-- (X17)-- (X16)-- (X5) -- cycle;
\draw [line width=1pt,black,opacity=1.00] (X9)-- (X17);
\draw [line width=1pt,black,opacity=1.00] (X17)-- (X16);
\draw [line width=1pt,black,opacity=1.00] (X16)-- (X5);
\draw [line width=1pt,black,opacity=1.00] (X5)-- (X9);
\fill [opacity=0.10,blue] (X16)-- (X18)-- (X11)-- (X6) -- cycle;
\draw [line width=1pt,black,opacity=1.00] (X16)-- (X18);
\draw [line width=1pt,black,opacity=1.00] (X18)-- (X11);
\draw [line width=1pt,black,opacity=1.00] (X11)-- (X6);
\draw [line width=1pt,black,opacity=1.00] (X6)-- (X16);
\fill [opacity=0.10,blue] (X10)-- (X19)-- (X20)-- (X3) -- cycle;
\draw [line width=1pt,black,opacity=1.00] (X10)-- (X19);
\draw [line width=1pt,black,opacity=1.00] (X19)-- (X20);
\draw [line width=1pt,black,opacity=1.00] (X20)-- (X3);
\draw [line width=1pt,black,opacity=1.00] (X3)-- (X10);
\fill [opacity=0.10,blue] (X15)-- (X21)-- (X7)-- (X12) -- cycle;
\draw [line width=1pt,black,opacity=1.00] (X15)-- (X21);
\draw [line width=1pt,black,opacity=1.00] (X21)-- (X7);
\draw [line width=1pt,black,opacity=1.00] (X7)-- (X12);
\draw [line width=1pt,black,opacity=1.00] (X12)-- (X15);
\fill [opacity=0.10,blue] (X14)-- (X22)-- (X21)-- (X15) -- cycle;
\draw [line width=1pt,black,opacity=1.00] (X14)-- (X22);
\draw [line width=1pt,black,opacity=1.00] (X22)-- (X21);
\draw [line width=1pt,black,opacity=1.00] (X21)-- (X15);
\draw [line width=1pt,black,opacity=1.00] (X15)-- (X14);
\fill [opacity=0.10,blue] (X11)-- (X23)-- (X13)-- (X12) -- cycle;
\draw [line width=1pt,black,opacity=1.00] (X11)-- (X23);
\draw [line width=1pt,black,opacity=1.00] (X23)-- (X13);
\draw [line width=1pt,black,opacity=1.00] (X13)-- (X12);
\draw [line width=1pt,black,opacity=1.00] (X12)-- (X11);
\end{tikzpicture}
}
 & \scalebox{0.65}{
\begin{tikzpicture}[scale=1.30 ,baseline=(ZUZU.base)]\coordinate(ZUZU) at (0,0); 
\coordinate (X0) at (-1.000,2.000);
\coordinate (X1) at (-0.134,1.500);
\coordinate (X2) at (0.366,2.366);
\coordinate (X3) at (-0.500,2.866);
\coordinate (X4) at (-1.000,1.000);
\coordinate (X5) at (-0.134,0.500);
\coordinate (X6) at (0.866,1.500);
\coordinate (X7) at (1.366,2.366);
\coordinate (X8) at (-0.134,3.232);
\coordinate (X9) at (-1.000,3.732);
\coordinate (X10) at (0.866,0.500);
\coordinate (X11) at (-1.366,2.366);
\coordinate (X12) at (-1.866,1.500);
\coordinate (X13) at (0.866,3.232);
\coordinate (X14) at (-1.866,0.500);
\coordinate (X15) at (1.366,1.366);
\coordinate (X16) at (-1.866,3.232);
\coordinate (X17) at (-2.366,2.366);
\coordinate (X18) at (-1.000,0.000);
\coordinate (X19) at (-2.366,1.366);
\coordinate (X20) at (0.000,3.732);
\coordinate (X21) at (-0.000,0.000);
\coordinate (X22) at (0.000,0.000);
\coordinate (X23) at (0.433,0.250);
\coordinate (X24) at (1.116,0.933);
\coordinate (X25) at (1.366,1.866);
\coordinate (X26) at (1.116,2.799);
\coordinate (X27) at (0.433,3.482);
\coordinate (X28) at (-0.500,3.732);
\coordinate (X29) at (-1.433,3.482);
\coordinate (X30) at (-2.116,2.799);
\coordinate (X31) at (-2.366,1.866);
\coordinate (X32) at (-2.116,0.933);
\coordinate (X33) at (-1.433,0.250);
\coordinate (X34) at (-0.500,0.000);
\fill [opacity=0.10,blue] (X0)-- (X1)-- (X2)-- (X3) -- cycle;
\draw [line width=1pt,black,opacity=1.00] (X0)-- (X1);
\draw [line width=1pt,black,opacity=1.00] (X1)-- (X2);
\draw [line width=1pt,black,opacity=1.00] (X2)-- (X3);
\draw [line width=1pt,black,opacity=1.00] (X3)-- (X0);
\fill [opacity=0.10,blue] (X4)-- (X5)-- (X1)-- (X0) -- cycle;
\draw [line width=1pt,black,opacity=1.00] (X4)-- (X5);
\draw [line width=1pt,black,opacity=1.00] (X5)-- (X1);
\draw [line width=1pt,black,opacity=1.00] (X1)-- (X0);
\draw [line width=1pt,black,opacity=1.00] (X0)-- (X4);
\fill [opacity=0.10,blue] (X1)-- (X6)-- (X7)-- (X2) -- cycle;
\draw [line width=1pt,black,opacity=1.00] (X1)-- (X6);
\draw [line width=1pt,black,opacity=1.00] (X6)-- (X7);
\draw [line width=1pt,black,opacity=1.00] (X7)-- (X2);
\draw [line width=1pt,black,opacity=1.00] (X2)-- (X1);
\fill [opacity=0.10,blue] (X8)-- (X9)-- (X3)-- (X2) -- cycle;
\draw [line width=1pt,black,opacity=1.00] (X8)-- (X9);
\draw [line width=1pt,black,opacity=1.00] (X9)-- (X3);
\draw [line width=1pt,black,opacity=1.00] (X3)-- (X2);
\draw [line width=1pt,black,opacity=1.00] (X2)-- (X8);
\fill [opacity=0.10,blue] (X5)-- (X10)-- (X6)-- (X1) -- cycle;
\draw [line width=1pt,black,opacity=1.00] (X5)-- (X10);
\draw [line width=1pt,black,opacity=1.00] (X10)-- (X6);
\draw [line width=1pt,black,opacity=1.00] (X6)-- (X1);
\draw [line width=1pt,black,opacity=1.00] (X1)-- (X5);
\fill [opacity=0.10,blue] (X11)-- (X12)-- (X0)-- (X3) -- cycle;
\draw [line width=1pt,black,opacity=1.00] (X11)-- (X12);
\draw [line width=1pt,black,opacity=1.00] (X12)-- (X0);
\draw [line width=1pt,black,opacity=1.00] (X0)-- (X3);
\draw [line width=1pt,black,opacity=1.00] (X3)-- (X11);
\fill [opacity=0.10,blue] (X7)-- (X13)-- (X8)-- (X2) -- cycle;
\draw [line width=1pt,black,opacity=1.00] (X7)-- (X13);
\draw [line width=1pt,black,opacity=1.00] (X13)-- (X8);
\draw [line width=1pt,black,opacity=1.00] (X8)-- (X2);
\draw [line width=1pt,black,opacity=1.00] (X2)-- (X7);
\fill [opacity=0.10,blue] (X12)-- (X14)-- (X4)-- (X0) -- cycle;
\draw [line width=1pt,black,opacity=1.00] (X12)-- (X14);
\draw [line width=1pt,black,opacity=1.00] (X14)-- (X4);
\draw [line width=1pt,black,opacity=1.00] (X4)-- (X0);
\draw [line width=1pt,black,opacity=1.00] (X0)-- (X12);
\fill [opacity=0.10,blue] (X10)-- (X15)-- (X7)-- (X6) -- cycle;
\draw [line width=1pt,black,opacity=1.00] (X10)-- (X15);
\draw [line width=1pt,black,opacity=1.00] (X15)-- (X7);
\draw [line width=1pt,black,opacity=1.00] (X7)-- (X6);
\draw [line width=1pt,black,opacity=1.00] (X6)-- (X10);
\fill [opacity=0.10,blue] (X9)-- (X16)-- (X11)-- (X3) -- cycle;
\draw [line width=1pt,black,opacity=1.00] (X9)-- (X16);
\draw [line width=1pt,black,opacity=1.00] (X16)-- (X11);
\draw [line width=1pt,black,opacity=1.00] (X11)-- (X3);
\draw [line width=1pt,black,opacity=1.00] (X3)-- (X9);
\fill [opacity=0.10,blue] (X16)-- (X17)-- (X12)-- (X11) -- cycle;
\draw [line width=1pt,black,opacity=1.00] (X16)-- (X17);
\draw [line width=1pt,black,opacity=1.00] (X17)-- (X12);
\draw [line width=1pt,black,opacity=1.00] (X12)-- (X11);
\draw [line width=1pt,black,opacity=1.00] (X11)-- (X16);
\fill [opacity=0.10,blue] (X14)-- (X18)-- (X5)-- (X4) -- cycle;
\draw [line width=1pt,black,opacity=1.00] (X14)-- (X18);
\draw [line width=1pt,black,opacity=1.00] (X18)-- (X5);
\draw [line width=1pt,black,opacity=1.00] (X5)-- (X4);
\draw [line width=1pt,black,opacity=1.00] (X4)-- (X14);
\fill [opacity=0.10,blue] (X17)-- (X19)-- (X14)-- (X12) -- cycle;
\draw [line width=1pt,black,opacity=1.00] (X17)-- (X19);
\draw [line width=1pt,black,opacity=1.00] (X19)-- (X14);
\draw [line width=1pt,black,opacity=1.00] (X14)-- (X12);
\draw [line width=1pt,black,opacity=1.00] (X12)-- (X17);
\fill [opacity=0.10,blue] (X13)-- (X20)-- (X9)-- (X8) -- cycle;
\draw [line width=1pt,black,opacity=1.00] (X13)-- (X20);
\draw [line width=1pt,black,opacity=1.00] (X20)-- (X9);
\draw [line width=1pt,black,opacity=1.00] (X9)-- (X8);
\draw [line width=1pt,black,opacity=1.00] (X8)-- (X13);
\fill [opacity=0.10,blue] (X18)-- (X21)-- (X10)-- (X5) -- cycle;
\draw [line width=1pt,black,opacity=1.00] (X18)-- (X21);
\draw [line width=1pt,black,opacity=1.00] (X21)-- (X10);
\draw [line width=1pt,black,opacity=1.00] (X10)-- (X5);
\draw [line width=1pt,black,opacity=1.00] (X5)-- (X18);
\end{tikzpicture}
}
 & \scalebox{0.65}{
\begin{tikzpicture}[scale=1.30 ,baseline=(ZUZU.base)]\coordinate(ZUZU) at (0,0); 
\coordinate (X0) at (-1.000,2.000);
\coordinate (X1) at (-0.134,1.500);
\coordinate (X2) at (0.366,2.366);
\coordinate (X3) at (-0.500,2.866);
\coordinate (X4) at (-1.000,1.000);
\coordinate (X5) at (-0.134,0.500);
\coordinate (X6) at (0.866,1.500);
\coordinate (X7) at (1.366,2.366);
\coordinate (X8) at (-0.134,3.232);
\coordinate (X9) at (-1.000,3.732);
\coordinate (X10) at (0.866,0.500);
\coordinate (X11) at (-1.366,2.366);
\coordinate (X12) at (-1.866,1.500);
\coordinate (X13) at (0.866,3.232);
\coordinate (X14) at (-1.866,0.500);
\coordinate (X15) at (1.366,1.366);
\coordinate (X16) at (-1.866,3.232);
\coordinate (X17) at (-2.366,2.366);
\coordinate (X18) at (-1.000,0.000);
\coordinate (X19) at (-2.366,1.366);
\coordinate (X20) at (0.000,3.732);
\coordinate (X21) at (-0.000,0.000);
\coordinate (X22) at (0.000,0.000);
\coordinate (X23) at (0.433,0.250);
\coordinate (X24) at (1.116,0.933);
\coordinate (X25) at (1.366,1.866);
\coordinate (X26) at (1.116,2.799);
\coordinate (X27) at (0.433,3.482);
\coordinate (X28) at (-0.500,3.732);
\coordinate (X29) at (-1.433,3.482);
\coordinate (X30) at (-2.116,2.799);
\coordinate (X31) at (-2.366,1.866);
\coordinate (X32) at (-2.116,0.933);
\coordinate (X33) at (-1.433,0.250);
\coordinate (X34) at (-0.500,0.000);
\draw [dashed, line width=0.5pt,black] (X0)-- (X1);
\draw [dashed, line width=0.5pt,black] (X1)-- (X2);
\draw [dashed, line width=0.5pt,black] (X2)-- (X3);
\draw [dashed, line width=0.5pt,black] (X3)-- (X0);
\fill [opacity=0.10,blue] (X0)-- (X1)-- (X2)-- (X3) -- cycle;
\node[scale=0.50,fill=black,draw,circle] (A0) at (X1) {};
\node[scale=0.50,fill=black,draw,circle] (C0) at (X3) {};
\node[scale=0.50,fill=white,draw,circle] (B0) at (X2) {};
\node[scale=0.50,fill=white,draw,circle] (D0) at (X0) {};
\draw[line width=1.5pt,black] (A0)--(C0);
\draw [dashed, line width=0.5pt,black] (X4)-- (X5);
\draw [dashed, line width=0.5pt,black] (X5)-- (X1);
\draw [dashed, line width=0.5pt,black] (X0)-- (X4);
\fill [opacity=0.10,blue] (X4)-- (X5)-- (X1)-- (X0) -- cycle;
\node[scale=0.50,fill=black,draw,circle] (A1) at (X4) {};
\node[scale=0.50,fill=black,draw,circle] (C1) at (X1) {};
\node[scale=0.50,fill=white,draw,circle] (B1) at (X5) {};
\node[scale=0.50,fill=white,draw,circle] (D1) at (X0) {};
\draw[line width=1.5pt,black] (A1)--(C1);
\draw [dashed, line width=0.5pt,black] (X1)-- (X6);
\draw [dashed, line width=0.5pt,black] (X6)-- (X7);
\draw [dashed, line width=0.5pt,black] (X7)-- (X2);
\fill [opacity=0.10,blue] (X1)-- (X6)-- (X7)-- (X2) -- cycle;
\node[scale=0.50,fill=black,draw,circle] (A2) at (X1) {};
\node[scale=0.50,fill=black,draw,circle] (C2) at (X7) {};
\node[scale=0.50,fill=white,draw,circle] (B2) at (X6) {};
\node[scale=0.50,fill=white,draw,circle] (D2) at (X2) {};
\draw[line width=1.5pt,black] (A2)--(C2);
\draw [dashed, line width=0.5pt,black] (X8)-- (X9);
\draw [dashed, line width=0.5pt,black] (X9)-- (X3);
\draw [dashed, line width=0.5pt,black] (X2)-- (X8);
\fill [opacity=0.10,blue] (X8)-- (X9)-- (X3)-- (X2) -- cycle;
\node[scale=0.50,fill=black,draw,circle] (A3) at (X8) {};
\node[scale=0.50,fill=black,draw,circle] (C3) at (X3) {};
\node[scale=0.50,fill=white,draw,circle] (B3) at (X9) {};
\node[scale=0.50,fill=white,draw,circle] (D3) at (X2) {};
\draw[line width=1.5pt,black] (A3)--(C3);
\draw [dashed, line width=0.5pt,black] (X5)-- (X10);
\draw [dashed, line width=0.5pt,black] (X10)-- (X6);
\fill [opacity=0.10,blue] (X5)-- (X10)-- (X6)-- (X1) -- cycle;
\node[scale=0.50,fill=black,draw,circle] (A4) at (X10) {};
\node[scale=0.50,fill=black,draw,circle] (C4) at (X1) {};
\node[scale=0.50,fill=white,draw,circle] (B4) at (X6) {};
\node[scale=0.50,fill=white,draw,circle] (D4) at (X5) {};
\draw[line width=1.5pt,black] (A4)--(C4);
\draw [dashed, line width=0.5pt,black] (X11)-- (X12);
\draw [dashed, line width=0.5pt,black] (X12)-- (X0);
\draw [dashed, line width=0.5pt,black] (X3)-- (X11);
\fill [opacity=0.10,blue] (X11)-- (X12)-- (X0)-- (X3) -- cycle;
\node[scale=0.50,fill=black,draw,circle] (A5) at (X12) {};
\node[scale=0.50,fill=black,draw,circle] (C5) at (X3) {};
\node[scale=0.50,fill=white,draw,circle] (B5) at (X0) {};
\node[scale=0.50,fill=white,draw,circle] (D5) at (X11) {};
\draw[line width=1.5pt,black] (A5)--(C5);
\draw [dashed, line width=0.5pt,black] (X7)-- (X13);
\draw [dashed, line width=0.5pt,black] (X13)-- (X8);
\fill [opacity=0.10,blue] (X7)-- (X13)-- (X8)-- (X2) -- cycle;
\node[scale=0.50,fill=black,draw,circle] (A6) at (X7) {};
\node[scale=0.50,fill=black,draw,circle] (C6) at (X8) {};
\node[scale=0.50,fill=white,draw,circle] (B6) at (X13) {};
\node[scale=0.50,fill=white,draw,circle] (D6) at (X2) {};
\draw[line width=1.5pt,black] (A6)--(C6);
\draw [dashed, line width=0.5pt,black] (X12)-- (X14);
\draw [dashed, line width=0.5pt,black] (X14)-- (X4);
\fill [opacity=0.10,blue] (X12)-- (X14)-- (X4)-- (X0) -- cycle;
\node[scale=0.50,fill=black,draw,circle] (A7) at (X12) {};
\node[scale=0.50,fill=black,draw,circle] (C7) at (X4) {};
\node[scale=0.50,fill=white,draw,circle] (B7) at (X14) {};
\node[scale=0.50,fill=white,draw,circle] (D7) at (X0) {};
\draw[line width=1.5pt,black] (A7)--(C7);
\draw [dashed, line width=0.5pt,black] (X10)-- (X15);
\draw [dashed, line width=0.5pt,black] (X15)-- (X7);
\fill [opacity=0.10,blue] (X10)-- (X15)-- (X7)-- (X6) -- cycle;
\node[scale=0.50,fill=black,draw,circle] (A8) at (X10) {};
\node[scale=0.50,fill=black,draw,circle] (C8) at (X7) {};
\node[scale=0.50,fill=white,draw,circle] (B8) at (X15) {};
\node[scale=0.50,fill=white,draw,circle] (D8) at (X6) {};
\draw[line width=1.5pt,black] (A8)--(C8);
\draw [dashed, line width=0.5pt,black] (X9)-- (X16);
\draw [dashed, line width=0.5pt,black] (X16)-- (X11);
\fill [opacity=0.10,blue] (X9)-- (X16)-- (X11)-- (X3) -- cycle;
\node[scale=0.50,fill=black,draw,circle] (A9) at (X16) {};
\node[scale=0.50,fill=black,draw,circle] (C9) at (X3) {};
\node[scale=0.50,fill=white,draw,circle] (B9) at (X11) {};
\node[scale=0.50,fill=white,draw,circle] (D9) at (X9) {};
\draw[line width=1.5pt,black] (A9)--(C9);
\draw [dashed, line width=0.5pt,black] (X16)-- (X17);
\draw [dashed, line width=0.5pt,black] (X17)-- (X12);
\fill [opacity=0.10,blue] (X16)-- (X17)-- (X12)-- (X11) -- cycle;
\node[scale=0.50,fill=black,draw,circle] (A10) at (X16) {};
\node[scale=0.50,fill=black,draw,circle] (C10) at (X12) {};
\node[scale=0.50,fill=white,draw,circle] (B10) at (X17) {};
\node[scale=0.50,fill=white,draw,circle] (D10) at (X11) {};
\draw[line width=1.5pt,black] (A10)--(C10);
\draw [dashed, line width=0.5pt,black] (X14)-- (X18);
\draw [dashed, line width=0.5pt,black] (X18)-- (X5);
\fill [opacity=0.10,blue] (X14)-- (X18)-- (X5)-- (X4) -- cycle;
\node[scale=0.50,fill=black,draw,circle] (A11) at (X18) {};
\node[scale=0.50,fill=black,draw,circle] (C11) at (X4) {};
\node[scale=0.50,fill=white,draw,circle] (B11) at (X5) {};
\node[scale=0.50,fill=white,draw,circle] (D11) at (X14) {};
\draw[line width=1.5pt,black] (A11)--(C11);
\draw [dashed, line width=0.5pt,black] (X17)-- (X19);
\draw [dashed, line width=0.5pt,black] (X19)-- (X14);
\fill [opacity=0.10,blue] (X17)-- (X19)-- (X14)-- (X12) -- cycle;
\node[scale=0.50,fill=black,draw,circle] (A12) at (X19) {};
\node[scale=0.50,fill=black,draw,circle] (C12) at (X12) {};
\node[scale=0.50,fill=white,draw,circle] (B12) at (X14) {};
\node[scale=0.50,fill=white,draw,circle] (D12) at (X17) {};
\draw[line width=1.5pt,black] (A12)--(C12);
\draw [dashed, line width=0.5pt,black] (X13)-- (X20);
\draw [dashed, line width=0.5pt,black] (X20)-- (X9);
\fill [opacity=0.10,blue] (X13)-- (X20)-- (X9)-- (X8) -- cycle;
\node[scale=0.50,fill=black,draw,circle] (A13) at (X20) {};
\node[scale=0.50,fill=black,draw,circle] (C13) at (X8) {};
\node[scale=0.50,fill=white,draw,circle] (B13) at (X9) {};
\node[scale=0.50,fill=white,draw,circle] (D13) at (X13) {};
\draw[line width=1.5pt,black] (A13)--(C13);
\draw [dashed, line width=0.5pt,black] (X18)-- (X21);
\draw [dashed, line width=0.5pt,black] (X21)-- (X10);
\fill [opacity=0.10,blue] (X18)-- (X21)-- (X10)-- (X5) -- cycle;
\node[scale=0.50,fill=black,draw,circle] (A14) at (X18) {};
\node[scale=0.50,fill=black,draw,circle] (C14) at (X10) {};
\node[scale=0.50,fill=white,draw,circle] (B14) at (X21) {};
\node[scale=0.50,fill=white,draw,circle] (D14) at (X5) {};
\draw[line width=1.5pt,black] (A14)--(C14);
\node[inner sep=3pt,scale=1.20000000000000,anchor=134] (B1) at (X10) {$\b_{6}$};
\node[inner sep=3pt,scale=1.20000000000000,anchor=194] (B2) at (X7) {$\b_{5}$};
\node[inner sep=3pt,scale=1.20000000000000,anchor=254] (B3) at (X20) {$\b_{4}$};
\node[inner sep=3pt,scale=1.20000000000000,anchor=314] (B4) at (X16) {$\b_{3}$};
\node[inner sep=3pt,scale=1.20000000000000,anchor=374] (B5) at (X19) {$\b_{2}$};
\node[inner sep=3pt,scale=1.20000000000000,anchor=435] (B6) at (X18) {$\b_{1}$};
\end{tikzpicture}
}
 & \setlength{\tabcolsep}{0pt}\begin{tikzpicture}[baseline=(ZUZU.base)]\coordinate(ZUZU) at (0,0);\node(A) at (0,1.5){\begin{tabular}{ccc}

&&\\
\scalebox{0.80}{
\scalebox{0.90}{
\begin{tikzpicture}[scale=1.20 ,baseline=(ZUZU.base)]\coordinate(ZUZU) at (0,0); 
\coordinate (X0) at (0.000,-0.700);
\coordinate (X1) at (0.500,0.166);
\coordinate (X2) at (0.277,1.141);
\coordinate (X3) at (-0.223,0.275);
\coordinate (X4) at (1.000,-0.700);
\coordinate (X5) at (1.500,0.166);
\coordinate (X6) at (1.277,1.141);
\coordinate (X7) at (0.500,-0.700);
\coordinate (X8) at (1.250,-0.267);
\coordinate (X9) at (1.389,0.653);
\coordinate (X10) at (0.777,1.141);
\coordinate (X11) at (0.027,0.708);
\coordinate (X12) at (-0.000,-0.700);
\coordinate (X13) at (-0.111,-0.213);
\draw [dashed, line width=0.5pt,black] (X0)-- (X1);
\draw [dashed, line width=0.5pt,black] (X1)-- (X2);
\draw [dashed, line width=0.5pt,black] (X2)-- (X3);
\draw [dashed, line width=0.5pt,black] (X3)-- (X0);
\fill [opacity=0.10,blue] (X0)-- (X1)-- (X2)-- (X3) -- cycle;
\node[scale=0.40,fill=black,draw,circle] (A0) at (X1) {};
\node[scale=0.40,fill=black,draw,circle] (C0) at (X3) {};
\node[scale=0.40,fill=white,draw,circle] (B0) at (X2) {};
\node[scale=0.40,fill=white,draw,circle] (D0) at (X0) {};
\draw[line width=1.4pt,black] (A0)--(C0);
\draw [dashed, line width=0.5pt,black] (X4)-- (X5);
\draw [dashed, line width=0.5pt,black] (X5)-- (X1);
\draw [dashed, line width=0.5pt,black] (X0)-- (X4);
\fill [opacity=0.10,blue] (X4)-- (X5)-- (X1)-- (X0) -- cycle;
\node[scale=0.40,fill=black,draw,circle] (A1) at (X4) {};
\node[scale=0.40,fill=black,draw,circle] (C1) at (X1) {};
\node[scale=0.40,fill=white,draw,circle] (B1) at (X5) {};
\node[scale=0.40,fill=white,draw,circle] (D1) at (X0) {};
\draw[line width=1.4pt,black] (A1)--(C1);
\draw [dashed, line width=0.5pt,black] (X5)-- (X6);
\draw [dashed, line width=0.5pt,black] (X6)-- (X2);
\fill [opacity=0.10,blue] (X5)-- (X6)-- (X2)-- (X1) -- cycle;
\node[scale=0.40,fill=black,draw,circle] (A2) at (X6) {};
\node[scale=0.40,fill=black,draw,circle] (C2) at (X1) {};
\node[scale=0.40,fill=white,draw,circle] (B2) at (X2) {};
\node[scale=0.40,fill=white,draw,circle] (D2) at (X5) {};
\draw[line width=1.4pt,black] (A2)--(C2);
\end{tikzpicture}
}

}
  & $\longleftrightarrow$ &
\scalebox{0.80}{
\scalebox{0.90}{
\begin{tikzpicture}[scale=1.20 ,baseline=(ZUZU.base)]\coordinate(ZUZU) at (0,0); 
\coordinate (X0) at (1.000,-0.700);
\coordinate (X1) at (0.777,0.275);
\coordinate (X2) at (-0.223,0.275);
\coordinate (X3) at (0.000,-0.700);
\coordinate (X4) at (1.277,1.141);
\coordinate (X5) at (0.277,1.141);
\coordinate (X6) at (1.500,0.166);
\coordinate (X7) at (0.500,-0.700);
\coordinate (X8) at (1.250,-0.267);
\coordinate (X9) at (1.389,0.653);
\coordinate (X10) at (0.777,1.141);
\coordinate (X11) at (0.027,0.708);
\coordinate (X12) at (-0.000,-0.700);
\coordinate (X13) at (-0.111,-0.213);
\draw [dashed, line width=0.5pt,black] (X0)-- (X1);
\draw [dashed, line width=0.5pt,black] (X1)-- (X2);
\draw [dashed, line width=0.5pt,black] (X2)-- (X3);
\draw [dashed, line width=0.5pt,black] (X3)-- (X0);
\fill [opacity=0.10,blue] (X0)-- (X1)-- (X2)-- (X3) -- cycle;
\node[scale=0.40,fill=black,draw,circle] (A0) at (X0) {};
\node[scale=0.40,fill=black,draw,circle] (C0) at (X2) {};
\node[scale=0.40,fill=white,draw,circle] (B0) at (X1) {};
\node[scale=0.40,fill=white,draw,circle] (D0) at (X3) {};
\draw[line width=1.4pt,black] (A0)--(C0);
\draw [dashed, line width=0.5pt,black] (X1)-- (X4);
\draw [dashed, line width=0.5pt,black] (X4)-- (X5);
\draw [dashed, line width=0.5pt,black] (X5)-- (X2);
\fill [opacity=0.10,blue] (X1)-- (X4)-- (X5)-- (X2) -- cycle;
\node[scale=0.40,fill=black,draw,circle] (A1) at (X4) {};
\node[scale=0.40,fill=black,draw,circle] (C1) at (X2) {};
\node[scale=0.40,fill=white,draw,circle] (B1) at (X5) {};
\node[scale=0.40,fill=white,draw,circle] (D1) at (X1) {};
\draw[line width=1.4pt,black] (A1)--(C1);
\draw [dashed, line width=0.5pt,black] (X0)-- (X6);
\draw [dashed, line width=0.5pt,black] (X6)-- (X4);
\fill [opacity=0.10,blue] (X0)-- (X6)-- (X4)-- (X1) -- cycle;
\node[scale=0.40,fill=black,draw,circle] (A2) at (X0) {};
\node[scale=0.40,fill=black,draw,circle] (C2) at (X4) {};
\node[scale=0.40,fill=white,draw,circle] (B2) at (X6) {};
\node[scale=0.40,fill=white,draw,circle] (D2) at (X1) {};
\draw[line width=1.4pt,black] (A2)--(C2);
\end{tikzpicture}
}

}
\\
&&\\

\end{tabular}};\end{tikzpicture}
 \\

\scalebox{1.00}{\begin{tabular}{c}(a) arbitrary \\ polygon $\Reg$\end{tabular}} & \scalebox{1.00}{\begin{tabular}{c}(b) regular \\ polygon $\Reg_N$\end{tabular}} & \scalebox{1.00}{(c) isoradial graph $G_{\Tiling}$} & \scalebox{1.00}{(d) star-triangle move} 

\end{tabular}}

\caption{\label{fig:intro1} 
(a) A rhombus tiling of an arbitrary polygon $\Reg$; (b) a rhombus tiling of a regular polygon $\Reg_\N$ for $\N=6$; (c) the associated isoradial graph $G_\Tiling$ consists of black vertices and black solid edges; (d) a flip of a rhombus tiling resulting in a star-triangle move on $G_\Tiling$. Figure reproduced from~\cite{ising_crit}.}
\end{figure}
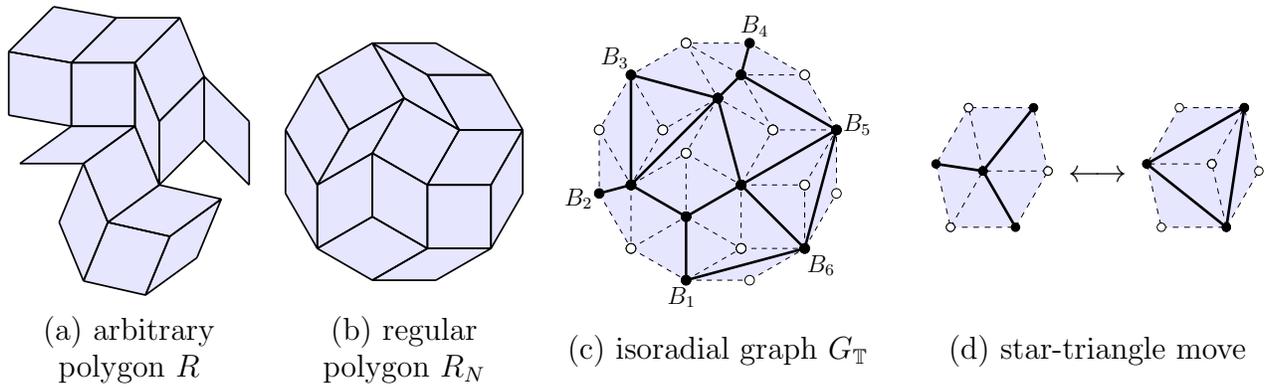

It is known~\cite{KenyonAlg} that any two rhombus tilings of the same region can be related by a sequence of \emph{flips} as in \figref{fig:intro1}(d). Applying a flip to a rhombus tiling results in applying a \emph{star-triangle move} to the electrical network $G_\Tiling$. A well-known property of the electrical response matrix $\La^\Tiling$ is that it is preserved by such moves. Therefore $\La^\Tiling$ depends only on the region $\Reg$ itself, and not on the particular choice of a rhombus tiling $\Tiling$. It is thus natural to denote $\La^\Reg:=\La^\Tiling$. 
A consequence of our boundary measurement formula is a formula for $\La^\Reg$ that depends manifestly only on the region~$\Reg$. 

Lam~\cite{Lam} has constructed an embedding $\pel$ of the space of $\N\times \N$ electrical response matrices into $\Grtnn(\N+1,2\N)$.\footnote{More precisely, Lam's embedding lands in $\Grtnn(\N-1,2\N)$. To get an element of $\Grtnn(N+1,2\N)$, one needs to apply the duality discussed in \cref{sec:duality}.} The image of the map $\pel$ is contained inside the positroid cell labeled by a bounded affine permutation $\fel\in\Bound(N+1,2N)$. One can also choose a (essentially unique) $\fel$-admissible tuple $\bth_\Reg:=(\th_1,\th_2,\dots,\th_n)$ such that the directions of the sides of $\Reg$ are given by $\exp(-2i\th_1),\exp(-2i\th_2),\dots,\exp(-2i\th_n)$ in clockwise order. The following result reflects the well-known connection~\cite[Section~6]{Kenyon} between the critical dimer model and critical electrical networks.
\begin{theorem}\label{thm:intro:appl}
For any region $\Reg$, we have
\begin{equation}\label{eq:intro:appl}
  \pel(\La^\Reg)=\Meas(\fel,\bth_\Reg) \quad\text{inside $\Grtnn(N+1,2N)$.}
\end{equation}
\end{theorem}
\noindent For generic regions, the right hand side of~\eqref{eq:intro:appl} is described by \cref{thm:intro:bound_meas}. The matrix $\La^\Reg$ can be easily recovered from its image under $\pel$; see \cref{sec:applications} for details.

\begin{remark}
An analogous result 
\begin{equation*}%
  \pis(\M^\Reg)=\Meas(\fis,\bth_\Reg)\quad\text{inside $\Grtnn(N,2N)$}
\end{equation*}
 holds for the critical Ising model; see \cref{thm:Meas_appl} and~\cite[Section~7]{KLRR}. Curiously, in both cases, the tuple $\bth_\Reg$ is the same. In addition to being $\fis$-admissible (equivalently, $\fel$-admissible), the tuple $\bth_\Reg$ is required to satisfy an additional \emph{isotropic condition}~\eqref{eq:bth_isotr} which is also identical in the Ising and electrical cases. 
\end{remark}

\subsection{Cyclically symmetric case}\label{sec:intro:cyc_symm}
Let $\bthr=(\th_1,\th_2,\dots,\th_n)$ be given by $\th_r:=\frac{r\pi}{n}$ for all $r\in[n]$. In this case, $\Meas(\fkn,\bthr)$ is easily seen to coincide with the unique cyclically symmetric point\footnote{The point $\XOkn$ is the unique element of $\Grtnn(k,n)$ satisfying $\shift(\XOkn)=\XOkn$, where $\shift:\Grtnn(k,n)\to\Grtnn(k,n)$ is the cyclic shift automorphism discussed in \cref{sec:cyclic-shift}.} $\XOkn\in\Grtnn(k,n)$, studied in~\cite{GKL,Karp}. This result is already of independent interest: previously, the Pl\"ucker coordinates of $\XOkn$ were known, but the corresponding explicit weighted planar bipartite graphs have not been constructed.

We describe the consequences of this observation for electrical networks. Having $\bth=\bthr$ corresponds to the case where the region $\Reg$ is a regular $2\N$-gon, denoted $\Reg_\N$.

\begin{theorem}\label{thm:reg_Elec}
For $1\leq p,q\leq \N$ and $d:=|p-q|$, we have
\begin{equation}\label{eq:reg_Elec}
\la^{\Reg_\N}_{p,q}=\frac{\sin(\pi/\N)}{\N\cdot \sin((2d-1)\pi/2\N)\cdot \sin((2d+1)\pi/2\N)}.
\end{equation}
\end{theorem}
\begin{example}
Consider the star electrical network as in \figref{fig:intro1}(d) inside a regular hexagon $\Reg_3$. Then the resistance of each edge equals $\frac1{\sqrt3}$. Applying the voltage of $1$ to $\b_1$ and the voltage of $0$ to $\b_2$ and $\b_3$, we calculate that the resulting voltage at the unique interior vertex is $\frac13$, and thus the currents through $\b_2$ and $\b_3$ are both equal to $\frac1{\sqrt3}$. This agrees with~\eqref{eq:reg_Elec} for $\N=3$ and $d=1,2$. For $d=0$, we also obtain the correct value $-\frac2{\sqrt3}$ for the current through $\b_1$, the negative sign representing the fact that the current flows \emph{into} the network.
\end{example}

\begin{remark}
Despite the simplicity of \cref{thm:reg_Elec} (and its Ising model analog~\cite[Theorem~1.1]{ising_crit}), both results are apparently new. In the Ising model case, this leads to new asymptotic consequences (including a convergence result to a conformally invariant limit~\cite{ising_crit}). 
\end{remark}

\subsection{Shift by $1$}\label{sec:shift-1}
 For $f\in\Bkn$, its \emph{shift} $\dsh f\in\Bound(k-1,n)$ is defined by $\dsh f(p):=f(p-1)$ for all $p\in\Z$ (this operation is well defined when $f$ is loopless). Taking $p$ and $p-1$ modulo $n$, we obtain a shift map $\fb\to\dsh \fb$ on permutations.

The first appearance of this combinatorial shift map for bounded affine permutations occurred in the construction of the \emph{BCFW triangulation}~\cite{BCFW} of the \emph{amplituhedron}~\cite{AHT}. More precisely, that construction involved a ``shift by $2$,'' corresponding to passing between the momentum space and the momentum-twistor space. A linear-algebraic map from a subset of $\Ptp_f$ to a subset of $\Ptp_{\dsh f}$ can be found on~\cite[Section~8.3]{abcgpt}; see also~\cite[Section~5.2]{LPW}. The combinatorial shift map $f\mapsto \dsh f$ played a major role also in the study of the \emph{parity duality}~\cite{GL} and \emph{$T$-duality}~\cite{LPW} operations for amplituhedra.

The second appearance of the shift map arises when one compares the results of~\cite{GP} for the Ising model with the results of~\cite{Lam} for electrical networks. Specifically, \cite[Question~9.2]{GP} and the discussion below it provides evidence for a stratification-preserving homeomorphism between a subset $\Xel\subset \Grtnn(N+1,2N)$ and a subset $\Xis\subset \Grtnn(N,2N)$ sending $\Xel\cap \Ptp_f\xrasim \Xis\cap \Ptp_{\dsh f}$ homeomorphically for all $f$ for which the intersection is nonempty. One easily checks that the above linear-algebraic map from~\cite{abcgpt} does not provide such a homeomorphism. In view of our current approach, it is natural to additionally require such a map to restrict to a homeomorphism between the critical parts of $\Xel$ and $\Xis$.

In \cref{sec:shift}, we give a new construction that provides partial progress towards this goal. Namely, building on our previous results~\cite{chord_sep} connecting planar bipartite graphs to zonotopal tilings and on the results of~\cite{GPW}, we describe a simple map on the level of weighted planar bipartite graphs that gives the desired result for critical varieties. We discuss the relationship of this map with the boundary measurement map and square moves of planar bipartite graphs, and prove some of its surprising properties. The problem of constructing a stratification-preserving homeomorphism $\Xel\xrasim\Xis$ however remains open.

\subsection*{Acknowledgments}
I am indebted to Pasha Pylyavskyy for his numerous contributions at various stages of the development of~\cite{ising_crit}, where the boundary measurement formula was first discovered in the context of the Ising model. The generalization to the Grassmannian level was inspired by the results of~\cite{CLR,KLRR}, presented by Marianna Russkikh at the ``Dimers in Combinatorics and Cluster Algebras'' conference at the University of Michigan. I thank Marianna for bringing these results to my attention, and also thank the organizers of the conference (Sebastian Franco, Gregg Musiker, Richard Kenyon, David Speyer, and Lauren Williams) for making such an interaction possible. Finally, I am grateful to Lauren Williams and to the anonymous referee for their valuable comments on the first version of the text.

\section{Background on the totally nonnegative Grassmannian}\label{sec:background}
The below constructions are well known in total positivity; see~\cite{Pos,LamCDM} for further details.

\subsection{Bounded affine permutations}\label{sec:BAP}
Positroid varieties are labeled by many families of combinatorial objects. We choose to work with bounded affine permutations introduced in~\cite{KLS}.
\begin{definition}\label{dfn:BAP}
A \emph{$(k,n)$-bounded affine permutation} is a bijection $f:\Z\to\Z$ such that
\begin{itemize}
\item $f(j+n)=f(j)+n$ for all $j\in\Z$,
\item $\sum_{j=1}^{n} (f(j)-j)=kn$, and
\item $j\leq f(j)\leq j+n$ for all $j\in\Z$.
\end{itemize}
\end{definition}
\noindent We let $\Bkn$ denote the (finite) set of $(k,n)$-bounded affine permutations. For $f\in\Bkn$, we let $\fb\in S_n$ be the permutation defined by the condition that $\fb(p)\equiv f(p)\pmod n$.  We let $\fkn\in\Bkn$ be the ``top cell'' bounded affine permutation given by $\fkn(p):=p+k$ for all $p\in\Z$. 
\begin{notation}\label{notn:fkn}
Whenever we have a family $X_f$ of objects labeled by $f\in\Bkn$, we denote $X_\fkn$ by $X_{k,n}$.  
\end{notation}

We say that $f\in\Bkn$ is \emph{loopless} if it satisfies $f(p)>p$ for all $p\in\Z$. Similarly, $f\in\Bkn$ is called \emph{coloopless} if it satisfies $f(p)<p+n$ for all $p\in\Z$. The procedure in \cref{rmk:intro:S_n_loopless} describes a bijection between the symmetric group $S_n$ and the set of loopless $(k,n)$-bounded affine permutations for $1\leq k\leq n$.

The \emph{length} $\ell(f)$ of $f\in\Bkn$ is the number of pairs $s,t\in\Z$ such that $s\in[n]$, $s<t$, and $f(s)>f(t)$. Given such a pair and assuming $f$ is loopless, the reductions of $p:=f(s)$ and $q:=f(t)$ modulo $n$ are said to form an \emph{alignment}; see \figref{fig:align}(b).

\begin{figure}
  \includegraphics{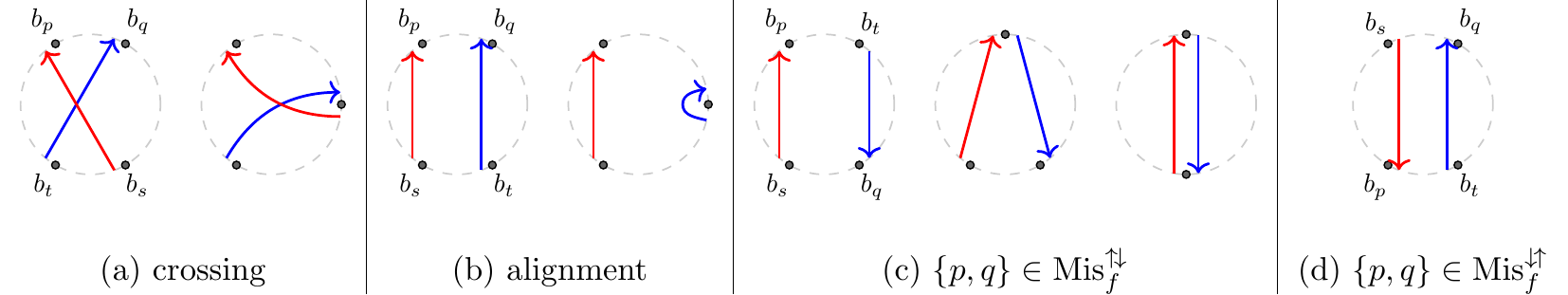}
  \caption{\label{fig:align} Crossings, alignments, and two types of misalignments for loopless bounded affine permutations.}
\end{figure}

\subsection{Planar bipartite graphs}\label{sec:plan-bipart-graphs}
Let $G$ be a planar bipartite graph embedded in a disk as in \cref{sec:intro:plan-bipart-graphs}. First, let us drop the assumption that the boundary vertices of $G$ are colored black. For an almost perfect matching $\Acal$ of $G$, let $I_\Acal\in{[n]\choose k}$ be the set of black boundary vertices used by $\Acal$ together with the set of white boundary vertices \emph{not} used by $\Acal$.

Recall that the boundary vertices of $G$ are labeled by $b_1,b_2,\dots,b_n$. We extend this labeling to all $p\in\Z$ by setting $b_p:=b_{\bar p}$ where $p\in[n]$ is the reduction of $p\in\Z$ modulo $n$.

An \emph{interior leaf} is an interior vertex of degree $1$. We always assume that $G$ admits an almost perfect matching, that every connected component of $G$ contains a boundary vertex, and that each interior leaf of $G$ is adjacent to the boundary. It follows~\cite{Pos} that if $G$ is reduced and $\perm_G(p)=p$ then either
\begin{itemize}
\item $b_p$ is white and is adjacent to a black interior leaf, or
\item $b_p$ is black and is adjacent to a white interior leaf.
\end{itemize}
In the former case, we set $f_G(p):=p$ and say that $p$ is a \emph{loop}, and in the latter case, we set $f_G(p):=p+n$ and say that $p$ is a \emph{coloop}. All other values of $f_G$ are uniquely determined by the values of $\perm_G$ since we require $f_G(p)\equiv \perm_G(p)\pmod n$ for all $p\in[n]$.  From now on, we refer to $f_G$ (as opposed to $\perm_G$) as \emph{the strand permutation} of $G$. Thus the reduced property may be restated as follows: a graph $G$ satisfying the above assumptions is \emph{reduced} if and only if it has $k(n-k)+1-\ell(f_G)$ faces. We refer to reduced planar bipartite graphs simply as \emph{reduced graphs} and denote by $\Gred(f)$ the set of reduced graphs with strand permutation $f\in\Bkn$. 

 We say that $G$ has \emph{black boundary} if all of its boundary vertices are black. The notion of having \emph{white boundary} is defined analogously. Unless stated otherwise, we assume that $G$ has black boundary, and the only other case we consider is when $G$ has white boundary.

For an edge $e\in E(G)$, we say that $e$ is \emph{labeled} by $\{p,q\}$ if the strands passing through $e$ terminate at $b_p$ and $b_q$. It is known~\cite{Pos} that if $p,q$ form an alignment (\figref{fig:align}(b)) then for any $G\in\Gred(f)$, no edge in $G$ is labeled by $\{p,q\}$.

We also label the faces of $G$ by $k$-element subsets of $[n]$. For a face $F$ of $G$, let $\lambda(F)\in{[n]\choose k}$ be the set of all $p\in[n]$ such that $f$ is to the left of the strand ending at $b_p$. This convention is known as \emph{target-labeling} of the faces.

\begin{figure}
  \includegraphics{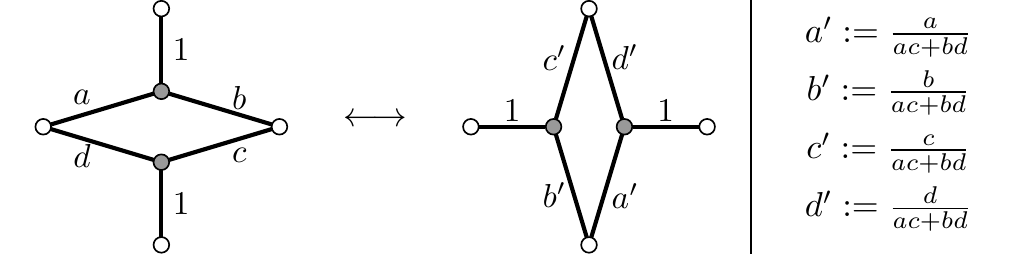}
  \caption{\label{fig:sq-mv} A square move.}
\end{figure}

It was shown in~\cite{Pos} that any two graphs $G,G'\in\Gred(f)$ may be related by a sequence of \emph{square moves} (\cref{fig:sq-mv}) and \emph{contraction-uncontraction moves} (\cref{fig:sq-mv-1}). Specifically, given a square face $F$ of $G$, one first uncontracts some edges so that all vertices of $F$ become trivalent, adding degree $2$ vertices as midpoints of uncontracted edges to preserve the bipartite property.\footnote{When (un)contracting a degree $2$ interior vertex, we always assume that both edges incident to it have weight $1$; this is always achievable by applying gauge transformations.} Next, one applies \emph{gauge transformations} at the black vertices (see \cref{sec:positroid-varieties} below) to fix the weights of the vertical edges in \figref{fig:sq-mv}(left) to $1$. Finally, one performs the local transformation as in \cref{fig:sq-mv}. These moves change the edge weights while preserving the boundary measurements (up to a common scalar).

\subsection{Bridge removal}\label{sec:bridge_rem}
Let $f\in\Bkn$ and $r\in[n]$. Following~\cite[Section~7.4]{LamCDM}, we say that $f$ \emph{has a bridge at $r$} if $f$ satisfies $r<r+1\leq f(r)<f(r+1)\leq r+n$. In this case, there exists a graph $G\in\Gred(f)$ such that the neighborhood of the points $b_r,b_{r+1}$ contains a \emph{bridge configuration} shown in \figref{fig:bridge}(left). Removing the bridge edge yields a configuration in \figref{fig:bridge}(right), and the corresponding graph is also reduced and has strand permutation denoted $s_rf\in\Bkn$ which sends $r\mapsto f(r+1)$, $r+1\mapsto f(r)$, and $q\mapsto f(q)$ for all $q\in\Z$ not congruent to $r$ or $r+1$ modulo $n$. 
Any $f\in\Bkn$ without loops and coloops has a bridge at some $r\in[n]$. Thus, starting with any $f\in\Bkn$ and removing bridges, loops, and coloops, we can always reach a permutation in either $\Bound(0,1)$ or $\Bound(1,1)$. 

\subsection{Positroid varieties}\label{sec:positroid-varieties}
Recall that the Grassmannian $\Gr(k,n)$ is identified with the space of full rank complex $k\times n$ matrices modulo row operations. Given a $k\times n$ matrix $A$, we let $\RowSpan(A)\in\Gr(k,n)$ denote its row span and $A_1,A_2,\dots, A_n$ be its columns. We extend this to a sequence $(A_q)_{q\in\Z}$ by requiring
\begin{equation}\label{eq:columns}
  A_{q+n}=(-1)^{k-1}A_q \quad\text{for all $q\in\Z$.}
\end{equation}
The sign twist is related to the cyclic shift automorphism of $\Grtnn(k,n)$ discussed in \cref{sec:cyclic-shift}.  For a full rank $k\times n$ matrix $A$, we let $f_A:\Z\to\Z$ be given by
\begin{equation}\label{eq:f_dfn}
  f_A(p)=\min\{q\geq p\mid A_p\in \Span \left(A_{p+1},A_{p+2},\dots,A_q\right)\} \quad\text{for $p\in\Z$.}
\end{equation}

\begin{figure}
  \includegraphics{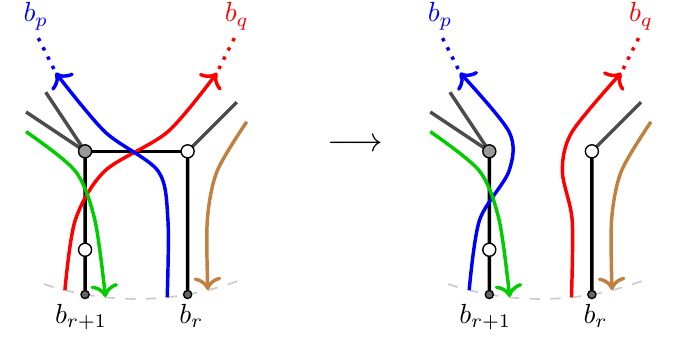}
  \caption{\label{fig:bridge} Removing a bridge (\cref{sec:bridge_rem}).}
\end{figure}

For example, if $A_p$ is a zero column (i.e., a \emph{loop}) then $f_A(p)=p$, and if $A_p$ is not in the span of other columns (i.e., a \emph{coloop}) then $f_A(p)=p+n$. It is known~\cite{KLS} that $f_A$ is a $(k,n)$-bounded affine permutation which depends only on the row span of $A$. The \emph{positroid stratification} of $\Gr(k,n)$ is given by
\begin{equation*}
  \Gr(k,n)=\bigsqcup_{f\in \Bkn} \Pio_{f}, \quad\text{where}\quad \Pio_{f}:=\{\RowSpan(A)\in\Gr(k,n)\mid f_A=f\}.
\end{equation*}
We let $\Ptp_f:=\Pio_f\cap \Grtnn(k,n)$ denote the corresponding \emph{positroid cell}. We also have the \emph{positroid variety} $\Pi_f$ which is the Zariski closure of $\Ptp_f$ (equivalently, of $\Pio_f$). In fact, $\Pio_f$ is an explicit open subvariety of $\Pi_f$. Namely, for each $q\in\Z$, let 
\begin{equation}\label{eq:It_q}
  \It_q:=\{f(p)\mid \text{$p\in\Z$ is such that $p<q\leq f(p)$}\}.
\end{equation}
For $q\in[n]$, let $I_q\in{[n]\choose k}$ be obtained from $\It_q$ by reducing all elements modulo $n$. The sequence $\Ical_f:=(I_1,I_2,\dots,I_n)$ is called the \emph{Grassmann necklace} of $f$. Alternatively, for loopless $f$, we have $I_q=J_q\sqcup\{q\}$ for all $q\in[n]$, where $J_q$ was defined in~\eqref{eq:intro:J_r}. We have
\begin{equation}\label{eq:Pio_open}
  \Pio_f:=\{X\in \Pi_f\mid \Delta_{I_q}(X)\neq 0\text{ for all $q\in[n]$}\}.
\end{equation}
Grassmann necklaces also allow one to describe $\Pi_f$ as an explicit subvariety of $\Gr(k,n)$. Namely, each $f\in\Bkn$ gives rise to a \emph{positroid} $\Mcal_f\subset{[n]\choose k}$ defined as follows. For each $q\in[n]$, introduce a total order $\preceq_q$ on $[n]$ given by 
\begin{equation*}%
  q\preceq_q q+1\preceq_q\dots\preceq_q q-1,
\end{equation*}
where the indices are taken modulo $n$. For two $k$-element sets $I=\{i_1\prec_q i_2\prec_q\dots\prec_q i_k\}$ and $J=\{j_1\prec_q j_2\prec_q\dots\prec_q j_k\}$, we write $I\preceq_q J$ if $i_r\preceq_q j_r$ for all $r\in[k]$. Then the positroid $\Mcal_f$ consists of all sets $J\in{[n]\choose k}$ such that $I_q\preceq_q J$ for all $q\in[n]$. (Thus a positroid is an intersection of $n$ cyclically shifted Schubert matroids.) The variety $\Pi_f$ is described by
\begin{equation}\label{eq:Pi_minors_descr}
  \Pi_f=\{X\in\Gr(k,n)\mid \Delta_J(X)=0\text{ for all $J\notin \Mcal_f$}\}.
\end{equation}
Finally, we have
\begin{equation}\label{eq:Ptp_minors_descr}
  \Ptp_f=\{X\in\Gr(k,n)\mid \Delta_J(X)>0\text{ for $J\in \Mcal_f$ and $\Delta_J(X)=0$ otherwise}\}.
\end{equation}
The dimension of $\Ptp_f$ (as well as $\Pio_f$, and $\Pi_f$) is given by $k(n-k)-\ell(f)$, where $\ell(f)$ is the length of $f$ introduced in \cref{sec:BAP}.

Positroid cells $\Ptp_f$ and Grassmann necklaces were first studied by Postnikov~\cite{Pos} while positroid varieties $\Pi_f$ and their open subvarieties $\Pio_f$ were introduced by Knutson--Lam--Speyer~\cite{KLS}.

Let $E=E(G)$ be the edge set of a reduced graph $G$. The map $\Measop_G:\Rtp^E\to \Grtnn(k,n)$ restricts to a homeomorphism $\Measop_G:\Rtp^E/\Gau\xrasim \Ptp_{f_G}$, where $\Rtp^E/\Gau$ denotes the space of positive edge weights of $G$ considered modulo \emph{gauge transformations}, that is, rescalings of the weights of all edges incident to a given interior vertex. For each (interior or boundary) face $F$ of $G$, let $e_1,e_2,\dots,e_{2m}$ be the edges on the boundary of $F$ in clockwise order. The number of edges is even since we are assuming that $G$ has either black boundary or white boundary. For any weight function $\wt\in\Rtp^E$, we may consider an alternating product
\begin{equation}\label{eq:wt_alt_prod}
  \frac{\wt(e_1)\wt(e_3)\cdots \wt(e_{2m-1})}{\wt(e_2)\wt(e_4)\cdots \wt(e_{2m})}.
\end{equation}
It is clearly invariant under gauge transformations, and in fact may be recovered from $\Meas(G,\wt)$ using the \emph{left twist map} (see \cite[Corollary~5.11]{MuSp}) of Muller--Speyer discussed in \cref{sec:twist}.

\section{Cyclic symmetry and duality}
\label{sec:cyclic-shift-duality}
In this section, we discuss how critical varieties are affected by some natural operations on the totally nonnegative Grassmannian, namely, cyclically shifting the columns and taking orthogonal complements.

\subsection{Cyclic symmetry}\label{sec:cyclic-shift}
The totally nonnegative Grassmannian $\Grtnn(k,n)$ admits a non-trivial shift homeomorphism $\shift:\Grtnn(k,n)\to\Grtnn(k,n)$. It sends (the row span of) a matrix $A$ with columns $A_1,A_2,\dots, A_n$ to (the row span of) the matrix $\shift(A)$ with columns $A_2,\dots A_n,(-1)^{k-1} A_1$. The sign $(-1)^{k-1}$ ensures that the nonnegativity of maximal minors is preserved. The map $\shift$ restricts to a homeomorphism $\shift:\Ptp_f\xrasim \Ptp_{\shf}$, where $\shcomb:\Z\to\Z$ sends $p\mapsto p+1$ for all $p\in\Z$. Thus $\shf\in\Bkn$ is defined by $(\shf)(p)=f(p+1)-1$ for $p\in\Z$. 

Note that the definition of the critical cell $\Ctp_f$ in \cref{sec:intro-critical-varieties} does not appear to respect this cyclic symmetry, since we choose the edge weights to be $\sin(\th_q-\th_p)$ for $1\leq p<q\leq n$. Nevertheless, we have the following result.

\begin{proposition}\label{prop:shift_Ctp}
For a loopless $f\in\Bkn$, the map $\shift$ restricts to a homeomorphism
\begin{equation*}%
  \shift: \Ctp_f\xrasim \Ctp_{\shf}.
\end{equation*}
\end{proposition}

While this result is not hard to see directly, we prefer to use this opportunity to introduce \emph{affine notation} that reflects the cyclic symmetry of critical cells. First, we always extend a tuple $\bth=(\th_1,\th_2,\dots,\th_n)\in\R^n$ to an infinite sequence $\btht:\Z\to\R$ uniquely determined by the conditions $\tht_p=\th_p$ for $p\in[n]$ and  %
\begin{equation}\label{eq:bth_extend}
  \tht_{p+n}=\tht_p+\pi \quad\text{for all $p\in\Z$.}
\end{equation}
Since  $\bth$ and $\btht$ determine each other, we use them interchangeably and write e.g. $\Meas(f,\btht)$ for $\Meas(f,\bth)$.

Let us describe $f$-admissibility in the affine language. For $p,q\in\Z$, we say that $(p,q)$ form an \emph{affine $f$-crossing} if we have $s<t<p<q\leq s+n$, where $s:=f^{-1}(p)$ and $t:=f^{-1}(p)$.

It is easy to check that if $1\leq p<q\leq n$ form an $f$-crossing then either $(p,q)$ or $(q,p+n)$ form an affine $f$-crossing. Conversely, if $(p,q)$ form an affine $f$-crossing then their reductions  $\bar p, \bar q\in[n]$ modulo $n$ form an $f$-crossing. It follows that a tuple $\bth=(\th_1,\th_2,\dots,\th_n)$ is $f$-admissible if and only if the corresponding sequence $\btht:\Z\to\R$ satisfies
\begin{equation}\label{eq:affine_f_crossing}
  \tht_p<\tht_q<\tht_{p+n} 
\end{equation}
whenever $(p,q)$ form an affine $f$-crossing. In this case, we say that $\btht$ is \emph{$f$-admissible}.

For a sequence $\btht:\Z\to\R$, let $\shbtht:\Z\to\R$ be given by $(\shbtht)_p:=\tht_{p+1}$. \Cref{prop:shift_Ctp} follows from the following observation.
\begin{lemma}
For a loopless $f\in\Bkn$, a sequence $\btht:\Z\to\R$ is $f$-admissible if and only if $\shbtht$ is $(\shf)$-admissible. In this case, we have
\begin{equation}\label{eq:Meas_shift}
  \Meas(f,\btht)=\Meas(\shf,\shbtht).
\end{equation}
\end{lemma}
\begin{proof}
The $f$-admissibility part is clear from~\eqref{eq:affine_f_crossing}. To prove~\eqref{eq:Meas_shift}, consider a graph $G\in\Gred(f)$. Relabeling its boundary vertices as $(b_n,b_1,\dots,b_{n-1})$, we obtain a graph $G'\in\Gred(\shf)$. Let $e\in E(G)=E(G')$ be an edge not adjacent to the boundary, and suppose that it is labeled in $G$ by $\{p,q\}$ with $1\leq p<q\leq n$. The weights $\wt_{\btht}(e)$ and $\wt'_{\shbtht}(e)$ coincide unless $p=1$. If $p=1$ then $\wt_{\btht}(e)=\sin(\th_q-\th_1)$ while $\wt'_{\shbtht}(e)=\sin(\th'_{n}-\th'_{q-1})$, where $\btht':=\shbtht$. It remains to note that $\th'_n=\th_{n+1}=\th_1+\pi$ and $\th'_{q-1}=\th_q$, thus $\wt'_{\shbtht}(e)=\sin(\th_q-\th_1)=\wt_{\btht}(e)$ in this case as well.
\end{proof}

\subsection{Duality}\label{sec:duality}
For $V\in\Gr(k,n)$, denote by $V^\perp\in\Gr(n-k,n)$ its orthogonal complement. We let $\alt(V)\in\Gr(k,n)$ be obtained from $V$ by changing the sign of every second column of the matrix representing $V$. We set 
\begin{equation*}%
  \altp(V):=\alt(V^\perp)=(\alt(V))^\perp.
\end{equation*}
We discuss several well-known properties of the map $\altp$; see e.g.~\cite[Lemma~1.11]{Karp_var} and references therein. The map $\altp$ restricts to an involutive homeomorphism $\Grtnn(k,n)\xrasim \Grtnn(n-k,n)$. For $f\in\Bkn$, let $\fd\in\Bound(n-k,n)$ be given by
\begin{equation}\label{eq:fd}
  \fd(p)=f^{-1}(p)+n \quad\text{for all $p\in\Z$.}
\end{equation}
The map $f\mapsto \fd$ is an involution. The map $\altp$ restricts to a homeomorphism $\Ptp_f\xrasim \Ptp_{\fd}$. It satisfies
\begin{equation}\label{eq:altp_minors}
  \Delta_I(X)=\Delta_{[n]\setminus I}(\altp(X)) \quad\text{for all $I\in{[n]\choose k}$}.
\end{equation}

Recall from \cref{sec:intro-critical-varieties} that any loopless $f\in\Bkn$ gives rise to a critical cell $\Ctp_f$. We are interested in the effect of the map $\altp$ on critical cells. Note that in general $\fd\in\Bound(n-k,n)$ need not be loopless, but it is coloopless. What we will show below is that the map $\altp$ sends critical cells to \emph{dual critical cells}.

Recall that we have placed points $\m p, b_p, \p p$ on the circle for each $p\in[n]$. 
\begin{definition}
Let $f\in\Bkn$ be coloopless. The \emph{dual reduced strand diagram} of $f$ is obtained by drawing a straight arrow $\m s\to \p p$ whenever  $\fb(s)=p$. We say that $p\neq q$ form a \emph{dual $f$-crossing} if the arrows $\m s\to \p p$ and $\m t\to \p q$ cross. We say that $\bth=(\th_1,\th_2,\dots,\th_n)$ is \emph{dual $f$-admissible} if whenever $1\leq p<q\leq n$ form a dual $f$-crossing,~\eqref{eq:adm_condition} is satisfied. 
\end{definition}
\noindent In general, the reduced strand diagram of $f$ and the dual reduced strand diagram of the same $f$ behave quite differently; see \cref{fig:dual} for an example.

\begin{figure}
  \includegraphics{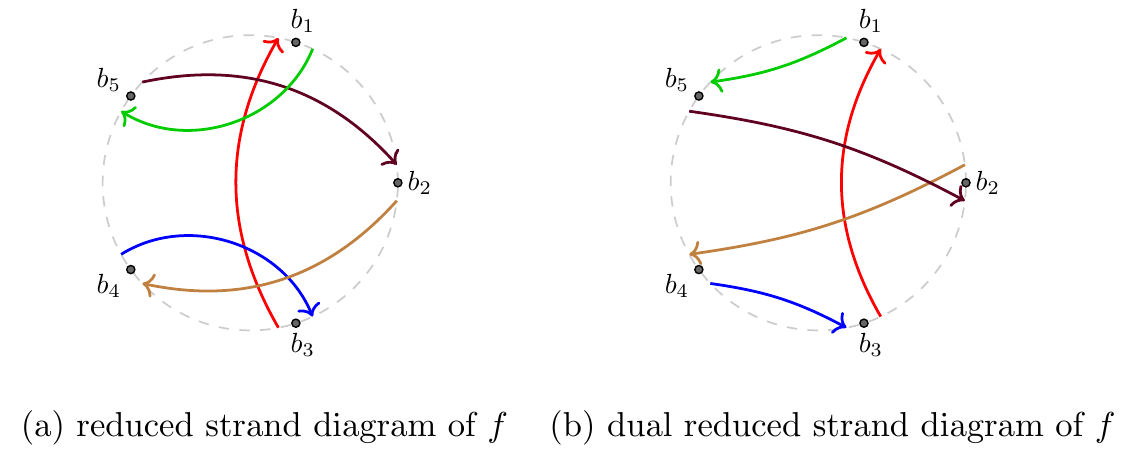}
  \caption{\label{fig:dual} A (dual) reduced strand diagram.}
\end{figure}

Let $f\in\Bkn$ be coloopless and let $G\in\Gred(f)$ be a graph with black boundary. For the purposes of this section, if $f(p)=p$ for some $p$ then by convention we treat $b_p$ as a black boundary vertex of degree zero, or, equivalently, as a white boundary vertex adjacent to a black interior leaf. 

 For an edge $e\in E(G)$ labeled by $\{p,q\}$, set
\begin{equation}\label{eq:wt_bth_dual}
  \wtd_\bth(e):=\sin(\th_q-\th_p).
\end{equation}
This is different from~\eqref{eq:wt_sin} in that the boundary edges no longer have weight $1$. Let 
\begin{equation*}%
  \Measd(f,\bth):=\Meas(G,\wtd_\bth).
\end{equation*}
Finally, define the \emph{dual critical cell}
\begin{equation*}%
    \Ctpd_f:=\{\Measd(f,\bth)\mid \text{$\bth=(\th_1,\th_2,\dots,\th_n)$ is a dual $f$-admissible tuple}\}.
\end{equation*}
\begin{proposition}\label{prop:duality}
Let $f\in\Bkn$ and $\btht:\Z\to\R$.
\begin{theoremlist}
\item\label{item:duality:adm} $\btht$ is $f$-admissible $\Longleftrightarrow$ $\btht\circ f$ is dual $\fd$-admissible.
\item\label{item:duality:Meas} $\Meas(f,\btht)=\altp(\Measd(\fd,\btht\circ f))$.
\item\label{item:duality:altp} The map $\altp$ yields an involutive homeomorphism
\begin{equation*}%
  \Ctp_f\cong \Ctpd_{\fd}.
\end{equation*}
\end{theoremlist}
\end{proposition}
\begin{remark}
Our constructions are invariant with respect to adding the same constant to all values of $\btht$. Modulo such transformations, the map $(f,\btht)\mapsto (\fd,\btht\circ f)$ is an involution.
\end{remark}
\begin{proof}
Suppose that $\btht$ is $f$-admissible and let $\btht':=\btht\circ f$. Let us say that $p<q\in\Z$ form a \emph{dual affine $\fd$-crossing} if we have $s<t\leq p<q<s+n$, where $\fd(s)=p$ and $\fd(t)=q$. As in \cref{sec:cyclic-shift}, we see that $\btht'$ is dual $\fd$-admissible if and only if for all $p<q$ forming a dual affine $\fd$-crossing, we have $\tht'_p<\tht'_q<\tht'_{p+n}$. We claim that this is equivalent to $\tht_s<\tht_t<\tht_{s+n}$. Indeed, by definition, we have $p=\fd(s)=f^{-1}(s)+n$ and $q=\fd(t)=f^{-1}(t)+n$, thus $f(p)=s+n$ and $f(q)=t+n$. The inequalities $s<t\leq p<q<s+n$ may be rewritten as $p<q<s+n<t+n\leq p+n$, and since $f(p)=s+n$ and $f(q)=t+n$, we see that $s<t$ form an affine $f$-crossing. Since $\btht$ is $f$-admissible, we find $\tht_s<\tht_t<\tht_{s+n}$, so $\btht'$ is dual $\fd$-admissible, proving the forward direction of~\itemref{item:duality:adm}. The converse direction is handled similarly.

To show~\itemref{item:duality:Meas}, choose $G\in\Gred(f)$. Let $\Gd$ be obtained by changing the colors of all vertices of $G$. Thus $\Gd$ has white boundary and strand permutation $\fd$. Let $\Gd'$ be obtained from $\Gd$ by putting a white degree $2$ vertex on each boundary edge of $\Gd$ and then changing the color of all boundary vertices to black. 

By~\eqref{eq:altp_minors}, it follows that $\altp\circ \Measop_G=\Measop_{\Gd}$ as maps $\Rtp^E\to \Grtnn(n-k,n)$. Here we identify the sets of edges of $G$ and $\Gd$ and denote them by $E$. Let $e\in E$ be an edge whose weight in $G$ is $\wt_\bth(e)=\sin(\th_q-\th_p)$. Suppose that $\fb(s)=p$ and $\fb(t)=q$ for some $s,t,p,q\in[n]$ such that $p<q$. Then the weight $\wtd_{\bth'}(e)$ in $\Gd'$ is given by $\sin(\th'_s-\th'_t)$ if $s>t$ and by $\sin(\th'_t-\th'_s)$ if $s<t$, where $\btht'=\btht\circ f$ as above. Note also that $\th'_s=\th_p$ if $s<p$ and $\th'_s=\th_p+\pi$ if $s>p$, and similarly for $\th'_q$. As explained in \cref{sec:plan-bipart-graphs}, since $e$ is labeled by $\{p,q\}$, the strands $s\to p$ and $t\to q$ cannot form an alignment. Using this condition, one checks directly that we have $\wtd_{\bth'}(e)=\sin(\th_q-\th_p)$ in all cases. Thus we have $\wt_\bth(e)=\wtd_{\bth'}(e)$ for each interior edge $e\in E$. If $e$ is a boundary edge of $G$ then $\wt_\bth(e)=1$ and $e$ corresponds to two edges $e',e''$ in $\Gd'$ sharing a white degree $2$ vertex and satisfying $\wtd_{\bth'}(e')=\wtd_{\bth'}(e'')$. Thus applying a gauge transformation at these white degree $2$ vertices, we find  $\altp\circ \Measop_G(\wt_\bth)=\Measop_{\Gd}(\wtd_{\bth'})$. 
 This completes the proof of~\itemref{item:duality:Meas}, and~\itemref{item:duality:altp} follows from~\itemref{item:duality:Meas} as a direct corollary.
\end{proof}

\section{Connected components and strand diagrams}\label{sec:conn-comp-strand}
Recall that we have defined in \cref{sec:intro-critical-varieties} a critical cell $\Ctp_f$ for any loopless $f\in\Bkn$. We mentioned that the edge weights in $\wt_\bth$ are positive when $\bth$ is $f$-admissible and  discussed the relationship between the dimension of $\Ctp_f$ and connected components of the reduced strand diagram of $f$. In this section, we justify these claims, studying the combinatorics of reduced strand diagrams along the way.

\begin{figure}
  \includegraphics{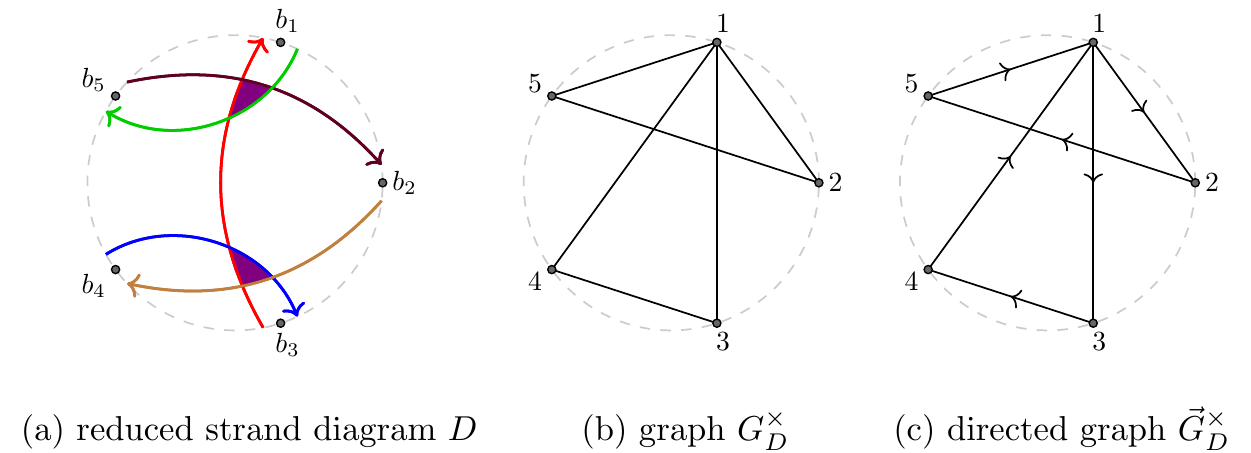}
  \caption{\label{fig:GX} The graphs $\GX_D$ and $\GXV_D$ defined in~\eqref{eq:E_GX_f} and~\eqref{eq:E_GXV}. The shaded regions on the left represent convex regions from \cref{dfn:convex}.}
\end{figure}

\subsection{Connected components}\label{sec:connected-components}
We start by stating several results concerning connected components of the reduced strand diagram of $f$. Their proofs turn out to be quite involved, and are deferred to later sections.

Fix a loopless $f\in\Bkn$. Let $\GX_f$ be the undirected graph with vertex set $[n]$ and edge set 
\begin{equation}\label{eq:E_GX_f}
  E(\GX_f):=\{\{p,q\}\mid p,q\text{ form an $f$-crossing}\}.
\end{equation}
See \figref{fig:GX}(b) for an example. Let $\conn_f$ be the number of connected components of $\GX_f$. (Thus $\conn_f$ is the number of connected components of the reduced strand diagram of $f$, viewed as a topological union of strands in a disk.) We set $\dimf:=n-\conn_f$.

\begin{definition}
We say that a reduced graph $G$ is \emph{contracted} if it has no degree $2$ vertices that are not adjacent to the boundary.
\end{definition}
The following result will be proved in \cref{sec:proof-of_edge_conn_cpts}.
\begin{proposition}\label{prop:edges_conn_cpts}
Suppose a non-boundary edge $e$ of a graph $G\in\Gred(f)$ is labeled by $\{p,q\}$ with $1\leq p<q\leq n$. Then for any $f$-admissible tuple $\bth$, we have
\begin{equation}\label{eq:edges_conn_cpts}
  \th_p<\th_q<\th_p+\pi.
\end{equation}
In particular, $p$ and $q$ belong to the same connected component of $\GX_f$ and we have $\wt_\bth(e)>0$.
\end{proposition}

Let $C\subset[n]$ be a connected component of $\GX_f$. By \cref{prop:edges_conn_cpts}, adding a constant to $\th_p$ for all $p\in C$ preserves $\Meas(f,\bth)$.  
 Choose some representatives $p_1,p_2,\dots,p_{\conn_f}\in[n]$, one from each connected component of $\GX_f$. Let
\begin{equation}\label{eq:THtp_dfn}
  \THtp_f:=\{\bth=(\th_1,\th_2,\dots,\th_n)\in\R^n\mid \text{$\bth$ is $f$-admissible and $\th_{p_1}=\th_{p_2}=\dots=\th_{p_{\conn_f}}=0$}\}.
\end{equation}
Thus the map $\Measf:\THtp_f\to \Ctp_f$ is surjective. Note that $\THtp_f$ is easily seen to be homeomorphic to $\Rtp^{\dimf}$ since it may be identified with the interior of a $\dimf$-dimensional polytope. 
\begin{conjecture}[The injectivity conjecture]\label{conj:inj}
The map $\Measf:\THtp_f\to \Ctp_f$ is a homeomorphism.
\end{conjecture}

The special case of \cref{conj:inj} for the top cell is proved in \cref{sec:proofs_top_cell}.
\begin{theorem}\label{thm:inj_fkn} 
The injectivity conjecture holds for $f=\fkn$ for $1\leq k\leq n$. In particular, %
\begin{equation*}%
  \dim_\C(\Critkn)=n-1 \quad\text{and}\quad \Ctpkn\cong \R_{>0}^{n-1} \quad\text{for $2\leq k\leq n-1$.}
\end{equation*}
\end{theorem}

\subsection{Reduced strand diagrams}\label{sec:RSD}
As in \cref{sec:intro-critical-varieties}, let us consider a disk with $2n$ boundary points $\m1, \p1,\m2,\p2,\dots,\m n,\p n$ ordered clockwise. 

\begin{definition}\label{dfn:strand_diag}
A \emph{strand diagram} $D$ is a collection of $n$ smooth oriented paths (\emph{strands}) $\str_1,\str_2,\dots,\str_n$ in a disk such that 
\begin{enumerate}
\item no three strands intersect at one point and no strand intersects itself,
\item all intersections are transversal and lie in the interior of the disk,
\item each strand $\str_p$ starts at $\p s$ (for some $s\in[n]$) and ends at $\m p$.
\end{enumerate}
\end{definition}
\noindent Thus the directions of the strands alternate around the boundary of the disk. 

As a special case, this definition contains Postnikov's \emph{alternating strand diagrams}~\cite{Pos}, which are essentially the diagrams that arise by drawing the strands associated to a given reduced graph $G$. Alternating strand diagrams usually contain pairs of strands that intersect multiple times. We will focus on the opposite special case.
\begin{definition}\label{dfn:RSD}
A \emph{reduced strand diagram} is a strand diagram in which any two strands intersect at most once.
\end{definition}
\noindent  For example, \figref{fig:move}(b) contains an alternating strand diagram on the left and a reduced strand diagram on the right. \Cref{dfn:RSD} includes \cref{dfn:intro:strand_diag} as a special case: given $f\in\Bkn$, we denote by $D_f$ its reduced strand diagram where all  strands are straight. 

Let $D$ be a reduced strand diagram. We introduce two graphs, $\GX_D$ and $\GXV_D$ with vertex set $[n]$. As before, $\GX_D$ is an undirected graph containing an edge $\{p,q\}$ whenever $\str_p$ and $\str_q$ form a crossing. Note that $\GX_D$ admits a natural orientation. Let us say that two strands $\str_p$ and $\str_q$ form a \emph{positive crossing}  if the points $\p s, \p t, \m p,\m q$ are cyclically ordered clockwise, where $\str_p$ connects $\p s\to \m p$ and $\str_q$ connects $\p t \to \m q$; see \figref{fig:align}(a). We then let $\GXV_D$ be the directed graph containing an edge $(p\to q)$ whenever $\str_p$ and $\str_q$ form a positive crossing. See \figref{fig:GX}(c).

\begin{definition}\label{dfn:convex_GXV}
An \emph{increasing cycle} is a directed cycle $(a_1\to a_2\to\dots\to a_m\to a_1)$ in $\GXV_D$ such that $1\leq a_1<a_2<\dots<a_m\leq n$.
\end{definition}
In the special case where each strand is straight, if the strands $\str_{a_1},\str_{a_2},\dots,\str_{a_m}$ bound a convex region $R$ in the disk such that the boundary of $R$ is oriented either clockwise or counterclockwise then the corresponding edges of $\GXV_D$ form an increasing cycle (after cyclically shifting the indices). We will define convex regions for arbitrary reduced strand diagrams below in \cref{dfn:convex}.  
 A fundamental tool that we will use to study critical cells is the following result which states that every crossing of $D$ is a vertex of such a convex region. 

\begin{figure}
  \includegraphics{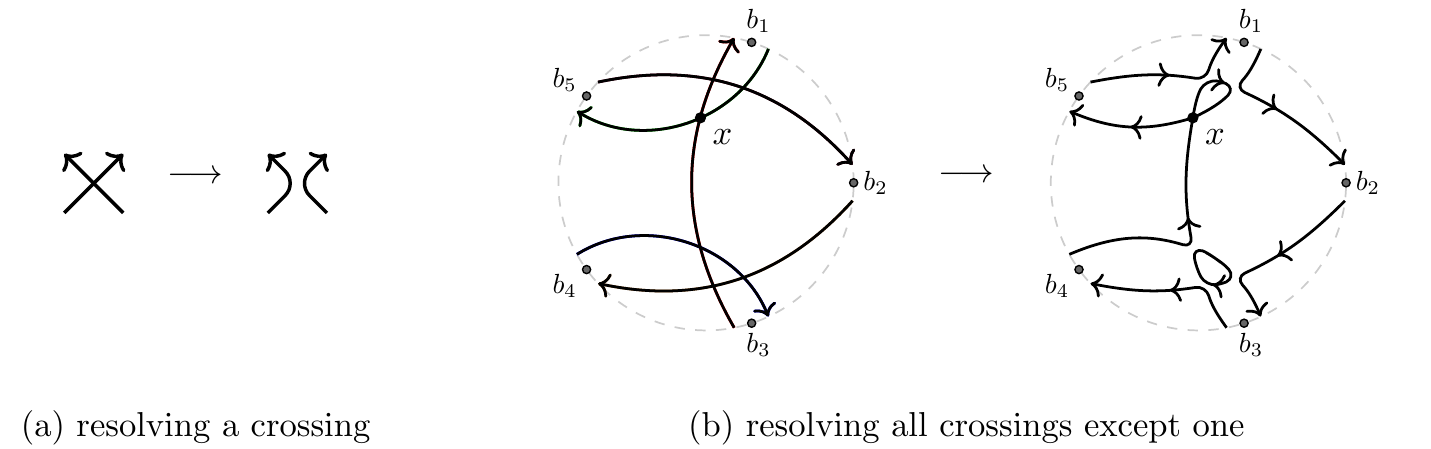}
  \caption{\label{fig:resolve} Resolving crossings in reduced strand diagrams. For the picture on the right, the point $x$ belongs to two strands denoted $\str'_p$ and $\str'_q$ in the proof of \cref{prop:convex}.}
\end{figure}

\begin{proposition}\label{prop:convex}
Every edge of $\GXV_D$ belongs to an increasing cycle.
\end{proposition}
\noindent For instance, the diagram in \figref{fig:GX}(a) contains six crossings, and they form two convex regions (shaded triangles). This corresponds to having increasing cycles $(1\to2\to5\to1)$ and $(1\to3\to4\to1)$ in $\GXV_D$ shown in \figref{fig:GX}(c).
\begin{proof}

Let $e=(p\to q)$ be an edge of $\GXV_D$, thus $\str_p$ and $\str_q$ form a positive crossing. Let $\cpt$ be the intersection point of $\str_p$ and $\str_q$. Let $D_\cpt$ be obtained from $D$ by ``resolving'' all crossings except for $\cpt$; see \figref{fig:resolve}(b) for an example. Here \emph{resolving} a crossing $\cpt'$ is a local transformation that replaces a neighborhood of a crossing point $\cpt'$ as shown in \figref{fig:resolve}(a).  Thus $D_\cpt$ contains a single crossing point $\cpt$. Some of the strands of $D_\cpt$ are closed curves in the interior of the disk, however, it is still true that for each $r\in[n]$, one strand of $D_\cpt$ starts at $\p r$ and one strand of $D_\cpt$ ends at $\m r$. Consider the two strands $\str'_p$ and $\str'_q$ of $D_\cpt$ emanating from $\cpt$. Assume that they both terminate at the boundary of the disk, say, at points $\m{p'}$ and $\m{q'}$ for some $p',q'\in[n]$. Then the arc between $\m{p'}$ and $\m{q'}$ contains more starting strands than ending  strands. None of such strands can intersect  either $\str'_p$ or $\str'_q$. We arrive at a contradiction. Thus at least one of the strands $\str'_p$ or $\str'_q$ does not terminate at the boundary, and therefore it terminates at $\cpt$. Thus it forms a closed directed path $\cyc'$ starting and ending at $\cpt$; see \figref{fig:resolve}(b).

Let us introduce another directed graph, the \emph{topological graph} $\Gtop_D$ of $D$. The vertices of $\Gtop_D$ are the crossing points and the boundary vertices of the strands of $D$. Thus each strand of $D$ passes through the vertices $u_0,u_1,\dots,u_r$ of $\Gtop_D$, where $u_0$ and $u_r$ lie on the boundary of the disk. The edge set of $\Gtop_D$ consists of these directed line segments $(u_0\to u_1)$, \dots, $(u_{r-1}\to u_r)$ for all strands of $D$. Thus each interior vertex of $\Gtop_D$ has two incoming and two outgoing edges.

\begin{figure}
  \includegraphics{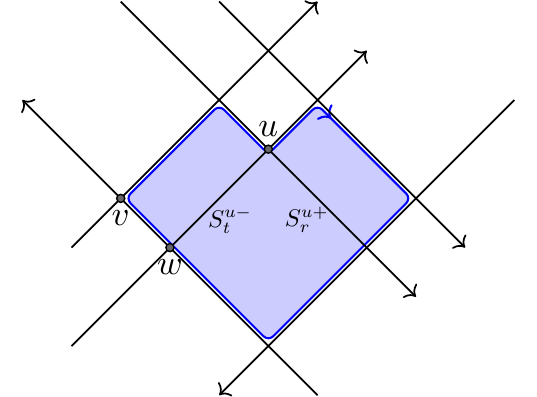}
  \caption{\label{fig:cycle-turn} A (clockwise) cycle $\cyc$ which turns right at $v$, goes straight at $w$, and turns left at $u$.}
\end{figure}

\begin{definition}\label{dfn:convex}
Consider a simple directed cycle $\cyc$ in $\Gtop_D$ passing through a vertex $u$. Then $\cyc$ either \emph{turns right}, \emph{turns left}, or \emph{goes straight} at $u$, depending on which incoming and which outgoing edge of $u$ it uses; see \cref{fig:cycle-turn}. We say that $\cyc$ is \emph{clockwise convex} if it either turns right or goes straight at each of its vertices. Similarly, $\cyc$ is \emph{counterclockwise convex} if it either turns left or goes straight at each of its vertices. Each (counter)clockwise convex cycle is a Jordan curve, and we say that it bounds a \emph{(counter)clockwise convex region}.
\end{definition}

Recall that we have constructed a directed cycle $\cyc'$ in $\Gtop_D$ passing through $\cpt$. It is straightforward to check that this cycle is simple, i.e., passes through each vertex at most once. By construction, $\cyc'$ does not go straight at $\cpt$. Without loss of generality, let us assume that it turns right at $\cpt$. Therefore $\cyc'$ is oriented clockwise around the boundary of a (not necessarily convex) region $R'$. Our goal is to find a clockwise convex region $R\subset R'$ that contains $\cpt$. Let $\cyc$ be a cycle that turns right at $\cpt$ and bounds a region $R\subset R'$ of minimal possible area. Then we claim that $\cyc$ is clockwise convex. Indeed, suppose otherwise that it turns left at some vertex $u$ of $\Gtop_D$; see e.g. \cref{fig:cycle-turn}. Consider the strand $\str_r$ passing through the unique outgoing edge of $u$ that is not used by $\cyc$. Let us consider the part $\str_r^{u+}$ of $\str_r$ between $u$ and the endpoint of $\str_r$. The path $\str_r^{u+}$ must intersect $\cyc$ at some other vertex. Let $u'\neq u$ be the first vertex on $\str_r^{u+}$ that belongs to $\cyc$. We see that $\cyc$ either turns left or goes straight at $u'$. In particular, neither $u$ nor $u'$ is equal to $\cpt$, since $\cyc$ turns right at $\cpt$. Let $\cyc^\parr r$ be obtained by replacing the arc of $\cyc$ connecting $u$ to $u'$ with the corresponding part of $\str_r$. 

Consider also the strand $\str_t$ passing through the unique incoming edge of $u$ not used by $\cyc$. We let $\str_t^{u-}$ denote the part of $\str_t$ before $u$, and let $u''\neq u$ be the last vertex belonging to both $\str_t^{u-}$ and $\cyc$. Replacing the arc of $\cyc$ connecting $u''$ to $u$ with the corresponding part of $\str_t^{u-}$, we get another cycle $\cyc^\parr t$. It is clear that $x$ must belong to either $\cyc^\parr r$ or $\cyc^\parr t$ (or both), since the strands $\str_r$ and $\str_t$ intersect only at $u$. Thus we have found a cycle satisfying the above conditions that bounds a region of area smaller than $R$. This is a contradiction, and thus $\cyc$ is clockwise convex. It is then easy to check that a clockwise (or counterclockwise) convex cycle in $\Gtop_D$ yields an increasing cycle in $\GXV_f$ in the sense of \cref{dfn:convex_GXV}.
\end{proof}

\subsection{Proof of Proposition~\ref{prop:edges_conn_cpts}}\label{sec:proof-of_edge_conn_cpts}
Let $e$ be a non-boundary edge of $G$ labeled by $\{p,q\}$ with $1\leq p<q\leq n$. First, observe that if \eqref{eq:edges_conn_cpts} holds for any $f$-admissible tuple $\bth$ then $p$ and $q$ belong to the same connected component of $\GX_f$. Indeed, suppose otherwise that they belong to different connected components. The notion of $f$-admissibility is invariant under adding a constant to all $\th_{q'}$ for $q'$ in the same connected component as $q$. We can choose the constant so that $\th_p$ becomes equal to $\th_q$, contradicting~\eqref{eq:edges_conn_cpts}. It is also clear that~\eqref{eq:edges_conn_cpts} implies $\wt_\bth(e)>0$. It thus remains to prove~\eqref{eq:edges_conn_cpts} for all $f$-admissible tuples $\bth$. 

Recall that $G$ is assumed to be contracted. More generally, each (not necessarily contracted) graph $G'$ may be transformed into a contracted graph $G$ by removing interior degree $2$ vertices. Clearly, showing~\eqref{eq:edges_conn_cpts} for some graph $G'$ whose contracted version is $G$ implies that~\eqref{eq:edges_conn_cpts} also holds for $G$. In addition, observe that if two graphs $G$ and $G'$ are connected by a square move then~\eqref{eq:edges_conn_cpts} holds for $G$ if and only if it holds for $G'$. To see that, notice that whenever one performs a square move (\cref{fig:sq-mv}) on a contracted reduced graph $G$, this graph $G$ contains four edges labeled by $\{a,b\}$, $\{b,c\}$, $\{c,d\}$, and $\{a,d\}$ for $1\leq a<b<c<d\leq n$. Assuming~\eqref{eq:edges_conn_cpts} holds for $G$, we must have $\th_a<\th_b<\th_c<\th_d<\th_a+\pi$ for all $f$-admissible $\bth$. It is then easy to see that after performing the square move,~\eqref{eq:edges_conn_cpts} still holds for all non-boundary edges of $G'$. Since all (contracted) graphs $G\in\Gred(f)$ are related by square moves, it suffices to prove the statement for just one of them.

We shall proceed by induction on $f$ using the bridge removal procedure from \cref{sec:bridge_rem}. For the base case, observe that when $k=1$ and $f$ is loopless, all edges of $G$ are adjacent to the boundary and therefore have weight $1$. Let us now assume that $f\in\Bkn$ is loopless and $k>1$. Choose $r\in[n]$ such that $f$ has a bridge at $r$. Thus we have $r<r+1\leq f(r)<f(r+1)\leq r+n$. Let $G\in\Gred(f)$ be such that $G$ has a bridge configuration at $b_r,b_{r+1}$ as in \figref{fig:bridge}(left). 

Let us first consider the case $r+1=f(r)$. Then the interior black vertex in \figref{fig:bridge}(left) has degree $2$ and therefore $G$ is not contracted. Contracting the two edges incident to that vertex, we see that both $b_r$ and $b_{r+1}$ become connected to the same interior white vertex. Therefore $r+1$ does not appear in the label of any non-boundary edge of $G$. We may therefore remove the vertex $b_{r+1}$ from $G$ and deduce the result by induction.

\begin{figure}
  \includegraphics{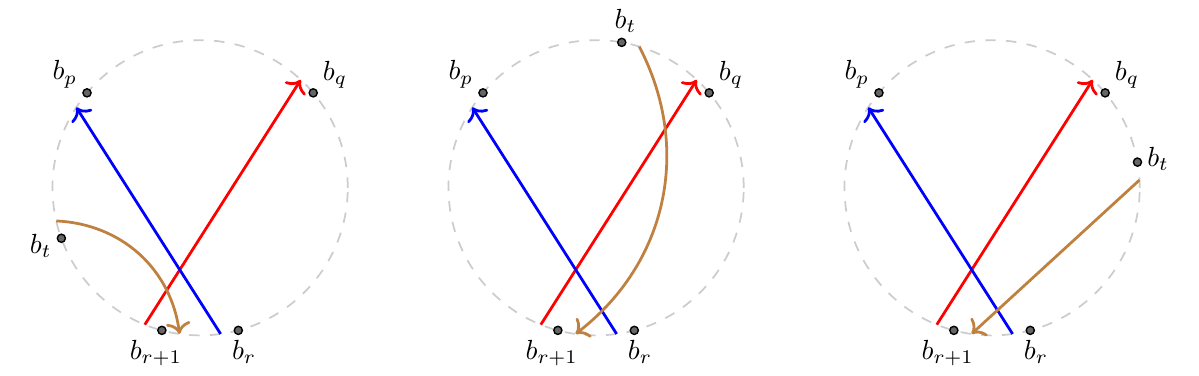}
  \caption{\label{fig:strand-bridge} The strand $S_{r+1}$ has to cross either $S_p$ or $S_q$.}
\end{figure}

Assume now that $r+1<f(r)$, thus $r<r+1 <f(r)<f(r+1)\leq r+n$. By definition, $p:=f(r)$ and $q:=f(r+1)$ form an affine $f$-crossing (cf. \cref{sec:cyclic-shift}), and thus by~\eqref{eq:affine_f_crossing}, we have $\tht_p<\tht_q<\tht_{p+n}$ for any $f$-admissible sequence $\btht$. Let $t:=f^{-1}(r+1)$. Since $p$ and $q$ form an affine $f$-crossing, the strands $\str_q=(\p{(r+1)}\to\m q)$ and $\str_p=(\p r\to\m p)$ form an $f$-crossing. The strand $\str_{r+1}=(\p t\to\m{(r+1)})$ therefore must cross either one or both of these strands; cf. \cref{fig:strand-bridge}.  Assume for example that it crosses $\str_q$ as in \figref{fig:strand-bridge}(left).  Choose a reduced strand diagram $D$ of $f$ such that both crossings are closer to $\p{(r+1)}$ than all other crossings. Denote the crossing point of the strands $\str_q$ and $\str_{r+1}$ by $\cpt$ and let $\cyc$ be a convex cycle from the proof of \cref{prop:convex} that passes through $\cpt$. We see that $\cyc$ must turn left at $\cpt$ and therefore it is a counterclockwise convex cycle. Moreover, denoting by $u$ the crossing point of $\str_q$ and $\str_p$, the construction in the proof of \cref{prop:convex} implies that $\cyc$ must turn left at $u$ as well. We have therefore found an increasing cycle passing through the edges $p\to q\to r+1$ in $\GXV_D$, where the indices again are taken modulo $n$. Label the vertices of this cycle by $(a_1\to a_2\to \cdots \to a_m\to a_1)$ as in \cref{dfn:convex_GXV}. Then we see that $\th_{a_1}<\th_{a_2}<\cdots<\th_{a_m}<\th_{a_1}+\pi$ for any $f$-admissible $\bth$. Since $r+1<p<q<r+1+n$, it follows that $\tht_{r+1}<\tht_p<\tht_q<\tht_{r+1}+\pi$.

Let $G'$ be the graph obtained from $G$ by removing the bridge at $p$. Thus $G'$ is reduced and has strand permutation $s_rf$. Even though $G'$ may not be contracted, the only edges that may need to get contracted are labeled by $\{p,r+1\}$. Let $\bth$ be an $f$-admissible tuple. Since $\tht_{r+1}<\tht_p<\tht_q<\tht_{r+1}+\pi$, it follows that $\bth$ is also $s_rf$-admissible. By the induction hypothesis, we may assume that~\eqref{eq:edges_conn_cpts} holds for all non-boundary edges of (the contracted version of) $G'$. Again using $\tht_{r+1}<\tht_p<\tht_q<\tht_{r+1}+\pi$, we see that~\eqref{eq:edges_conn_cpts} holds for all non-boundary edges of $G$. This completes the induction step. \qed

\subsection{Factorization}\label{sec:factorization}
Let $f\in\Bkn$ be loopless and consider the connected components $[n]=\Conc^\parr1\sqcup \Conc^\parr2\sqcup\dots\sqcup \Conc^\parr{\conn_f}$ of $\GX_f$. Our goal is to define ``restrictions'' of $f$ to each connected component and argue that $\Measop_f$ ``factors'' as an independent product of the boundary measurement maps for the restrictions. This is not completely straightforward: for $p\in[n]$, it may happen that the vertices $\m p$ and $\p p$ belong to different connected components of the strand diagram $D_f$, thus the vertex $b_p$ of the corresponding reduced graph $G$ appears to belong to both components simultaneously.

\begin{figure}
  \includegraphics{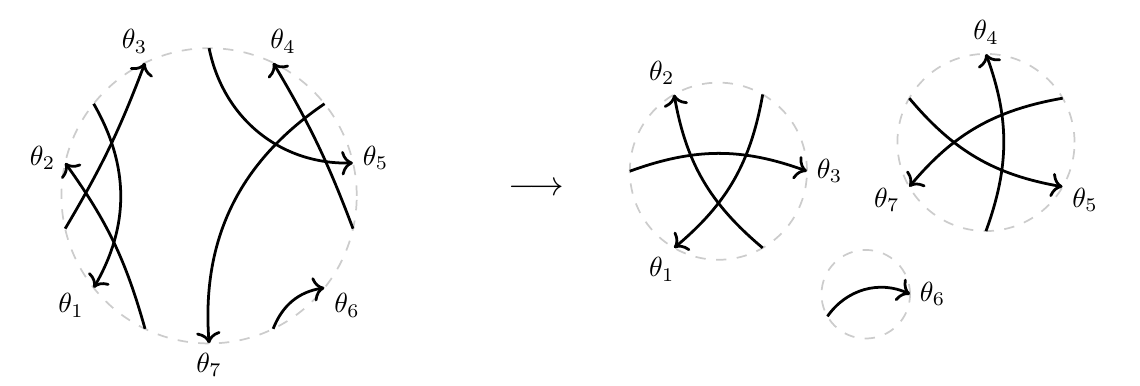}
  \caption{\label{fig:strand-factor} Splitting a strand diagram into connected components.}
\end{figure}

Consider a connected component $\Conc\subset [n]$ of $\GX_f$ and consider the strands $\{\str_p\mid p\in \Conc\}$ of $D_f$ that belong to $\Conc$. Let $D_f |_\Conc$ be obtained from $D_f$ by erasing all other strands. Then the directions of the strands of $D_f|_\Conc$ still alternate around the boundary of the disk. Thus after relabeling the strands and boundary vertices by integers in $[n']$ where $n':=|\Conc|$, $D_f |_\Conc$ becomes a strand diagram in the sense of \cref{dfn:strand_diag}. We denote by $f|_\Conc\in\Bound(k',n')$ the corresponding loopless bounded affine permutation. Finally, given any $f$-admissible tuple $\bth$, the restriction $\bth|_\Conc$ is defined in an obvious way: for $p\in \Conc$, if the strand $\str_p$ in $D_f$ is labeled as $\str_{p'}$ in $D_f|_\Conc$ for some $p'\in[n']$ then we set $(\bth|_\Conc)_{p'}:=\th_p$. In other words, $\th_p$ is viewed as a real parameter attached to the \emph{endpoint} of the strand $\str_p$. See \cref{fig:strand-factor} for an example. 

For $r\in[\conn_f]$, we denote by $f^\parr r:=f|_{\Conc^\parr r}$ the restriction of $f$ to the connected component $\Conc^\parr r$ of $\GX_f$. We are ready to state our factorization result.

\begin{proposition}\label{prop:factor}
We have a homeomorphism
\begin{equation*}%
  \Ctp_f\xrasim \Ctp_{f^\parr1}\times\Ctp_{f^\parr2}\times\cdots \times \Ctp_{f^\parr{\conn_f}}.
\end{equation*}
\end{proposition}
\begin{proof}

Let $\npm:=\{\p1,\m1,\p2,\m2,\dots,\p n,\m n\}$. For $\Conc\subset[n]$, let $\Cpm\subset\npm$ denote the set containing $\p s$ and $\m p$ for each $p\in\Conc$, where $s:=\perm^{-1}(p)$. We have a non-crossing partition $\npm=\Conc^\parr1_\pme\sqcup \Conc^\parr2_\pme\sqcup \dots \sqcup \Conc^\parr{\conn_f}_\pme$ into parts of even sizes. 

Consider a contracted graph $G\in\Gred(f)$. If $G$ is disconnected then the statement follows by considering each connected component independently, so let us assume that $G$ is connected. This implies that there exists $p\in[n]$ such that $\p p$ and $\m p$ belong to different parts of the above non-crossing partition. Call these parts $\Apm$ and $\Bpm$ so that $\p p\in \Apm$ and $\m p\in\Bpm$, and let $A,B\subset[n]$ be the corresponding connected components of $\GX_f$.

Recall from \cref{prop:edges_conn_cpts} that every interior edge of $G$ is labeled by $\{p',q'\}$ where $p',q'$ belong to the same connected component of $\GX_f$. Let $v$ be the white interior vertex connected to $b_p$. Then the (boundary) edge $\{v,b_p\}\in E(G)$ is labeled by $\{p,q'\}$ where $q'=f(p)\in A$ while $p\in B$. Label the edges incident to $v$ by $e_1,e_2,\dots,e_d$ in clockwise order, starting with $e_1:=\{v,b_p\}$. The strand labeled $q'$ passes through $e_1$ and $e_2$, thus either $e_2$ is a boundary edge or some strand labeled by $q_2\in A$ (cf. \cref{prop:edges_conn_cpts}) passes through $e_2$ and $e_3$, etc. Since the strand passing through $e_d$ and $e_1$ is labeled by $p\in B$, at some point we must encounter a boundary edge $e_j=\{v,b_q\}$ for some $q\in A$ such that $f(q)\notin A$. See \cref{fig:fact-AB}.

\begin{figure}
  \includegraphics{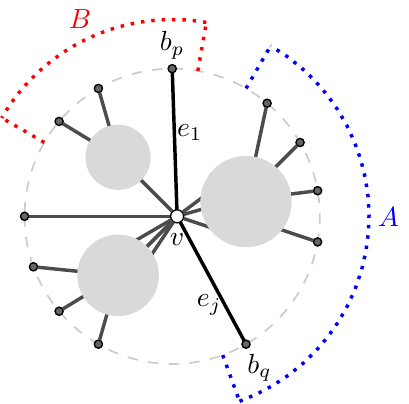}
  \caption{\label{fig:fact-AB} The graph $G$ from the proof of \cref{prop:factor}.}
\end{figure}

Observe that each of the cyclic intervals $P:=\{q+1,\dots,p\}$ and $Q:=\{p+1,\dots,q\}$ (indices taken modulo $n$) is a union of connected components of $\GX_f$. Let $g:=f|_P$ and $h:=f|_Q$ be the corresponding restrictions of $f$. We would like to establish a homeomorphism $\Ctp_f\xrasim \Ctp_g\times \Ctp_h$.

Let $G_P$ be the connected component of $b_p$ in the graph obtained from $G$ by removing the edges $e_{2},\dots,e_j$ defined above. Similarly, let $G_Q$ be the connected component of $b_q$ in the graph obtained from $G$ by removing the edges $e_{j+1},\dots,e_{1}$. Abusing notation, we preserve the original boundary labeling of $G_P$ and $G_Q$ by $(b_{q+1},\dots,b_{p})$ and $(b_{p+1},\dots,b_{q})$, respectively.

Let $\bth\in\R^n$. Clearly $\bth$ is $f$-admissible if and only if the restrictions  $\bth|_P$ and $\bth|_Q$ are $g$-admissible and $h$-admissible, respectively. Given such $\bth$, let $X:=\Meas(f,\bth)$, $X_P:=\Meas(g,\wt_{\bth|_P})$ and $X_Q:=\Meas(h,\wt_{\bth|_Q})$.  Our goal is to understand the relationship between $X$ and $(X_P,X_Q)$.

Let $k_P,n_P,k_Q,n_Q$ be such that $g\in \Bound(k_P,n_P)$ and $h\in\Bound(k_Q,n_Q)$, thus $k_P+k_Q=k+1$. Any almost perfect matching of $G$ contains a unique edge incident to $v$. By considering the possible options for this edge, we arrive at the following formulas. For two sets $L\subset P\setminus\{p\}$ and $R\subset Q\setminus \{q\}$, we have
\begin{align}
\label{eq:Delta_LR1}
  \Delta_{L\cup R\cup p}(X)=\Delta_{L\cup R\cup q}(X)&=\Delta_{L\cup p}(X_P)\cdot \Delta_{R\cup q}(X_Q), &&\text{if $|L|=k_P-1$ and $|R|=k_Q-1$};\\
\label{eq:Delta_LR2}
  \Delta_{L\cup R}(X)&=\Delta_{L}(X_P)\cdot \Delta_{R\cup q}(X_Q), &&\text{if $|L|=k_P$ and $|R|=k_Q-1$};\\
\label{eq:Delta_LR3}
  \Delta_{L\cup R}(X)&=\Delta_{L\cup p}(X_P)\cdot \Delta_{R}(X_Q), &&\text{if $|L|=k_P-1$ and $|R|=k_Q$}.
\end{align}
Here we abbreviate $R\cup p:=R\cup\{p\}$, etc. Each nonzero minor of $X$, $X_P$, and $X_Q$ appears in these formulas. We claim that $X$ and $(X_P,X_Q)$ determine each other uniquely via~\eqref{eq:Delta_LR1}--\eqref{eq:Delta_LR3}. First, clearly knowing the minors of $(X_P,X_Q)$ allows one to reconstruct the minors of $X$. Conversely, suppose that the minors of $X$ are known and we need to recover the minors of, say, $X_P$. Since the minors are defined up to multiplication by a common scalar, we only need to find the ratios $\Delta_M(X_P)/\Delta_N(X_P)$. If $p\notin M,N$ then by~\eqref{eq:Delta_LR2}, we get $\Delta_M(X_P)/\Delta_N(X_P)=\Delta_{M\cup R}(X_P)/\Delta_{N\cup R}(X_P)$ for any set $R$ such that $\Delta_{R\cup q}(X_Q)\neq 0$. (Such a set $R$ exists since otherwise $p$ and $q$ must be loops of $f$.) If $M=L\cup p$ but $p\notin N$ then by~\eqref{eq:Delta_LR1}--\eqref{eq:Delta_LR2}, $\Delta_{L\cup p}(X_P)/\Delta_{N}(X_P)=\Delta_{L\cup R\cup p}(X)/\Delta_{N\cup R}(X)$ for any $R$ such that $\Delta_{R\cup q}(X_Q)\neq 0$. The case $p\in M$ and $p\notin N$ is handled similarly. Finally, in the case $M=L\cup p$ and $N=L'\cup p$, by~\eqref{eq:Delta_LR1} and~\eqref{eq:Delta_LR3} we have $\Delta_{L\cup p}(X_P)/\Delta_{L'\cup p}(X_P)=\Delta_{L\cup R}(X)/\Delta_{L'\cup R}(X)$ for any $R\subset Q$ such that $\Delta_R(X_Q)\neq0$. We have established the desired homeomorphism $\Ctp_f\xrasim \Ctp_g\times \Ctp_h$. The result follows by induction.
\end{proof}

We finish by clarifying the relationship\footnote{We thank Lauren Williams for comments motivating the below results.} between our notion of connectedness for reduced strand diagrams and the standard notion of a \emph{connected positroid}~\cite{OPS,ARW2}. Given $f\in\Bkn$, let $\GH_f$ be the undirected graph with vertex set $[n]$ and edge set  consisting of all pairs $\{p,q\}\in{[n]\choose 2}$ such that the line segments $[b_s, b_p]$ and $[b_t, b_q]$ (where $\perm(s)=p$ and $\perm(t)=q$) have nonempty intersection. Then the \emph{positroid} $\Mcal_f$ is connected if and only if $\GH_f$ is connected; see~\cite[Corollary~7.9]{ARW2}. The connected components of $\GH_f$ form a non-crossing partition of $[n]$ denoted $\Pi(\GH_f)$. Similarly, if $f$ is loopless, we denote by $\Pi(\GX_f)$ the non-crossing partition of $[n]$ into the connected components of $\GX_f$.

\begin{proposition}\label{prop:connected:refine}
Let $f\in\Bkn$ be loopless. 
\begin{theoremlist}
\item\label{connected:refine1} The reduced strand diagram of $f$ is connected if and only if the positroids $\Mcal_f$ and $\Mcal_{\dsh f}$ are both connected.
\item\label{connected:refine2} The non-crossing partition $\Pi(\GX_f)$ is the common refinement of $\Pi(\GH_f)$ and $\Pi(\GH_{\dsh f})$.
\end{theoremlist}
\end{proposition}
\begin{proof}
Let $p,q\in [n]$ and denote $s:=\perm^{-1}(p)$, $t:=\perm^{-1}(q)$. Observe that $p,q$ belong to different connected components of $\GX_f$ if and only if there exists a chord $\alpha\to \beta$ (for two points $\alpha,\beta$ on the circle not equal to any of $b_j, \p j, \m j$ for $j\in[n]$) which separates $\p p$ and $\p q$ and does not intersect any of the strands in the reduced strand diagram of $\GX_f$. 

As in the proof of \cref{prop:factor}, we denote $\npm:=\{\p1,\m1,\p2,\m2,\dots,\p n,\m n\}$. The chord $\alpha\to\beta$ separates $\npm$ into two subsets of even size. Let $x\in\npm$ be the closest point to $\alpha$ in the clockwise direction. If $x=\m j$ for some $j\in[n]$ then the chord $\alpha\to\beta$ does not intersect any line segment $[b_r, b_{\perm(r)}]$ for $r\in[n]$, and thus $p,q$ belong to different connected components of $\GH_f$. 

Assume now that $x=\p j$ for some $j\in[n]$. Consider the dual reduced strand diagram of $\dsh f$. It coincides with the reduced strand diagram of $f$ up to a simple relabeling of boundary vertices; see \cref{sec:shift_commut_diag} for further details. In particular, it follows that $p,q$ belong to different connected components of $\GH_{\dsh f}$. 

Conversely, it is clear that if $p,q$ belong to different connected components of either $\GH_f$ or $\GH_{\dsh f}$ then they belong to different connected components of $\GX_f$.
\end{proof}

\section{Critical varieties}\label{sec:critical-varieties}
In this section, we study Zariski closures of critical cells as discussed in \cref{sec:intro:critical-varieties}. 
In particular, we show that the boundary measurement map $\Measop_f$ is well defined and gives rise to \emph{open critical varieties} $\Cio_f$ whose definition depends on the \emph{Laurent phenomenon} (\cref{thm:Laurent_ph}). We study the real points of $\Cio_f$ in \cref{sec:real-part-critical}. An important tool we rely on is the \emph{twist map} of~\cite{MuSp}, reviewed in \cref{sec:twist}. Throughout, we assume that $f\in\Bkn$ is loopless.

\subsection{Open critical varieties}\label{sec:open-crit-vari}
To a sequence $\btht:\Z\to\R$ we associate a sequence $\btt:\Z\to\Cast$ of complex numbers defined by
\begin{equation}\label{eq:th_to_t}
\tt_q:=\exp(i\tht_q) \quad\text{for all $q\in\Z$.} 
\end{equation}
If $\btht$ satisfies~\eqref{eq:bth_extend} then we have $\tt_{q+n}=-\tt_q$ for all $q\in\Z$. Similarly to~\eqref{eq:bth_extend}, we identify sequences $\btt:\Z\to\Cast$ satisfying this condition with tuples $\bt\in(\Cast)^n$.

Recall from \cref{dfn:intro:f_adm_C} that $\bt=(t_1,t_2,\dots,t_n)\in(\Cast)^n$ is \emph{$f$-admissible} if $t_p\neq \pm t_q$ for all pairs $p,q$ that from an $f$-crossing.

As in \cref{sec:connected-components}, we  choose representatives $p_1,p_2,\dots,p_{\conn_f}\in[n]$, one from each connected component of $\GX_f$, and set
\begin{equation}\label{dfn:Tspace}
  \Tspace_f:=\{\bt=(t_1,t_2,\dots,t_n)\in (\Cast)^{n}\mid \bt\text{ is $f$-admissible and $t_{p_1}=t_{p_2}=\cdots=t_{p_{\conn_f}}=1$}\}.
\end{equation}

Recall also that to each graph $G\in\Gred(f)$ and each $\bt\in(\Cast)^n$ we assign a weight function $\wt_\bt:E(G)\to\C$ given by~\eqref{eq:wt_T}. As \cref{ex:t1=t3_t2=t4} demonstrates, the complex algebraic analog of \cref{prop:edges_conn_cpts} no longer holds, thus we have to be more careful in defining $\Measop_f$.

We denote $\Meas(G,\wt_\bt):=(\Delta_I(G,\wt_\bt))_{I\in{[n]\choose k}}$. If the entries of $\Meas(G,\wt_\bt)$ are not all zero then we say that $\Measop_G$ \emph{is well defined at $\bt$} and view the result as an element of the complex Grassmannian $\Gr(k,n)$. The following result is proved in \cref{sec:Laurent_proofs}.
\begin{proposition}\label{prop:Meas_exists}
 There exists a unique regular map $\Measop_f:\Tspace_f\to \Pio_f$ satisfying 
\begin{equation*}%
  \Meas(f,\bt)=\Meas(G,\wt_\bt) \quad\text{inside $\Gr(k,n)$}
\end{equation*}
 for all $G\in\Gred(f)$ and all $\bt\in\Tspace_f$ such that $\Measop_G$ is well defined at $\bt$.
\end{proposition}

By \cref{dfn:intro:Cio}, the \emph{open critical variety} $\Cio_f\subset \Pio_f$ is the image of the map $\Measop_f$:%
\begin{equation*}%
  \Cio_f:=\Meas(f,\Tspace_f).
\end{equation*}
Even though we use the terms \emph{open} and \emph{variety}, it is an open problem to describe $\Cio_f$ as an open subvariety of $\Crit_f$, even in the case $f=\fkn$.
\begin{problem}\ \label{prob:Cio}
\begin{enumerate}
\item Show that the variety $\CCrit_f$ is irreducible of dimension $\dimf$ and describe $\CCrit_f$ by polynomial equations.
\item Show that $\Cio_f$ is an open subvariety of $\CCrit_f$ and describe its complement by polynomial equations.
\item Describe $\Ctp_f\subset\Ptp_f$ by polynomial equations and inequalities. 
\end{enumerate}
\end{problem}
\noindent See~\eqref{eq:Pio_open}--\eqref{eq:Ptp_minors_descr} for the analogous results for positroid varieties. We caution that in general, $\Ctp_f\subsetneq \Cio_f\cap \Grtnn(k,n)$, unlike in the case of positroid varieties. For example, $\Ctp_{2,3}$ consists of all points of $\Grtp(2,3)$ whose Pl\"ucker coordinates satisfy the triangle inequalities.

\begin{remark}\label{rmk:dual_cio}
Fix a coloopless $f\in\Bkn$. We say that $\bt\in(\Cast)^n$ is \emph{dual $f$-admissible} if $t_p\neq \pm t_q$ whenever $p,q$ form a dual $f$-crossing. Similarly to what we did above, one can introduce the \emph{dual boundary measurement map} $\Measdop:\Tspaced_f\to\Ciod_f$ whose image is the \emph{dual open critical variety} $\Ciod_f$, and whose restrictions to appropriate open subsets of $\Tspaced_f$ coincide with $\Measdop_G$ for  $G\in\Gred(f)$.
\end{remark}
\subsection{The Laurent phenomenon}

\begin{figure}
  \includegraphics{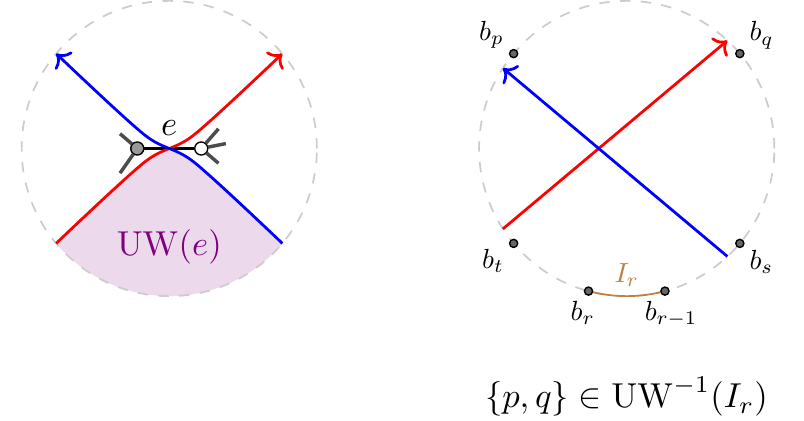}
  \caption{\label{fig:upstream} The upstream wedge.}
\end{figure}

Let $\Ical_f=(I_1,I_2,\dots,I_n)$ be the Grassmann necklace of $f$ defined in \cref{sec:positroid-varieties}. For $r\in[n]$, we label by $I_r$ the arc connecting $\p{(r-1)}$ to $\m r$. Choose $1\leq p<q\leq n$ such that the arrows $\p s\to \m p$ and $\p t\to \m q$ form an $f$-crossing. Following the terminology of~\cite{MuSp}, we say that $I_r$ belongs to the \emph{upstream wedge of $\{p,q\}$} if the arc labeled $I_r$ is contained in the arc connecting $\p s$ to $\p t$ that does not contain $\m p, \m q$; see \figref{fig:upstream}(right). We let $\UW^{-1}(I_r)$ denote the set of all pairs $\{p,q\}$ such that $I_r$ belongs to the upstream wedge of $\{p,q\}$.

Our first goal is to give a product formula for the boundary measurements associated with the Grassmann necklace, which is a simple consequence of the results of Muller--Speyer~\cite{MuSp}. 
\begin{proposition}\label{prop:Gr_neck_formula}
Let $G\in\Gred(f)$ and suppose that $\Measop_G$ is well defined at $\bt\in \Tspace_f$. Then, after a multiplication by a common scalar, we have
\begin{equation}\label{eq:Gr_neck_formula}
  \Delta_{I_r}(G,\wt_\bt)=\prod_{p<q:\ \{p,q\}\in\UW^{-1}(I_r)} \H qp \quad\text{for all $r\in[n]$.}
\end{equation}
\end{proposition}
See \cref{sec:Laurent_proofs} for a proof. Observe that the right hand side of~\eqref{eq:Gr_neck_formula} is a Laurent polynomial in $\bt$ (homogeneous of degree $0$) that does not depend on the choice of $G$. Moreover, all $\Delta_{I_r}(G,\wt_\bt)$ are nonzero precisely when $\bt$ is $f$-admissible.

\begin{example}
Consider the case $f=f_{2,4}$ from \cref{fig:plabic}. The Grassmann necklace is 
\begin{equation*}%
  I_1=\{1,2\},\quad   I_2=\{2,3\},\quad   I_3=\{3,4\},\quad   I_4=\{1,4\}.
\end{equation*}
From the reduced strand diagram on the right in \figref{fig:move}(b), we find $\UW^{-1}(I_1)=\{\{2,3\}\}$,  $\UW^{-1}(I_2)=\{\{3,4\}\}$,    $\UW^{-1}(I_3)=\{\{1,4\}\}$, and    $\UW^{-1}(I_4)=\{\{1,2\}\}$. 
 Thus~\eqref{eq:Gr_neck_formula} agrees with the values computed in \figref{fig:plabic}(d), after canceling out the term $\brx[24]$.
\end{example}

According to \cref{prop:Gr_neck_formula}, the entries of $\Measop_G$ may be rescaled so that~\eqref{eq:Gr_neck_formula} holds. We refer to this rescaling as the \emph{canonical gauge-fix} of $\Measop_G$.

\begin{theorem}[Laurent phenomenon]\label{thm:Laurent_ph}
Let $G\in\Gred(f)$ and suppose that $\Measop_G$ is well defined at $\bt\in \Tspace_f$. Then the entries $\Delta_I(G,\wt_\bt)$ of the canonical gauge-fix of $\Measop_G$ are Laurent polynomials in $\bt$.
\end{theorem}
\noindent See \cref{sec:Laurent_proofs} for a proof. By~\eqref{eq:Pio_open}, we see that the map $\Measop_G$ may be extended to a map
\begin{equation*}%
  \Measop_f:\Tspace_f\to \Pio_f.
\end{equation*}
Moreover, $\Meas(G,\wt_\bt)\notin \Pio_f$ when $\bt\in\Cast$ is not $f$-admissible. This again confirms that reduced strand diagrams give the ``correct'' notion of $f$-admissibility, even though a priori it may appear that a more restrictive notion is required (cf. \cref{ex:t1=t3_t2=t4}).

\begin{example}\label{ex:non_inj}
Unlike in the case of open positroid varieties, the map $\Measop_f:\Tspace_f\to\Cio_f$ is in general \emph{not} injective. For example, consider two tuples 
\begin{equation*}%
  \bt:=(1,\exp(i\pi/4),\exp(i\pi/2),\exp(3i\pi/4)), \quad%
\bt':=(1,\exp(3i\pi/4),\exp(-i\pi/2),\exp(i\pi/4)).
\end{equation*}
Then for $f=f_{3,4}$ (cf. \figref{fig:crit-ex}(c)), we find that $\bt,\bt'\in\Tspace_f$ and $\Meas(f,\bt)=\Meas(f,\bt')$, since both $\Meas(f,\bt)$ and $\Meas(f,\bt')$ give rise to the cyclically symmetric point $\XO 34\in \Grtnn(3,4)$ (all of whose Pl\"ucker coordinates are equal). Note that if $\bth$ and $\bth'$ are related respectively to $\bt$ and $\bt'$ via~\eqref{eq:th_to_t} then $\bth\in \THtp_f$ but $\bth'\notin \THtp_f$, so this example does not contradict \cref{thm:inj_fkn}.
\end{example}

Limited computational evidence suggests that even a stronger form of the Laurent phenomenon may hold for critical varieties. It was shown in~\cite{posit_cluster} that the coordinate ring $\C[\Pio_f]$ admits a cluster algebra structure, and thus we have a family of regular functions on $\Pio_f$ called \emph{cluster variables}.
\begin{conjecture}[Strong Laurent phenomenon]\label{conj:strong_Laurent}
All cluster variables in $\C[\Pio_f]$, when restricted to $\Cio_f$, become Laurent polynomials in $\bt$, assuming the canonical gauge-fix~\eqref{eq:Gr_neck_formula}.
\end{conjecture}

\subsection{The real part of a critical variety}\label{sec:real-part-critical}
Recall that the set $\Gr_\R(k,n)$ of real points of the complex Grassmannian consists of all $X\in \Gr(k,n)$ such that the ratio of any two nonzero Pl\"ucker coordinates belongs to $\R$. The problem of determining the set $\Cio_f(\R)$ of real points of $\Cio_f$ turns out to be quite non-trivial, and we solve it only partially even in the case of the top cell. Recall from \cref{sec:shift-1} that $\dsh f\in\Bound(k-1,n)$ is defined by $\dsh f(p):=f(p-1)$ for all $p\in\Z$.
 We let $i\R\subset\C$ be the set of purely imaginary complex numbers. Denote
\begin{align*}
    \TspaceR_f=\{\bt\in\Tspace_f\mid |t_p|=1&\text{ for all $p\in[n]$}\}\\
  &\cup\{\bt\in\Tspace_f\mid t_p\in\R\cup i\R \text{ and ${t_p}/{t_{\dsh\perm(p)}}\in\R$ for all $p\in[n]$}\}.
\end{align*}
Thus if $\bt\in\TspaceR_f$ then the points $\v_p:=t_p^2$, $p\in[n]$, all belong to the same \emph{generalized circle}, that is, either a circle or a line. (Moreover, the circle is required to have its center at $0$ while the line is required to pass through $0$.)

\begin{lemma}\label{lemma:GX_conn}
Assume that $\GX_f$ is connected. Then for $\bt\in\Tspace_f$, we have 
\begin{equation}\label{eq:real_=>}
\bt\in\TspaceR_f \quad\Longrightarrow\quad   \Meas(f,\bt)\in \CioR_f.
\end{equation}
\end{lemma}
\begin{proof}
Assume first that $\bt$ is generic and let $G\in\Gred(f)$. 
 If $|t_p|=1$ for all $p\in[n]$ then $\wt_\bt$ is gauge-equivalent to $\wt_\bth$, where $\bth=(\th_1,\th_2,\dots,\th_n)\in\R^n$ is any tuple related to $\bt$ by~\eqref{eq:th_to_t}. Indeed, we have $\sin(\th_q-\th_p)=\frac1{2i}\H qp$, and thus $\wt_\bth$ is obtained from $\wt_\bt$ by rescaling all edges incident to each interior black vertex of $G$ by $\frac1{2i}$. If $t_p\in\R$ for all $p\in[n]$ then the edge weights $\wt_\bt(e)$ are already real numbers. In either case, we see that $\Meas(G,\wt_\bt)\in\Gr_\R(k,n)$. 

Let us now consider the case where $t_p\in\R\cup i\R$ for all $p\in[n]$. Since the entries of $\bt$ are nonzero, we have a map $\eps:[n]\to\{1,i\}$ such that $\eps(p)=1$ if $t_p\in\R$ and $\eps(p)=i$ if $t_p\in i\R$. For each edge $e\in E(G)$ labeled by $\{p,q\}$, we have $\wt_\bt(e)\in \R$ if $\eps(p)=\eps(q)$ and $\wt_\bt(e)\in i\R$ otherwise. 
 Consider an interior face $F$ of $G$. Each strand of $G$ passes through an even number of edges of $F$, and therefore $\wt_\bt(e)\in i\R$ for an even number of edges of $F$. Thus the alternating product~\eqref{eq:wt_alt_prod} is real for each interior face $F$. Suppose now that $F$ is a boundary face of $G$. Since $\GX_f$ was assumed to be connected, let $b_p$ and $b_{p+1}$ be the two boundary vertices of $G$ belonging to $F$. Again, we see that each strand of $G$ passes through an even number of edges of $F$ except for the two strands labeled by $f(p)$ and $p+1$. Therefore the alternating product~\eqref{eq:wt_alt_prod} is real if and only if $\eps(f(p))=\eps(p+1)$ for all $p\in[n]$, where the indices are taken modulo $n$. This is equivalent to having $\eps(\dsh f(p))=\eps(p)$ for all $p\in[n]$. Thus all alternating products~\eqref{eq:wt_alt_prod} are real, which implies that $\wt_\bt$ is gauge-equivalent to a real edge weight function.

We have shown the result for generic $\bt\in\TspaceR_f$. The general case follows by continuity.
\end{proof}

\noindent The converse to~\eqref{eq:real_=>} is false in general, even for the top cell. For instance, if $f=f_{2,4}$ then $\Meas(f,\bt)\in\CioR_f$  when $\bt\in\Tspace_f$ satisfies $t_1=t_3$ and $t_2=t_4$; see \cref{ex:t1=t3_t2=t4}. However, we expect the converse to hold in the \emph{generic} case.

\begin{conjecture}
Let $f\in\Bkn$ be loopless and assume that $\GX_f$ is connected. Then for all generic $\bt\in\Tspace_f$, we have 
\begin{equation*}%
\bt\in\TspaceR_f \quad\Longleftrightarrow\quad   \Meas(f,\bt)\in \CioR_f.
\end{equation*}
\end{conjecture}
\noindent In \cref{sec:proofs_top_cell}, we prove this in the following special case (cf. \cref{notn:fkn}).
\begin{theorem}\label{thm:real_fkn}
Let $2\leq k\leq n-2$ and $f=f_{k,n}$. Then for any generic $\bt\in(\Cast)^n$, we have
\begin{equation}\label{eq:real_fkn_thm}
\bt\in\TspaceRkn \quad\Longleftrightarrow\quad   \Measop_{k,n}(\bt)\in \CioknR.
\end{equation}
\end{theorem}
\noindent When $k=1$ or $k=n$, $\fkn$ is not connected (and the problem is trivial). When $k=n-1$, $\fkn$ is connected but our methods do not extend to this case.

\subsection{The twist map}\label{sec:twist}
We review the twist map introduced by Muller--Speyer~\cite{MuSp} generalizing the earlier results of Marsh--Scott~\cite{MaSc}. Our goal is to express the \emph{twisted minors} of $\Measop_f(\bt)$ in terms of $\bt$; see \cref{prop:twist}.

We will use the \emph{left twist} automorphism $\cev \tau:\Pio_f\xrasim \Pio_f$. Suppose that an element $A\in\Pio_f$ is the row span of a matrix with columns $A_1,A_2,\dots, A_n$. Extend this to a sequence $(A_q)_{q\in\Z}$ via~\eqref{eq:columns}. Then $\cev\tau(A)\in\Pio_f$ has columns $(\ctau(A)_q)_{q\in\Z}$ defined by
\begin{equation}\label{eq:left_twist_dfn}
  \<  \ctau(A)_q, A_q\>=1 \quad\text{and}\quad \<\ctau(A)_q,A_p\>=0 \quad\text{for all $p,q\in\Z$ such that $p<q<f(p)$.}
\end{equation}
Here $\<\cdot,\cdot\>$ denotes the standard inner product on $\C^k$. The set $\It'_q:=\{p\in\Z\mid p\leq q<f(p)\}$ is an element of the \emph{reverse Grassmann necklace} of $f$. Similarly to the set $\It_q$ defined in~\eqref{eq:It_q}, it has size $k$, and the corresponding columns $(A_p)_{p\in\It'_q}$ form a basis of $\C^k$. Since $f\in\Bkn$ is assumed to be loopless, we have $q\in \It'_q$ and the column $A_q$ is nonzero. Thus~\eqref{eq:left_twist_dfn} yields a well-defined vector $\ctau(A)_q$.

Recall that we have introduced a directed graph $\GXV_f$ in \cref{sec:RSD} (see \figref{fig:GX}(c)) with edge set
\begin{equation}\label{eq:E_GXV}
  E(\GXV_f)=\{(p,q)\in[n]\times[n]\mid \text{$p\neq q$ form a positive $f$-crossing}\}.
\end{equation}
Following~\cite{Pos}, we consider \emph{misalignments} of $f$. Let $p,q,s,t\in[n]$, be such that $p\neq q$, $s=\perm^{-1}(p)$, and $t=\perm^{-1}(q)$. We write $\{p,q\}\in \Misud_f$ if the points $\p s,\m p,\p t,\m q$ are ordered clockwise. Similarly, we write $\{p,q\}\in \Misdu_f$ if the points $\p s,\m p,\p t,\m q$ are ordered counterclockwise. Our definition is slightly different from that of~\cite{Pos} in that in our case the $2$-element sets $\{s,t\}$ and $\{p,q\}$ are not necessarily disjoint. See \figref{fig:align}(c,d).

For convenience, denote
\begin{equation*}%
  \Hun pq=\Hun qp:=\H qp \quad\text{for $1\leq p<q\leq n$.}
\end{equation*}
Note that we have $\H pq=-\H qp$ but $\Hun pq=\Hun qp$.

Recall from \cref{sec:plan-bipart-graphs} that the faces of a reduced graph $G$ are labeled by $k$-element sets. We are ready to give a formula for the corresponding Pl\"ucker coordinates of $\ctau(\Meas(f,\bt))$.
\begin{proposition}\label{prop:twist}
Let $G\in\Gred(f)$ and assume that $\Measop_G$ is well defined at $\bt\in\Tspace_f$. 
Then after a multiplication by a common scalar, we have
\let\BG\Big
\begin{equation}\label{eq:Delta_I_twist}
   \Delta_I(\cev\tau(\Meas(G,\wt_\bt)))= 
\BG(\prod\limits_{\substack{\{p,q\}\in\Misud_f:\\ p,q\in I}}\Hun qp\BG) \cdot 
\BG(\prod\limits_{\substack{\{p,q\}\in\Misdu_f:\\ p,q\notin I}}\Hun qp\BG) \cdot 
\BG( \prod\limits_{\substack{(p,q)\in E(\GXV_f):\\ p\in I,q\notin I}}\Hun qp\BG)^{-1}
\end{equation}
 for all face labels $I$ of $G$.
\end{proposition}
\noindent (Again, observe that the right hand side depends only on $f$ and $I$ but not on $G$.)
\begin{proof}
According to~\cite[Theorem~7.1]{MuSp}, $\Delta_I\circ\ctau\circ\Meas(G,\wt_\bt)$ is a monomial in the edge weights $\wt_\bt$. Specifically, it is the product of $\wt_\bt(e)^{-1}$ over all edges $e$ such that the face of $G$ labeled by $I$ belongs to the \emph{upstream wedge} of $e$; see~\cite[Figure~4]{MuSp} and~\figref{fig:upstream}(left). The set of such edges forms an almost perfect matching denoted $\Mleft(F)$, where $F$ is the face of $G$ labeled by $I$. Consider two strands terminating at $b_p$ and $b_q$. In general, they may intersect several times, giving rise to several interior edges $e$ labeled by $\{p,q\}$. The corresponding edge weights are all equal to $\wt_\bt(e)=\Hun qp$. Let $\ncrs_{p,q}(G)$ denote the number of edges labeled by $\{p,q\}$. Observe that $\ncrs_{p,q}(G)$ is odd if $p,q$ form an $f$-crossing and is even otherwise. Let  
\begin{equation}\label{eq:nc_pq_dfn}
  \nc_{p,q}(G):=
  \begin{cases}
    \frac12(\ncrs_{p,q}(G)-1), &\text{if $p,q$ form an $f$-crossing,}\\
    \frac12\ncrs_{p,q}(G), &\text{otherwise;}\\
  \end{cases}
\end{equation}
see \cref{fig:ypq}. Consider the product
\begin{equation}\label{eq:common_G}
  \common_G(\bt):=\prod_{1\leq p<q\leq n} (\H qp)^{\nc_{p,q}(G)}.
\end{equation}
It is then straightforward to see that the product of $\wt_\bt(e)^{-1}$ over all $e\in \Mleft(f)$ is equal to the right hand side of~\eqref{eq:Delta_I_twist} divided by $\common_G(\bt)$. 
\end{proof}

\begin{figure}
  \includegraphics{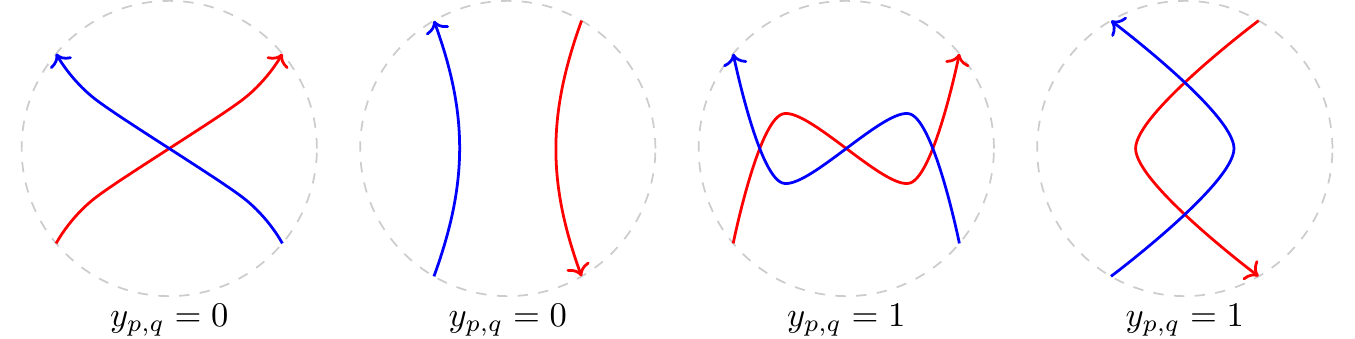}
  \caption{\label{fig:ypq} The integer $\nc_{p,q}(G)$ defined in~\eqref{eq:nc_pq_dfn}.}
\end{figure}

\subsection{Proofs}\label{sec:Laurent_proofs}
 First, we use the twist map to deduce the formula for Grassmann necklace minors.
\begin{proof}[Proof of \cref{prop:Gr_neck_formula}]
As explained in the proof of~\cite[Proposition~6.6]{MuSp}, we have 
\begin{equation*}%
  \Delta_{I_r}(A)=\frac1{\Delta_{I_r}(\ctau(A))} \quad\text{for all $r\in[n]$.}
\end{equation*}
The result follows from \cref{prop:twist}.
\end{proof}

Next, we focus on the Laurent phenomenon (\cref{thm:Laurent_ph}). For that, we will need several straightforward lemmas.

\begin{lemma}\label{lemma:bth_to_bt}
Let $\bth$ be an $f$-admissible tuple, and suppose that $\bth$ and $\bt$ are related by~\eqref{eq:th_to_t}. Then for any $G\in\Gred(f)$, $\bt$ is $f$-admissible, $\Measop_G$ is well defined at $\bt$, and $\Meas(G,\wt_\bth)=\Meas(G,\wt_\bt)$ inside $\Gr(k,n)$.
\end{lemma}
\begin{proof}
The $f$-admissibility claim is obvious. Next, observe that $\wt_\bth$ and $\wt_\bt$ are gauge-equivalent; cf. the proof of \cref{lemma:GX_conn}. This implies the remaining statements.
\end{proof}

It is known that when all edge weights $\wt_\bth(e)$ are positive reals, $\Meas(G,\wt_\bth)$ gives rise to a point in $\Ptp_f$. In particular, $\Delta_I(G,\wt_\bth)$ is zero for $I\notin \Mcal_f$ and positive for $I\in\Mcal_f$; see~\eqref{eq:Ptp_minors_descr}. In what follows, we treat $\bt=(t_1,t_2,\dots,t_n)$ as a collection of algebraically independent variables. To avoid confusion, we denote $Z_{I,G}(\wt_\bt):=\Delta_I(G,\wt_\bt)$ and treat  $\bZ_G:=(Z_{I,G}(\wt_\bt))_{I\in{[n]\choose k}}$ as a collection of rational functions in $\bt$ defined by~\eqref{eq:intro:Delta_I}. We do \emph{not} consider the entries of $\bZ_G$ modulo rescaling since we need to explicitly talk about the entries being Laurent polynomials.
\begin{corollary}\label{cor:Z_I_G_nonzero}
For a graph $G\in\Gred(f)$, $Z_{I,G}(\wt_\bt)$ is zero for $I\notin\Mcal_f$. For $I\in\Mcal_f$, $Z_{I,G}(\wt_\bt)$ is a nonzero Laurent polynomial in $\bt$.
\end{corollary}
\begin{proof}
The fact that $Z_{I,G}(\wt_\bt)$ is zero for $I\notin\Mcal_f$ follows since the edge weights $\wt_\bth$ are all positive, so $\Delta_I(G,\wt_\bth)=0$ implies that $G$ has no almost perfect matchings with boundary $I$. If $I\in\Mcal_f$ then $\Delta_I(G,\wt_\bth)>0$, and by \cref{lemma:bth_to_bt}, it is a specialization of $Z_{I,G}(\wt_\bt)$, hence $Z_{I,G}(\wt_\bt)\neq0$ as a rational function in $\bt$. Finally, $Z_{I,G}(\wt_\bt)$ is a Laurent polynomial in $\bt$ since it is a polynomial in the edge weights $\wt_\bt(e)$, each of which is a Laurent polynomial in $\bt$.
\end{proof}

Next we prove the complex algebraic analog of the square move invariance of $\Measop_f$.
\begin{lemma}\label{lemma:complex_sq_move}
Suppose that reduced graphs $G$ and $G'$ are related by a square move. 
Then $\bZ_G(\wt_\bt)$ and $\bZ_{G'}(\wt'_{\bt})$ agree up to multiplication by a common factor, where $\wt_\bt$ and $\wt'_\bt$ are defined by~\eqref{eq:wt_T} on the edges of $G$ and $G'$, respectively.
\end{lemma}
\begin{proof}
The result follows from the algebraic identity
\begin{equation}\label{eq:Pluck_t}
  \H 12 \cdot \H 34 +\H 14 \cdot \H 23=\H 13 \cdot \H 24 \quad\text{for all $t_1,t_2,t_3,t_4\in\Cast$.}\qedhere
\end{equation}
\end{proof}

\begin{proof}[Proof of \cref{prop:Meas_exists,thm:Laurent_ph}]
Let $G\in\Gred(f)$. By~\cite[Proposition~5.13]{MuSp}, for each $r\in[n]$, $G$ contains a unique almost perfect matching with boundary $I_r$, and this almost perfect matching coincides with $\Mleft(F_r)$ from the proof of \cref{prop:twist}. Here $F_r$ is the boundary face of $G$ labeled by $I_r$. In particular, $Z_{I,G}(\wt_\bt)$ is the product of weights of edges in $\Mleft(F_r)$. For $I\in{[n]\choose k}$, define
\begin{equation}\label{eq:Y_I_dfn}
  Y_I(\bt):=\frac{Z_{I,G}(\wt_\bt)}{\common_G(\bt)},
\end{equation}
where $\common_G(\bt)$ is defined in~\eqref{eq:common_G}. We see that $Y_{I_r}(\bt)$ is equal to the product on the right hand side of~\eqref{eq:Gr_neck_formula}, thus the tuple $(Y_I(\bt))_{I\in{[n]\choose k}}$ is canonically gauge-fixed. By \cref{cor:Z_I_G_nonzero}, $Y_I(\bt)$ is nonzero (as a rational function in $\bt$) precisely when $I\in\Mcal_f$. By \cref{lemma:complex_sq_move}, $Y_I(\bt)$ depends only on $f$ and $I$ and not on $G$.

Fix $I\in\Mcal_f$. We need to show that $Y_I(\bt)$ is a Laurent polynomial in $\bt$. The denominator $\common_G(\bt)$ on the right hand side of~\eqref{eq:Y_I_dfn} is a product of linear factors since $\H qp=\frac1{t_pt_q}(t_q-t_p)(t_q+t_p)$. Thus up to a monomial, the denominator of $Y_I(\bt)$ is the product of some linear factors that divides $\common_G(\bt)$ for any $G$.

\begin{figure}
  \includegraphics[width=1.0\textwidth]{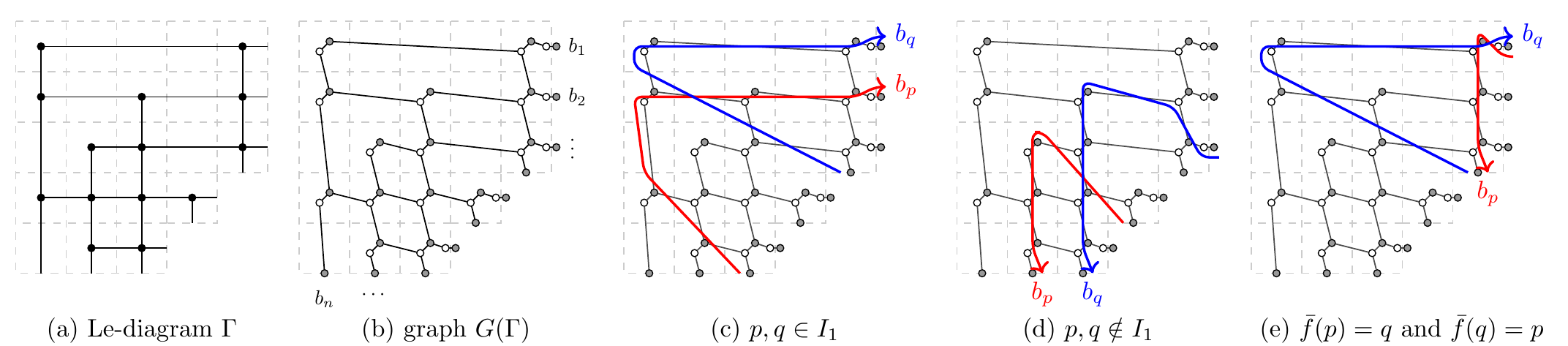}
  \caption{\label{fig:Le} A Le-diagram $\Gamma$ and the associated planar bipartite graph $G(\Gamma)$ from the proof of \cref{prop:Meas_exists,thm:Laurent_ph}.}
\end{figure}

 Fix $1\leq p<q\leq n$. 
 By~\eqref{eq:common_G}, it suffices to show that there exists a graph $G\in\Gred(f)$ such that the integer $\nc_{p,q}(G)$ defined in~\eqref{eq:nc_pq_dfn} is zero. The problem is trivial when either $p$ or $q$ is a coloop. Next, assume that we have either $\perm(p)\neq q$ or $\perm(q)\neq p$. Then one can check using the description~\eqref{eq:intro:J_r} of the Grassmann necklace of $f$ that there exists an index $r\in[n]$ such that either $p,q\in I_r$ or $p,q\notin I_r$. After cyclically shifting (as in \cref{sec:cyclic-shift}), we may assume that $r=1$. 
Let $\Gamma$ be the \emph{Le-diagram} of $f$ (see \figref{fig:Le}(a)); we refer to~\cite[Section~20]{Pos} for background on Le-diagrams. Let $G(\Gamma)$ be the corresponding planar bipartite graph (shown in \figref{fig:Le}(b)) and let $G$ be the contracted version of $G(\Gamma)$. It is clear that in both cases $p,q\in I_1$ and $p,q\notin I_1$, the strands terminating at $b_p$ and $b_q$ intersect at most once; see \figref{fig:Le}(c,d). Finally, consider the case $\perm(p)= q,\perm(q)= p$. After applying the cyclic symmetry, we may assume that $p=1$, and then taking $G$ to be the Le-diagram graph again, we see that the strands terminating at $b_p$ and $b_q$ do not intersect; see \figref{fig:Le}(e). (Recall that $G$ is contracted unlike $G(\Gamma)$ so the double crossing in \figref{fig:Le}(e) appears in $G(\Gamma)$ but not in $G$.) This implies that $\nc_{p,q}(G)=0$, and thus the denominator of $Y_I(\bt)$ is not divisible by $\H qp$ for any $1\leq p<q\leq n$, so $Y_I(\bt)$ is a Laurent polynomial in $\bt$.

Now we can finally introduce the map $\Measop_f:\Tspace_f\to\Pio_f$ in \cref{prop:Meas_exists} whose image is the open critical variety $\Cio_f$. 
 Namely, we set $\Meas(f,\bt):=(Y_I(\bt))_{I\in{[n]\choose k}}$; cf.~\eqref{eq:Y_I_dfn}. We have already shown above that this map satisfies all of the required properties.
\end{proof}

\section{The boundary measurement formula}
In this section, we prove the formula stated in \cref{thm:intro:bound_meas} and extend it to non-generic tuples $\bth\in\THtp_f$, as well as to complex tuples $\bt\in\Tspace_f$. The proofs are adapted from the analogous arguments developed in~\cite{ising_crit}.
 Throughout, we assume that $f\in\Bkn$ is loopless.

\subsection{Nondegenerate boundary measurements} 
We start by recasting the formula~\eqref{eq:curve} for $\curve_{f,\bth}(t)$ using the affine notation introduced in \cref{sec:cyclic-shift}. Recall from~\eqref{eq:It_q} that to each $r\in\Z$ we assign a $k$-element subset $\It_r\subset \Z$ whose reduction modulo $n$ gives an element $I_r$ of the Grassmann necklace of $f$. Since $f$ is loopless, $\It_r$ contains $r$ and we set $\Jt_r:=\It_r\setminus \{r\}$. For $r\in[n]$, the reduction of $\Jt_r$ modulo $n$ is the set $J_{r}$ introduced in~\eqref{eq:intro:J_r}. The affine version of~\eqref{eq:curve} then reads
\begin{equation}\label{eq:curve_aff}
    \gt_r(t)= \prod_{p\in \Jt_r} \sin(t-\tht_p) \qquad\text{for $r\in\Z$.}
\end{equation}
Here $(\gt_r(t))_{r\in\Z}$ is the unique sequence of functions of $t$ satisfying $\gt_r(t)=\g_r(t)$ for $r\in[n]$ and $\gt_{r+n}(t)=(-1)^{k-1}\gt_r(t)$ for all $r\in\Z$. The latter condition follows from the analogous condition~\eqref{eq:bth_extend} on $\btht$. In particular, the sign $(-1)^{k-1}$ is compatible with the cyclic symmetry of $\Grtnn(k,n)$ in \cref{sec:cyclic-shift}. As we see, the sign $\eps_r$ appears in~\eqref{eq:curve} but disappears in~\eqref{eq:curve_aff}.

Recall from \figref{fig:align}(d) that for $1\leq p\neq q\leq n$, we write $\{p,q\}\in \Misdu_f$ if the points $\p s,\m p,\p t,\m q$ are ordered counterclockwise, where $s:=\perm^{-1}(p)$ and $t:=\perm^{-1}(q)$.
\begin{definition}
  We say that $\bth \in\THtp_f$ is \emph{$f$-nondegenerate} if  we have $\th_p\not\equiv \th_q$ modulo $\pi$ for all $\{p,q\}\in\Misdu_f$. Similarly, $\bt\in\Tspace_f$ is \emph{$f$-nondegenerate} if $t_p\neq\pm t_q$ for all $\{p,q\}\in\Misdu_f$.
\end{definition}
\noindent Recall that \cref{thm:intro:bound_meas} was stated for the case when $\bth$ is \emph{generic} (and $f$-admissible); every such tuple is also $f$-nondegenerate.

Next, let us generalize the curve $\curve_{f,\bth}(t)$ to the complex algebraic setting. Write
\begin{equation}\label{eq:ggt_dfn}
    \ggt_r(t):= \frac1{(2i)^{k-1}} \prod_{p\in \Jt_r} \br[t,\tt_p] \quad\text{for $r\in\Z$}.
\end{equation}
We again have $\ggt_{r+n}(t)=(-1)^{k-1}\ggt_r(t)$ for $r\in\Z$; cf. \cref{sec:open-crit-vari}. For $r\in[n]$, let $\gg_r(t):=\ggt_r(t)$ and introduce a map $\curvec_{f,\bt}:\Cast\to \C^n$ given by $\curvec_{f,\bt}(t)=(\gg_1(t),\dots,\gg_n(t))$. Thus $\curvec_{f,\bt}(t)$ specializes to $\curve_{f,\bth}(t)$ when $\bt$ and $\bth$ are related by~\eqref{eq:th_to_t}.

\begin{theorem}\label{thm:bound_meas_C}
  Suppose that $\bt\in\Tspace_f$ is $f$-nondegenerate. Then $\Span_\C(\curvec_{f,\bt})$ has dimension $k$ and we have 
\begin{equation*}%
  \Meas(f,\bt)=\Span_\C(\curvec_{f,\bt}) \quad\text{inside $\Gr(k,n)$.}
\end{equation*}
\end{theorem}

\begin{corollary}
Suppose that $\bth\in\THtp_f$ is $f$-nondegenerate. Then $\Span_\R(\curve_{f,\bth})$ has dimension $k$ and we have 
\begin{equation*}%
  \Meas(f,\bth)=\Span_\R(\curve_{f,\bth}) \quad\text{inside $\Grtnn(k,n)$.}
\end{equation*}
\end{corollary}
\noindent As explained above, this result generalizes \cref{thm:intro:bound_meas} from generic to $f$-nondegenerate tuples $\bth\in\THtp_f$.

We also mention a boundary measurement formula for dual critical varieties. Fix a coloopless $f\in\Bkn$ and recall the notation from \cref{rmk:dual_cio}. 
We write $\{p,q\}\in\MisDdu_f$ if the points $\m s,\p p,\m t,\p q$ are ordered counterclockwise for $s:=\perm^{-1}(p)$ and $t:=\perm^{-1}(q)$.  We say that $\bt\in\Tspaced_f$ is \emph{dual $f$-nondegenerate} if $t_p\neq\pm t_q$ whenever $\{p,q\}\in\MisDdu_f$. Let $\curvecd_{f,\bt}(t)=(\ggd_1(t),\ggd_2(t),\dots,\ggd_n(t))$ be given by 
\begin{equation}\label{eq:ggd}
  \ggd_r(t):=\br[\tt_{f(r)},\tt_r]\cdot \gg_r(t) \quad\text{for $r\in[n]$,}
\end{equation}
 where $\gg_r(t)$ is defined by~\eqref{eq:ggt_dfn}. Note that when $r\in[n]$ is a loop, we have $\ggd_r(t)=0$ since $\br[\tt_{f(r)},\tt_r]=0$. The proof of the following result is completely analogous to the proof of \cref{thm:bound_meas_C} given below.

\begin{theorem}\label{thm:bound_meas_dual}
Let $f\in\Bkn$ be coloopless and $\bt\in\Tspaced_f$ be dual $f$-nondegenerate. Then 
\begin{equation*}%
  \Measd(f,\bt)=\Span_\C(\curvecd_{f,\bt}) \quad\text{inside $\Gr(k,n)$.}
\end{equation*}
\end{theorem}
\begin{example}
Consider the case $f=f_{1,4}$. Even though $\Ctp_{f}$ is a single point, $\Ctpd_f$ is not a single point since the four boundary edges of the corresponding reduced graph $G$ have weights $\H21,\H32,\H43,\H41$. Therefore $\Meas(f,\bt)\in \Gr(1,4)$ is given by the point $\Measd(f,\bt)=(\H21:\H32:\H43:\H41)$.
For each $r\in [n]$, we have $J_r=\emptyset$, thus $\gg_r(t)=1$. But because we have the extra term in~\eqref{eq:ggd}, we find that $\curvecd_{f,\bt}$ depends on $\bt$ (but not on $t$):
\begin{equation*}%
  \curvecd_{f,\bt}(t)=(\H21,\H32,\H43,\br[\tt_5,\t_4]).
\end{equation*}
We see that $\Span_\C(\curvecd_{f,\bt})$ agrees with $\Measd(f,\bt)$ since  $\br[\tt_5,\t_4]=-\H14=\H41$.
\end{example}

\begin{remark}
It follows by combining \cref{prop:duality} with \cref{thm:bound_meas_C,thm:bound_meas_dual} that the curves $\curve_{f,\btht}$ and $\dual\curve_{\dual f,\btht\circ f}$ span orthogonal subspaces of $\R^n$. We do not have a direct explanation for this phenomenon, even for $f=\fkn$.
\end{remark}

\begin{proof}[Proof of \cref{thm:bound_meas_C}]
Recall that $f\in\Bkn$ is assumed to be loopless. Suppose that $f(r)=r+n$ for some coloop $r\in[n]$. Then $r\notin J_r$ but $r\in J_q$ for all $q\in[n]\setminus\{r\}$. Moreover, since $\bt$ is $f$-nondegenerate, we see that $t_r\neq \pm t_p$ for all $p\in J_r$. Thus $\curvec_{f,\bt}(t_r)\in\C^n$ has a single nonzero coordinate in position $r$. Let $f'\in\Bound(k-1,n-1)$ be obtained from $f$ by removing the coloop at $r$, and let $\bt'$ be obtained from $\bt$ by omitting $t_r$. Clearly $f'$ is loopless and $\bt'$ is $f'$-nondegenerate. The entries of $\curvec_{f',\bt'}$ are obtained from the corresponding entries $\gg_q(t)$ of $\curvec_{f,\bt}$ (for $q\neq r$) by dividing by $\frac1{2i}\br[t,t_r]$. It follows that if the statement of \cref{thm:bound_meas_C} is true for $f'$ then it is true for $f$.

We proceed by induction using the bridge removal procedure from \cref{sec:bridge_rem}.  For the base case $k=n=1$, the statement is clear. Let $f\in\Bkn$ be loopless. We have shown above that we may assume that $f$ is also coloopless. As explained in \cref{sec:bridge_rem}, there exists some index $r\in[n]$ such that $f$ has a bridge at $r$, thus $r<r+1\leq f(r)<f(r+1)\leq r+n$. Let $f':=s_rf\in\Bkn$ and consider a graph $G\in\Gred(f)$ that contains a bridge at $b_r,b_{r+1}$ as shown in \figref{fig:bridge}(left). It may happen that $f'$ is not loopless when $f(r)=r+1$, thus we first consider this case.

Assume that $f(r)=r+1$. Then $r+1$ does not appear in $J_q$ for any $q\in[n]$ and we have $\ggt_r(t)=\ggt_{r+1}(t)$ for all $t\in\Cast$. The $r$-th and $(r+1)$-th columns of $\Meas(f,\bt)$ also agree; see \cref{sec:factorization}. Let $f''\in\Bound(k,n-1)$ be obtained from $f'$ by removing the loop at $r$. We see that if the boundary measurement formula holds for $f''$ then it holds for $f$.

Assume now that $f(r)>r+1$. Let $a=f(r)$ and $b=f(r+1)$. Since both $\Meas(f,\bt)$ and $\curvec_{f,\bt}$ are compatible with the cyclic shift from \cref{sec:cyclic-shift}, we may assume that $1\leq r<r+1<a< b\leq n$. Then it is easy to check using~\cite[Lemma~7.6]{LamCDM} and \cref{fig:bridge} that $\Meas(f',\bt)=\Meas(f,\bt)\cdot g$, where $g= x_r(-\H ba/\H b{r+1})\cdot d_{r+1}(\H b{r+1}/\H a{r+1})$ and the matrices $x_j(t),d_{j}(t)\in\GL_n(\C)$ differ from the identity matrix as follows: $x_j(t)$ contains a single nonzero off-diagonal entry equal to $t$ in row $j$ and column $j+1$ while $d_j(t)$ is a diagonal matrix whose $(j,j)$-th entry is equal to $t$ and all other diagonal entries of $d_j(t)$ are equal to $1$. The $2\times 2$ block of $g$ in rows and columns $r,r+1$ is given by 
\begin{equation*}%
g|_{\{r,r+1\}\times \{r,r+1\}}=  \begin{pmatrix}
    1 & -\H ba/\H a{r+1}\\
    0 & \H b{r+1}/\H a{r+1}
  \end{pmatrix}.
\end{equation*}
We claim that $\curvec_{f',\bt}(t)=\curvec_{f,\bt}(t)\cdot g$ for all $t\in\Cast$. Denote the coordinates of $\curvec_{f',\bt}(t)$ by $\ggp_q(t)$ for $q\in[n]$, and let $J'_q$ be the set associated to $q$ by~\eqref{eq:intro:J_r} using the strand diagram of $f'$. First, observe that $\ggp_q(t)=\gg_q(t)$ for $q\neq r+1$. The $(r+1)$-th coordinates differ since $J'_{r+1}=J_{r+1}\setminus \{a\}\sqcup \{b\}$. Thus $\gg_{r+1}(t)=P(t)\cdot \HT a$ while $\ggp_{r+1}(t)=P(t)\cdot \HT b$ for some rational function $P(t)$. Note also that $\gg_r(t)=\ggp_r(t)=P(t)\cdot \HT{r+1}$. A direct computation yields
\begin{equation*}%
  \begin{pmatrix}
\HT{r+1}& \HT a
  \end{pmatrix}\cdot \begin{pmatrix}
    1 & -\H ba/\H a{r+1}\\
    0 & \H b{r+1}/\H a{r+1}
  \end{pmatrix}=\begin{pmatrix}
\HT{r+1} & \HT b
  \end{pmatrix}.
\end{equation*}
This implies the desired identity $\curvec_{f',\bt}(t)=\curvec_{f,\bt}(t)\cdot g$. Thus if the boundary measurement formula holds for $f'$ then it holds for $f$, which finishes the induction step.
\end{proof}

\subsection{Choosing a basis}\label{sec:fourier}
Even though \cref{thm:bound_meas_C} describes $\Meas(f,\bt)$ as an element of $\Gr(k,n)$, it is convenient to specify an explicit $k\times n$ matrix representative for such a space. We discuss two ways of describing such representatives: taking $k$ distinct points on the curve $\curve_{f,\bth}(t)$ and taking the basis of its Fourier coefficients. The proofs translate verbatim from~\cite[Section~3]{ising_crit}.

\begin{proposition}\ \label{prop:basis}
\begin{theoremlist}
\item Assume that $\bth\in\THtp_f$ is $f$-nondegenerate. Then for any $0\leq s_1<s_2<\dots<s_k< \pi$, the vectors $\curve_{f,\bth}(s_1),\curve_{f,\bth}(s_2),\dots,\curve_{f,\bth}(s_k)$ form a basis of $\Span_\R(\curve_{f,\bth})$.
\item Assume that $\bt\in\Tspace_f$ is $f$-nondegenerate. Then for any generic tuple $(s_1,s_2,\dots,s_k)\in(\Cast)^k$, the vectors $\curvec_{f,\bt}(s_1),\curvec_{f,\bt}(s_2),\dots,\curvec_{f,\bt}(s_k)$ form a basis of $\Span_\C(\curvec_{f,\bt})$.
\end{theoremlist}
\end{proposition}
\begin{proof}
By \cref{thm:bound_meas_C}, the spaces $\Span_\R(\curve_{f,\bth})$ and $\Span_\C(\curvec_{f,\bt})$ are $k$-dimensional. It suffices to show that the given vectors span these subspaces. This can be shown by multiplying on the left by an appropriate Vandermonde-type matrix; see~\cite[Lemma~3.2]{ising_crit} for details.
\end{proof}

To describe the Fourier basis, observe that the coordinates of $\curvec_{f,\bt}(t)$ are Laurent polynomials in $t$:
\begin{equation*}%
  \gg_r(t)=\frac1{(2i)^{k-1}}\sum_{p=1}^k(-1)^{k-p} c_{p,r} t^{2p-k-1}.
\end{equation*}
Here $c_{p,r}$ is the $(r-1)$-th elementary symmetric polynomial in the variables $(t^2_q)_{q\in J_r}$ divided by $\prod_{q\in J_r} t_q$; see~\cite[Equation~(3.2)]{ising_crit}.

Let $F_{f,\bt}=(c_{p,r})_{p\in[k],r\in[n]}$ be the corresponding $k\times n$ \emph{Fourier coefficient matrix}.\footnote{The name is explained by the fact that if $\bth$ and $\bt$ are related by~\eqref{eq:th_to_t} then up to a simple transformation, the rows of $F_{f,\bt}$ yield the Fourier coefficients of $\curve_{f,\bth}(t)$, viewed as a $2\pi$-periodic function of $t$.} The next result also follows by combining the proof of~\cite[Lemma~3.2]{ising_crit} with \cref{thm:bound_meas_C}.
\begin{proposition}\label{prop:Fourier}
Assume that $\bt\in\Tspace_f$ is $f$-nondegenerate. Then the rows of $F_{f,\bt}$ form a $\C$-basis of $\Span_\C(\curvec_{f,\bt})$.\qed
\end{proposition}

\subsection{Degenerate boundary measurements}\label{sec:degen-bound-meas}

 The boundary measurement formula in \cref{thm:bound_meas_C} holds when $\bt\in\Tspace_f$ is $f$-nondegenerate. If $\bt\in\Tspace_f$ is $f$-degenerate (that is, not $f$-nondegenerate) then it is not hard to see that the span of $\curvec_{f,\bt}(t)$ has dimension strictly less than $k$. Nevertheless, $\Meas(f,\bt)\in\Gr(k,n)$ is still defined for all $\bt\in\Tspace_f$. In order to extend the boundary measurement formula to all of $\Tspace_f$, we adapt the constructions from~\cite[Section~6]{ising_crit}.

 First, we introduce a slight modification of $\curvec_{f,\bt}(t)$. Let
\begin{equation}\label{eq:gggt_dfn}
  \gggt_r(t):= \frac1{(2i)^{k-1}} \prod_{p\in \Jt_r} \left(\frac t{\tt_p}-\tt_p\right) \quad\text{for $r\in\Z$}.
\end{equation}
Thus~\eqref{eq:gggt_dfn} differs from~\eqref{eq:ggt_dfn} in that the terms on the right hand side are of the form $\left(t/{\tt_p}-\tt_p\right)$ rather than $\br[t,\tt_p]=(t/\tt_p-\tt_p/t)$. We let $\ggg_r(t):=\gggt_r(t)$ for $r\in[n]$ and $\GGG_{f,\bt}(t)=(\ggg_1(t),\ggg_2(t),\dots,\ggg_n(t))$. Unlike for $\curvec_{f,\bt}(t)$, the coordinates of $\GGG_{f,\bt}(t)$ are genuine polynomials in $t$.

For $q\in[n]$, let $\supp_f(q):=\{r\in[n]\mid q\notin J_r\}$ and $\v_q:=t_q^2$. 
 Thus $\ggg_r(\v_q)=0$ for $r\notin\supp_f(q)$. For $x\in\C^{n}$ and $S\subset[n]$, let $x|_S\in\C^{n}$ be the vector obtained from $x$ by sending the coordinates $x_r$ to zero for $r\notin S$. 
 For $m\geq0$, denote by $\GGG^\parr m_\ft(\t)$ the $m$-fold derivative of $\GGG_\ft(\t)$. 
Finally, for $r\in[n]$, let $m_r:=\#\{q\in\J_r\mid \v_q=\v_r\}$ be the degree with which $(\t-\v_r)$ divides~$\ggg_r(\t)$. Denote
\begin{equation*}%
  \u r:=\GGG^\parr{m_r}_\ft(\v_r)|_{\supp_f(r)} \quad\text{for $r\in[n]$.}
\end{equation*}
Thus $\u r\in\C^n$ is obtained by (i) differentiating $\GGG_\ft(t)$ $m_r$ times, (ii) substituting $\t=\v_r$, and (iii) sending all coordinates not in $\supp_f(r)$ to $0$.

\begin{theorem}\label{thm:arb_reg}
Let $f\in\Bkn$ be loopless and $\bt\in\Tspace_f$. Then for each $r\in[n]$, the vectors 
\begin{equation}\label{eq:u_j} 
\{\u p\mid p\in I_r\}
\end{equation}
form a basis of $\Meas(f,\bt)$. 
\end{theorem}
\begin{proof}
The proof is obtained by modifying the details of~\cite[Proof of Theorem~6.1]{ising_crit} in a straightforward fashion.
\end{proof}

\section{Applications}\label{sec:applications}
We explain the results on electrical networks from \cref{sec:intro:applications} in more detail. Along the way, we compare them to their Ising model counterparts obtained in~\cite{ising_crit}.

\subsection{Embeddings into \texorpdfstring{$\Grtnn(k,n)$}{the totally nonnegative Grassmannian}}
Let $\Tiling$ be a rhombus tiling of a polygonal region $\Reg$ as in \cref{sec:intro:applications}. Let $G_\Tiling$ be the corresponding isoradial graph with boundary vertices $\b_1,\b_2,\dots,\b_N$. Let $\v_1,\v_2,\dots,\v_{2n}\in\C$ be the unit vectors traversing the sides of $\Reg$ in clockwise order. They are labeled so that $\b_p$ is incident to $\v_{2p-1}$ and $\v_{2p}$ for each $p\in[N]$. This data gives rise to a fixed-point-free involution $\taur:[2N]\to[2N]$ (we refer to such involutions as \emph{pairings}; see e.g. \figref{fig:pis-pel}(a)) defined as follows. Choose $p\in[2N]$ and consider the (unique) rhombus of $\Tiling$ containing the side of $\Reg$ labeled by $\v_p$. Let $\v^\parr1_p$ be the opposite side of this rhombus. Next, $\v^\parr1_p$ is contained in a unique other rhombus of $\Tiling$, so we let $\v^\parr2_p$ denote the opposite side of that rhombus, etc. This way, we create a family of parallel line segments $\v_p,\v^\parr1_p,\v^\parr2_p,\dots$ which terminates at some boundary line segment labeled by $\v_q$ for $q\in[2N]$. We then set $\taur(p):=q$. It is easy to see that we therefore must have $q\neq p$ and $\taur(q)=p$, thus $\taur$ is indeed a fixed-point-free involution. See~\cite[Figure~2(c)]{ising_crit} for an example.

Following~\cite{Lam}, we associate a bounded affine permutation $\fel=\feltr\in\Bound(N+1,2N)$ to $\Reg$. It is the unique loopless bounded affine permutation such that for each $p\in[2N]$, we have $\permel(p)\equiv\taur(p+1)$ modulo $n:=2N$ (where $p+1$ is taken modulo $n$). See \figref{fig:pis-pel}(b).

\begin{figure}
  \includegraphics[width=1.0\textwidth]{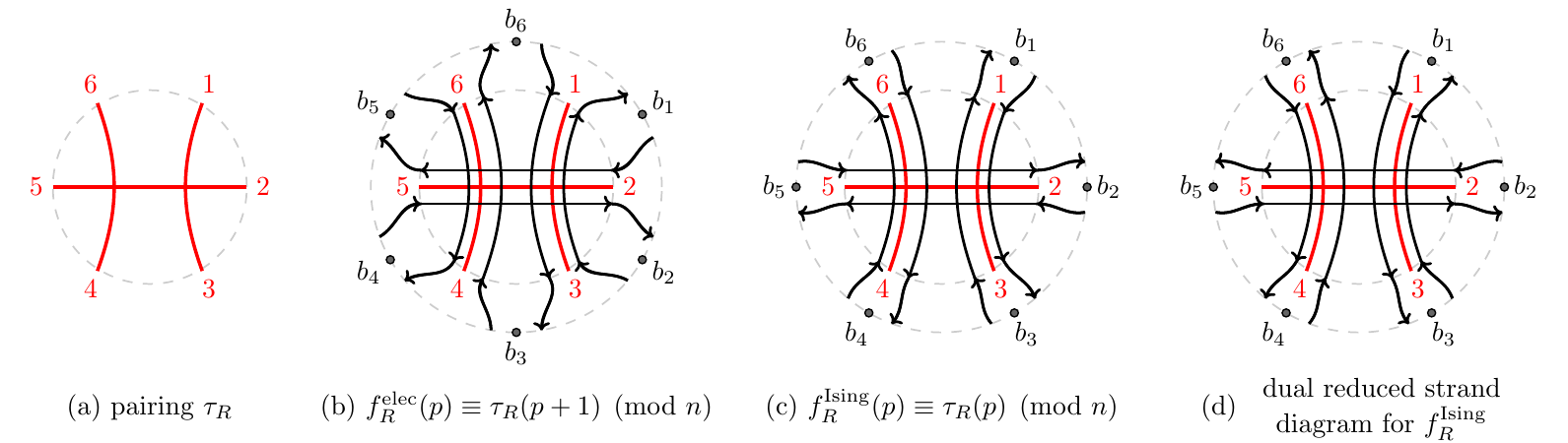}
  \caption{\label{fig:pis-pel} Converting a pairing $\tau_\Reg$ into bounded affine permutations $\fel$ and $\fis$.}
\end{figure}

Recall from \cref{sec:intro:applications} that the graph $G_\Tiling$ is viewed as an electrical network whose $N\times N$ response matrix is denoted $\La^\Reg$. We describe Lam's embedding $\pel$ of the space of $N\times N$ electrical  response matrices into $\Grtnn(N+1,2N)$. Let $\nodd:=\{1,3,\dots,2N-1\}$. Then $\pel(\La^\Reg)$ is the unique element $X\in\Gr(N+1,2N)$ such that 
\begin{equation}\label{eq:Lam_emb_dfn}
  \Delta_{\nodd\cup\{2p\}}(X)=1 \quad\text{and}\quad \Delta_{\nodd\setminus\{2q-1\}\cup\{2p-2,2p\}}(X)=\la^\Reg_{p,q} \quad\text{for all $p\neq q\in[N]$.}
\end{equation}
Here the index $2p-2$ is taken modulo $n$. The element $\pel(\La^\Reg)$ turns out to belong to $\Grtnn(N+1,2N)$. Moreover, it belongs to $\Ptp_{\fel}$.
\begin{remark}
The above embedding into $\Grtnn(N+1,2N)$ is obtained from Lam's embedding into $\Grtnn(N-1,2N)$ by composing it with the map $\altp$ from \cref{sec:duality}. The description~\eqref{eq:Lam_emb_dfn} is deduced by combining~\cite[Proposition~2.4]{Lam} with the description of the embedding in terms of \emph{concordant sets} given in~\cite[Section~5.2]{Lam}. 
\end{remark}

The \emph{Ising model} is a probability distribution on the space $\{\pm1\}^{V(G)}$ of \emph{spin configurations} on the vertices of a weighted graph $G$. Given two vertices $u,v\in V(G)$, one can consider their \emph{spin correlation}, denoted $\<\sigma_u\sigma_v\>$. It is a real number between $-1$ and $1$. The \emph{critical Ising model}~\cite{Bax2,Bax} is obtained in the case where $G=G_\Tiling$ is the graph described above, but it has a different choice of edge weights. Consider the \emph{boundary correlations} $\<\sigma_p\sigma_q\>_\Tiling:=\<\sigma_{\b_p}\sigma_{\b_q}\>_{G_\Tiling}$ forming an $N\times N$ \emph{boundary correlation matrix} $\M^\Tiling:=(\<\sigma_{p}\sigma_{q}\>_\Tiling)_{p,q\in[N]}$. As in the case of electrical networks, the matrix $\M^\Tiling$ is invariant under star-triangle moves and therefore depends only on the region $\Reg$. We denote $\M^\Reg:=\M^\Tiling$.

Let $\fistr=\fis\in\Bound(N,2N)$ be the unique bounded affine permutation such that for each $p\in[2N]$, we have  $\permis(p)\equiv\taur(p)$ modulo $n$; see \figref{fig:pis-pel}(c). In~\cite{GP}, we associated an element $\pis(\M^\Reg)\in\Grtnn(N,2N)$ to any boundary correlation matrix. It belongs to the positroid cell $\Ptp_{\fis}$.

\subsection{Pairings}
Let $\tau:[n]\to[n]$ be a pairing. We say that $\{p,q\}$ is a \emph{$\tau$-pair} if $\tau(p)=q$, and we identify $\tau$ with the corresponding set $\{\{p_1,\tau(p_1)\},\{p_2,\tau(p_2)\},\dots,\{p_N,\tau(p_N)\}$ of $\tau$-pairs. We say that two $\tau$-pairs $\{p,q\}\neq\{p',q'\}$ \emph{form a $\tau$-crossing} if the points $p,p',q,q'$ are cyclically ordered (either clockwise or counterclockwise). We let $\xing(\tau)$ denote the number of $\tau$-crossings. Similarly to~\eqref{eq:E_GX_f}, we introduce an undirected graph $\GX_\tau$ with vertex set $[n]$ and two indices $p,p'\in[n]$ forming an edge whenever the corresponding $\tau$-pairs $\{p,\tau(p)\}$ and $\{p',\tau(p')\}$ form a $\tau$-crossing. 
 We say that $\tau$ is \emph{connected} if $\GX_\tau$ is connected. 

Consider $n$ points $d_1,d_2,\dots,d_n$ on a circle in clockwise order. A \emph{$\tau$-pseudoline arrangement} $\Acal$ is a collection of $N$ embedded unoriented paths in a disk, each connecting $d_p$ to $d_{\tau(p)}$ for some $p\in[n]$, such that no two paths intersect more than once and no three paths intersect at a single point. Each rhombus tiling $\Tiling$ of a polygonal region $\Reg$ is planar dual to a $\taur$-pseudoline arrangement $\Acal_\Tiling$: the pseudoline connecting $d_p$ to $d_{\taur(p)}$ is obtained by connecting the midpoints of the above line segments $\v_p,\v^\parr1_p,\v^\parr2_p,\dots$; see \cite[Figure~2(c)]{ising_crit}.

 We say that a tuple $\bth=(\th_1,\th_2,\dots,\th_n)\in\R^n$ is \emph{$\tau$-isotropic} if 
\begin{equation}\label{eq:bth_isotr}
  \th_{q}=\th_p+\pi/2 \quad\text{for all $\tau$-pairs $\{p,q\}$ with $1\leq p<q\leq n$.}
\end{equation}
Extending $\bth$ to $\btht:\Z\to\R$ via~\eqref{eq:bth_extend}, the above condition becomes 
\begin{equation*}%
  \tht_{q}=\tht_p+\pi/2 \quad\text{for all $p<q\in \Z$ such that $\fis(p)=q$.}
\end{equation*}

We say that a tuple $\bth\in\R^n$ is \emph{$\tau$-admissible} if it is $\tau$-isotropic and satisfies 
\begin{equation*}%
  \th_p<\th_{p'}<\th_q<\th_{q'}
\end{equation*}
for all $1\leq p<p'<q<q'\leq n$ such that the $\tau$-pairs $\{p,q\}$ and $\{p',q'\}$ form a $\tau$-crossing. One easily observes that given a rhombus tiling $\Tiling$ of $\Reg$, one can choose a $\tau$-admissible tuple $\bth$ such that $\v_p=\exp(-2i\th_p)$ for all $p\in[n]$. 

We may now view the region $\Reg$ as a pair $(\taur,\bth)$ where $\bth$ is $\taur$-admissible. Recall that for any rhombus tiling $\Tiling$ of $\Reg$, the response matrix $\La^\Tiling=\La^\Reg$ of the associated electrical network $G_\Tiling$ depends only on $\Reg$. More generally, it is possible to associate an electrical network $G_{(\Acal,\bth)}$ to an arbitrary pair $(\Acal,\bth)$ consisting of a $\tau$-pseudoline arrangement $\Acal$ and a $\tau$-admissible tuple $\bth$; see~\cite[Section~2]{ising_crit}. The resulting response matrix $\La^{(\Acal,\bth)}$ again depends only on $(\tau,\bth)$ and is denoted $\La^{(\tau,\bth)}$. For a pair $(\tau,\bth)$ where $\bth$ is $\tau$-admissible, we denote $\Reg:=(\tau,\bth)$ and refer to $\Reg$ as a \emph{generalized region} (called a \emph{valid region} in~\cite{ising_crit}).

\subsection{Back to critical cells}
Given a pairing $\tau:[n]\to[n]$, we have defined the bounded affine permutations $\felt\in\Bound(N+1,2N)$, $\fist\in\Bound(N,2N)$, $\tau$-pseudoline arrangements $\Acal$, and $\tau$-admissible tuples $\bth$. In this section, we discuss the relationship between these notions and the notions introduced above in the context of critical varieties, such as reduced strand diagrams and $f$-admissible tuples $\bth$.

Let $\Acal$ be a $\tau$-pseudoline arrangement. Recall that the endpoints of pseudolines in $\Acal$ are labeled $d_1,d_2,\dots,d_n$. For each $p\in[n]$, place a point $\md p$ (resp., $\pd p$) slightly before (resp., after) $d_p$ in clockwise order.  Let $\Avec$ be a reduced strand diagram whose (directed) strands connect $\md p\to \pd q$ whenever $\tau(p)=q$.

Relabeling the boundary points of $\Avec$ by $\m p:=\md p$ and $\p p:=\pd p$ for all $p\in[n]$, $\Avec$ becomes a dual reduced strand diagram of $\fist$; see \figref{fig:pis-pel}(d). However, we may also relabel the boundary points in a different way: setting $\m p:=\pd p$ and $\p p:=\md {p+1}$ for all $p\in[n]$ gives rise to a reduced strand diagram of $\felt$ as in \figref{fig:pis-pel}(b). See \cref{sec:shift} for a detailed discussion of this ``shift by $1$'' correspondence. For now, observe that this identification of strand diagrams allows one to relate the notions of $\tau$-admissible, $\felt$-admissible, and $\fist$-admissible tuples $\bth$.

\begin{proposition}
Let $\tau:[n]\to[n]$ be a pairing and $\bth\in\R^n$. The following are equivalent:
\begin{itemize}
\item $\bth$ is $\tau$-admissible;
\item $\bth$ is $\tau$-isotropic and $\felt$-admissible;
\item $\bth$ is $\tau$-isotropic and $\fist$-admissible.\qed
\end{itemize}
\end{proposition}

We are ready to prove a generalization of \cref{thm:intro:appl}.
\begin{theorem}\label{thm:Meas_appl}
  Let $R:=(\tau,\bth)$ be a generalized region. Then we have
\begin{equation*}%
  \pel(\La^\Reg)=\Meas(\felt,\bth) \quad\text{and}\quad \pis(\M^\Reg)=\Meas(\fist,\bth).
\end{equation*}
\end{theorem}
\begin{proof}
As we mentioned in \cref{sec:intro:applications}, both results follow from the well-known fact that the critical dimer model specializes to critical electrical networks and the critical Ising model. In both cases, one transforms the weighted graph $G_\Tiling$ into a weighted reduced planar bipartite graph $(\Gel,\wtel_\bth)$ (resp., $(\Gis,\wtis_\bth)$) with strand permutation $\felt$ (resp., $\fist$) satisfying $\Meas(\Gel,\wtel_\bth)=\pel(\La^\Reg)$ (resp., $\Meas(\Gis,\wtis_\bth)=\pis(\M^\Reg)$).

\begin{figure}
  \includegraphics{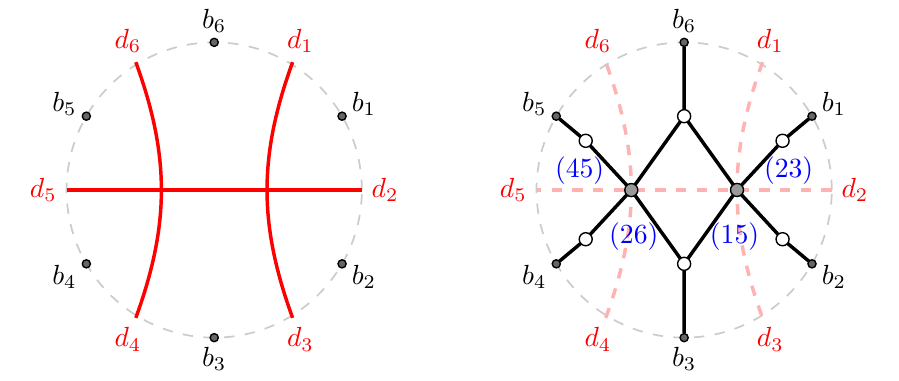}
  \caption{\label{fig:Temperley} Converting a pairing $\tau_\Reg$ into a weighted graph $\Gel\in\Gred(\fel)$.}
\end{figure}

To define $\Gel$, we recall from~\cite{Lam} the \emph{generalized Temperley trick} of~\cite{KPW}. Place $4n$ points 
\begin{equation*}%
\m 1=\pd 1, b_1, \p 1=\md 2, d_2, \m 2=\pd 2, b_2,\dots, b_n, \p n=\md 1
\end{equation*}
on the circle in clockwise order. For each $p\in[n]$, connect $d_p$ to $d_{\tau(p)}$ by a pseudoline as in  \figref{fig:Temperley}(left). We get a $\tau$-pseudoline arrangement $\Acal$. Place a black interior vertex on every crossing of two pseudolines and a white interior vertex in every face of $\Acal$. Here by a \emph{face} we mean a connected component of the complement of $\Acal$ in the disk. These are the black and white interior vertices of the graph $\Gel$, which also has $n$ black boundary vertices $b_1,b_2,\dots,b_n$. Each interior black vertex $v$ of $\Gel$ is an intersection point of two pseudolines and therefore is adjacent to four faces of $\Acal$. Connect $v$ to the corresponding four white interior vertices of $\Gel$. In addition, observe that each boundary vertex $b_p$ is adjacent to a single face of $\Acal$, so we connect $b_p$ by an edge to the corresponding white vertex of $\Gel$. We have described the vertices and the edges of $\Gel$. To describe $\wtel_\bth$, color the faces of $\Acal$ black and white in a bipartite way so that for each $r\in[n]$, the face containing $b_r$ is colored black if and only if $r$ is odd. Then for each interior edge $e$ of $\Gel$, we set $\wtel_\bth(e):=1$ if $e$ is contained in a white face of $\Acal$ and $\wtel_\bth(e):=\sin(\th_q-\th_p)$ if $e$ is labeled by $\{p,q\}$ (as in \cref{sec:intro:critical-dimer-model}) for $1\leq p<q\leq n$ and is contained in a black face of $\Acal$. See \figref{fig:Temperley}(right). Consider a black interior vertex $v$ of $\Gel$ corresponding to an intersection of two pseudolines connecting $\tau$-pairs $\{p,q\}$ and $\{p',q'\}$ with $1\leq p<p'<q<q'\leq n$. Then $v$ has degree $4$ in $\Gel$ and is incident to edges of weights either $(1,\sin(\th_{q'}-\th_{q}),1,\sin(\th_{p'}-\th_{p}))$ or $(1,\sin(\th_{q'}-\th_{p}),1,\sin(\th_{q}-\th_{p'}))$ in clockwise order. Since $\bth$ is $\tau$-admissible, we have $\sin(\th_{q'}-\th_{q})=\sin(\th_{p'}-\th_{p})$ and $\sin(\th_{q'}-\th_{p})=\sin(\th_{q}-\th_{p'})$. Therefore these edge weights coincide with (the dual version of) the edge weights in~\cite[Section~5.1]{Lam}. On the other hand, we clearly have $\wtel_\bth(e)=\wt_\bth(e)$ for all edges $e$ of $\Gel$. We have shown the result for electrical networks.

Similarly, for the Ising model, the edge weights $\wtis_\bth$  of the graph $\Gis$ studied in~\cite{Dubedat,GP} are  easily seen to coincide with $\wt_\bth$. We refer to~\cite{GP,ising_crit} for further details.
\end{proof}

\subsection{Cyclically symmetric case}
Recall from \cref{sec:intro:cyc_symm} that for each $k\leq n$, there exists a unique point $\XOkn\in\Grtnn(k,n)$ that is invariant under the cyclic shift: $\shift(\XOkn)=\XOkn$. In other words, the point $\XOkn$ is characterized by the property that its Pl\"ucker coordinates are all positive and satisfy $\Delta_I(\XOkn)=\Delta_{\sigma I}(\XOkn)$ for all $I\in{[n]\choose k}$. Recall also that we set $\bthr=(\th_1,\th_2,\dots,\th_n)$ to be given by $\th_r:=\frac{r\pi}{n}$ for all $r\in[n]$. 

\begin{proposition}\label{prop:Meas_reg}
For all $k\in[n]$, we have
\begin{equation*}%
  \Measop_{k,n}(\bthr)=\XOkn.
\end{equation*}
\end{proposition}
\begin{proof}
Choose a graph $G\in\Gred(\fkn)$, then $\Measop_{k,n}(\bthr)=\Meas(G,\wt_{\bthr})$. Now let $G'$ be obtained from $G$ by cyclically relabeling the boundary vertices. Since $G'\in\Gred(\fkn)$, it follows that $\shift(\Measop_{k,n}(\bthr))=\Measop_{k,n}(\bthr)$. Since $\Measop_{k,n}(\bthr)\in\Grtnn(k,n)$, the result follows.
\end{proof}

The above observation can be applied to deduce new simple formulas for electrical networks and the Ising model.

Recall that the boundary correlations $\<\sigma_p\sigma_q\>_\Tiling$ of the critical Ising model do not depend on the choice of a rhombus tiling $\Tiling$ of a region $\Reg$ and may thus be denoted by $\<\sigma_p\sigma_q\>_\Reg$. Let $\Reg_N$ be a regular $2N$-gon.

\begin{theorem}[{\cite[Theorem~1.1]{ising_crit}}]%
For $1\leq p,q\leq \N$ and $d:=|p-q|$, we have
\begin{equation*}%
  \<\sigma_p\sigma_q\>_{\Reg_\N}=\frac2\N \left(\frac1{\sin \left((2d-1)\pi/2\N\right)}-\frac1{\sin \left((2d-3)\pi/2\N\right)}+\dots\pm \frac1{\sin \left(\pi/2\N\right)} \right)\mp 1.
\end{equation*}
\end{theorem}

Below we prove the analog of this result for electrical networks. 

\begin{theorem}%
For $1\leq p,q\leq \N$ and $d:=|p-q|$, we have
\begin{equation}\label{eq:reg_Elec_body}
\la^{\Reg_\N}_{p,q}=\frac{\sin(\pi/\N)}{\N\cdot \sin((2d-1)\pi/2\N)\cdot \sin((2d+1)\pi/2\N)}.
\end{equation}
\end{theorem}
\begin{proof}
We sketch the argument; the details may be found in~\cite[Section~5.2]{ising_crit}. 
Rather than working with dual version of Lam's embedding, we will work with the original embedding $\altp\circ \pel$ of the space of $N\times N$ response matrices into $\Grtnn(N-1,2N)$. It sends a response matrix $\La^\Reg$ to the unique element $X\in\Gr(N-1,2N)$ whose Pl\"ucker coordinates satisfy 
\begin{equation}\label{eq:Lam_emb_dfn_dual}
    \Delta_{\nev\setminus\{2p\}}(X)=1 \quad\text{and}\quad \Delta_{\nev\setminus\{2p-2,2p\}\cup\{2q-1\}}(X)=\la^\Reg_{p,q} \quad\text{for all $p\neq q\in[N]$.}
\end{equation}
This description is obtained from~\eqref{eq:Lam_emb_dfn} by applying~\eqref{eq:altp_minors}. In fact, we also have
\begin{equation}\label{eq:Lam_pp}
  \Delta_{\nev\setminus\{2p-2,2p\}\cup\{2p-1\}}(X)=-\la^\Reg_{p,p} \quad\text{for all $p\in[N]$.}
\end{equation}
Our goal is to determine $\la^{\Reg_N}_{p,q}$, where $\Reg_N$ is a regular $2N$-gon. By \cref{prop:Meas_reg} and \cref{thm:Meas_appl}, we see that $\pel(\La^{\Reg_N})=\XO{N+1}{2N}$ and thus $\altp\circ\pel(\La^{\Reg_N})=\XO{N-1}{2N}$. Let $\zeta:=\exp(i\pi/2N)$ and 
\begin{equation*}%
  z_p:=\zeta^{-N+2p}\quad\text{for $p\in[N-1]$.}
\end{equation*}
Then $X:=\XO{N-1}{2N}$ is the row span of an $(N-1)\times 2N$ Vandermonde matrix 
\begin{equation*}%
 A:= \begin{pmatrix}
    1 & z_1 & z_1^2 & \dots & z_1^{2N-1}\\
    1 & z_2 & z_2^2 & \dots & z_2^{2N-1}\\
    \vdots & \vdots & \vdots & \ddots & \vdots\\
    1 & z_{N-1} & z_{N-1}^2 & \dots & z_{N-1}^{2N-1}
  \end{pmatrix}.
\end{equation*}
Let $K$ be a $2N\times (N-1)$ matrix such that $K_{p,q}=1$ if $p=2q$ and $K_{p,q}=0$ otherwise. Thus $AK$ is the submatrix of $A$ with columns $2,4,\dots,2N-2$, and by~\eqref{eq:Lam_emb_dfn_dual}, we see that $AK$ is invertible. Moreover, $AK$ is very close to a discrete Fourier transform matrix (which is unitary), so the inverse of $AK$ is easy to compute. We will only be interested in the last row of $(AK)^{-1}$, which is given by
\begin{equation}\label{eq:last_row_AKi}
  ((AK)^{-1})_{N-1,p}=(-1)^{N} \frac{z_p(1+z_p^2)}{N} \quad\text{for all $p\in[N-1]$.}
\end{equation}
Let us now consider the matrix $M:=(AK)^{-1}A$ and denote its entries by $M=(m_{p,q})$. The submatrix of $M$ with columns $2,4,\dots,2N-2$ is the identity matrix. It follows from~\eqref{eq:Lam_emb_dfn_dual}--\eqref{eq:Lam_pp} that 
\begin{equation}\label{eq:la_pq}
 \la^{\Reg_N}_{p,1}=(-1)^{p}m_{N-1,2N-2p+1} \quad\text{for all $p\in[N]$.} 
\end{equation}
By~\eqref{eq:last_row_AKi}, we have 
\begin{equation*}%
  m_{N-1,q}=\frac{(-1)^{N}}{N}\sum_{p=1}^{N-1} z_p^{q+1}(1+z_p^2) \quad\text{for all $q\in [2N]$.}
\end{equation*}
Summing the two geometric progressions and applying~\eqref{eq:la_pq}, we obtain the formula~\eqref{eq:reg_Elec_body} for $\la^{\Reg_N}_{p,1}$ for all $p\in[N]$. The formula for arbitrary $\la^{\Reg_N}_{p,q}$ then follows from the compatibility of all constructions with the cyclic symmetry of $\Grtnn(N-1,2N)$.
\end{proof}

\section{Shift by $1$}\label{sec:shift}
Recall from \cref{sec:shift-1} that for a loopless bounded affine permutation $f\in\Boundkn$, we let $\dsh f\in\Bound(k-1,n)$ be defined by $\dsh f(p):=f(p-1)$ for all $p\in\Z$. The goal of this section is to relate $\Ctp_f$ to $\Ctpd_{\dsh f}$ and $\Ptp_f$ to $\Ptp_{\dsh f}$.

\subsection{Shift for reduced graphs}
The following two classes of reduced graphs were introduced (for the case $f=\fkn$) in~\cite[Section~7.7]{GPW}.
\begin{definition}
A reduced graph is called \emph{black-trivalent} if it has black boundary and all of its interior black vertices are trivalent. Similarly, a reduced graph is called \emph{white-trivalent} if it has white boundary and all of its interior white vertices are trivalent.
\end{definition}
\noindent See \cref{fig:bpt-wpt} for an example.

\begin{figure}
  \includegraphics{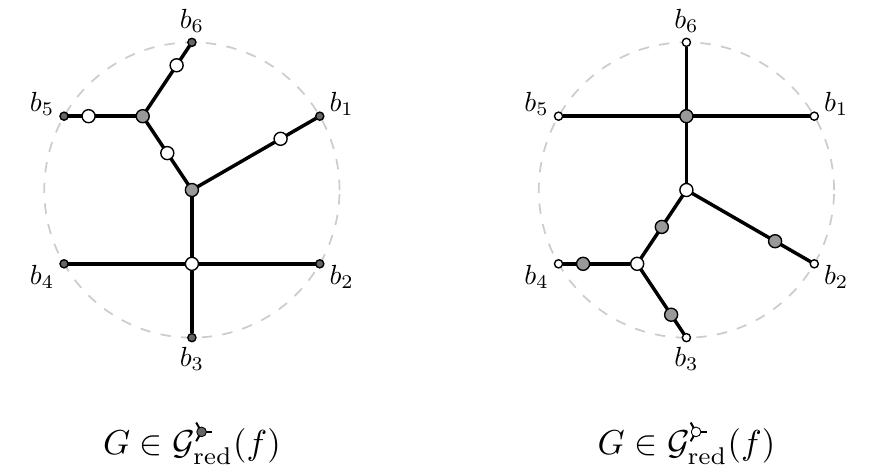}
  \caption{\label{fig:bpt-wpt} A black-trivalent and a white-trivalent graph.}
\end{figure}

\begin{remark}
Unlike~\cite{GPW}, we continue to require all reduced graphs to be bipartite. Thus our white-trivalent graphs are obtained from the \emph{black-partite} graphs of~\cite[Definition~7.14]{GPW} by placing a degree $2$ black vertex in the middle of each edge connecting two trivalent white vertices.
\end{remark}

Denote
\begin{align*}
\BTV(f)&:=\{G\in\Gred(f)\mid \text{$G$ is black-trivalent}\},\\
\WTV(f)&:=\{G\in\Gred(f)\mid \text{$G$ is white-trivalent}\}.
\end{align*}

Recall from \cref{sec:plan-bipart-graphs} that we label the faces of a reduced graph $G$ by $k$-element sets. The following result generalizes~\cite[Proposition~7.15]{GPW}.
\begin{proposition}\label{prop:dsh}
  Let $f\in\Bkn$ be loopless. Then there is a bijection 
  \begin{equation*}%
\BTV(f)\to \WTV(\dsh f),\quad  G\mapsto \dsh G,
  \end{equation*}
characterized by the following property: for any trivalent black vertex of $G$ with adjacent faces labeled by $Sab,Sac,Sbc$, the graph $\dsh G$ contains a trivalent white vertex with adjacent faces labeled $Sa,Sb,Sc$.
\end{proposition}
\noindent Here we abbreviate $Sab:=S\sqcup\{a,b\}$, etc. See \cref{fig:dsh-24} for an example.

\begin{remark}
The bijection $G\mapsto \dsh G$ was independently considered in~\cite[Section~8.2]{PSBW}, where it is described using the planar dual of $G$.
\end{remark}

\begin{proof}
We assume familiarity with the results of~\cite{OPS,chord_sep}. Let us denote by $\Fcal(G)\subset{[n]\choose k}$ the collection of face labels of $G$. We say that two sets $S,T\subset[n]$ (not necessarily of the same size) are \emph{chord separated} if there do not exist indices $1\leq a<b<c<d\leq n$ such that $a,c\in S\setminus T$ and $b,d\in T\setminus S$ or vice versa. By the results of~\cite{OPS},  the collection $\Fcal(G)$ is chord separated and there exists $\Gkn\in\Gred(\fkn)$ such that $\Fcal(G)\subset\Fcal(\Gkn)$. The planar dual of $G$ is a certain polygonal complex called a \emph{plabic tiling} and denoted $\Sigma_k(G)$. Each face of $\Sigma_k(G)$ is a convex polygon whose vertices are labeled by the elements of $\Fcal(G)$. Whenever $G$ has a degree $2$ white vertex, $\Sigma_k(G)$ contains a degenerate white $2$-gon of width zero as in \cref{fig:dsh-24}. Since $G$ is black-trivalent, each black face of $\Sigma_k(G)$ is triangulated (with diagonals replaced with zero-width $2$-gons). The plabic tiling $\Sigma_k(G)$ is a subcomplex of a larger plabic tiling $\Sigma_k(\Gkn)$. Specifically, $\Sigma_k(G)$ consists of all faces of $\Sigma_k(\Gkn)$ which lie inside a polygonal curve passing through the Grassmann necklace of $f$; see~\cite[Proposition~9.8]{OPS}.\footnote{In order to invoke~\cite[Proposition~9.8]{OPS}, one needs to assume that the Grassmann necklace of $f$ is connected. However, even when it is not connected, the Grassmann necklace curve may be deformed slightly into a simple closed curve surrounding $\Sigma_k(G)$; see the discussion around~\cite[Definition~4.4]{BaWe}.} We call this curve the \emph{Grassmann necklace curve of $f$}.

\begin{figure}
  \includegraphics{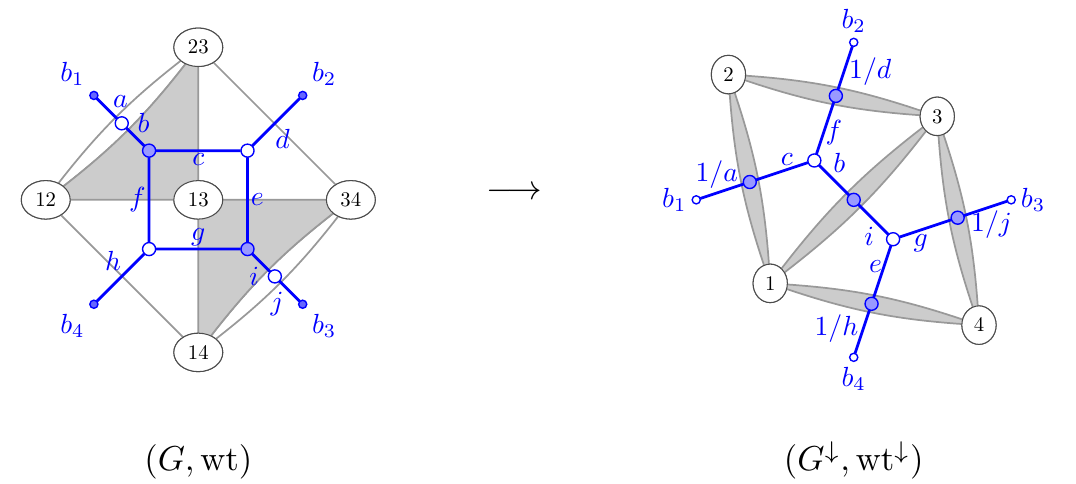}
  \caption{\label{fig:dsh-24} The bijection $(G,\wt)\mapsto (\dsh G,\dsh \wt)$, where $G\in\BTV(f)$ and $\dsh G\in \WTV(\dsh f)$ for $f=f_{2,4}$. See \cref{prop:dsh} and~\eqref{eq:bgau_wgau}. The triangulated plabic tilings $\Sigma_k(G)$ and $\Sigma_{k-1}(\dsh G)$ are shown in grey.}
\end{figure}

We triangulate the remaining black and white faces of $\Sigma_k(G)$, after which it becomes a \emph{triangulated plabic tiling} in the language of~\cite{chord_sep}. By~\cite[Remark~1.5]{chord_sep}, any triangulated plabic tiling appears as a horizontal section $\Sigma_k$ by the plane $z=k$ of some fine zonotopal tiling $\Sigma$ of a cyclic polytope in $\R^3$. Let $\Sigma_{k-1}$ be the horizontal section of $\Sigma$ by the plane $z=k-1$. Then its planar dual is a graph $\Gknd\in\Gred(f_{k-1,n})$.  Moreover, the face labels of $\Gknd$ contain the Grassmann necklace of $\dsh f$, which consists of the sets $(J_2,J_3,\dots,J_n,J_1)$ where $J_p=I_p\setminus p=I_p\cap I_{p+1}$ as above and $(I_1,I_2,\dots,I_n)=\Ical_f$ is the Grassmann necklace of $f$. The reason this is true is that by construction, $\Sigma_k(G)$ contains an edge connecting $I_p$ to $I_{p+1}$ for all $p\in[n]$ (indices taken modulo $n$). When $I_p\neq I_{p+1}$, $\Sigma$ contains a $2$-dimensional face with vertices labeled $J_{p},I_p,I_{p+1},I_p\cup I_{p+1}$. When $J_{p}\neq J_{p+1}$, $\Sigma$ contains a $2$-dimensional face with vertices labeled $J_{p}\cap J_{p+1},J_p,J_{p+1},I_{p+1}$. It follows that $\Sigma$ contains a $2$-dimensional subcomplex $C$ whose intersection with the plane $z=k$ is the Grassmann necklace curve of $f$ and whose intersection with the plane $z=k-1$ is the Grassmann necklace curve of $\dsh f$. Let $\dsh G$ be the planar dual of the subcomplex $\Sigma_{k-1}(\dsh G)$ of $\Sigma_{k-1}$ contained inside the Grassmann necklace curve of $\dsh f$.

Consider a black trivalent vertex of $G$ with adjacent faces labeled by $Sab,Sac,Sbc$. Then $\Sigma_k(G)$ contains a black triangle with vertices labeled $Sab,Sac,Sbc$, thus $\Sigma_{k-1}$ contains a white triangle with vertices labeled $Sa,Sb,Sc$. Moreover, $\Sigma$ contains a cube with vertices labeled $S,Sa,Sb,Sc,Sab,Sac,Sbc,Sabc$. Since the triangle $Sab,Sac,Sbc$ lies inside the Grassmann necklace curve of $\dsh f$, and since the cube cannot intersect $C$ transversally, we see that the triangle $Sa,Sb,Sc$ lies inside the Grassmann necklace curve of $\dsh f$. Thus $\Fcal(\dsh G)$ is a chord separated collection all of whose elements belong to $\Mcal_{\dsh f}$. Using Euler's formula, one can check that it is in fact maximal by size; cf. the proof of~\cite[Lemma~4.2]{chord_sep}. By the results of~\cite{OPS}, $\dsh G$ must belong to $\Gred(\dsh f)$. It remains to note that by construction, the white faces of the plabic tiling $\Sigma_{k-1}(\dsh G)$ are triangulated, and therefore $\dsh G\in\WTV(\dsh f)$.
\end{proof}

\subsection{Shift for positroid cells}\label{sec:shift-positr-cells}
We extend the above bijection $G\mapsto \dsh G$ to a bijection $(G,\wt)\mapsto (\dsh G,\dsh \wt)$ on weighted reduced graphs. Let $G\in\BTV(f)$ and $\dsh G\in\WTV(\dsh f)$ be as above. Denote $E:=E(G)$ and $\dsh E:=E(\dsh G)$. Choose a weight function $\wt:E(G)\to\Rtp$. Consider an interior edge $e\in E$. Exactly one of its endpoints is a black trivalent vertex $v\in V(G)$. Assume that the faces adjacent to $v$ are labeled by $Sab,Sac,Sbc$, with the faces adjacent to $e$ being labeled by $Sab$ and $Sac$. Then the edge $e$ is labeled by $\{b,c\}$. Let $\dsh v\in V(\dsh G)$ be the trivalent white vertex corresponding to $v$ as in \cref{prop:dsh}. The faces adjacent to $\dsh v$ are labeled by $Sa,Sb,Sc$, and we let $\dsh e\in \dsh E$ be the edge adjacent to the faces labeled $Sb$ and $Sc$. In particular, the edge $\dsh e$ is also labeled by $\{b,c\}$. We set $\dsh \wt(\dsh e):=\wt(e)$. It remains to extend this construction to the boundary edges. Let $e_p$ be the boundary edge of $G$ adjacent to $b_p$ for $p\in[n]$. Let $\dsh e_p\in \dsh E$ be the edge adjacent to $b_p$ in $\dsh G$. We set $\dsh \wt(\dsh e_p):=\frac1{\wt(e_p)}$. We have described the bijection $(G,\wt)\mapsto (\dsh G,\dsh \wt)$, and along the way we have also constructed a bijection $E\to \dsh E$ sending $e\mapsto \dsh e$. See \cref{fig:dsh-24} for an example.

Recall from \cref{sec:positroid-varieties} that the boundary measurement map yields a homeomorphism $\Measop_G:\Rtp^E/\Gau\xrasim \Ptp_{f}$. Instead of considering gauge transformations at all vertices of $G$, let us denote by $\Rtp^E/\bGau$ the space of positive edge weight functions on $G$ modulo gauge transformations at trivalent black vertices. Similarly, $\Rtp^{\dsh E}/\wGau$ is defined as the space of positive edge weight functions on $\dsh G$ modulo gauge transformations at trivalent white vertices. For a black trivalent vertex $v$ of $G$, gauge transformations of $(G,\wt)$ at $v$ correspond to gauge transformations of $(\dsh G,\dsh \wt)$ at $\dsh v$. 
 We obtain a natural homeomorphism
\begin{equation}\label{eq:bgau_wgau}
    \Rtp^E/\bGau\xrasim \Rtp^{\dsh E}/\wGau,\quad \wt\mapsto \dsh \wt.
\end{equation}
We will see later (\cref{fig:sq-mv-1}) that any two graphs $G_1,G_2\in\BTV(f)$ are related by certain kinds of \emph{moves}, and that these moves preserve the space $\Rtp^E/\bGau$ and commute with the shift map~\eqref{eq:bgau_wgau}. As a result, we will show (\cref{rmk:monodromy}) that both sides of~\eqref{eq:bgau_wgau} depend only on $f$ and not on the choice of $G$. We thus denote
\begin{equation*}%\label{eq:*}
  \Pmid_f:=\Rtp^E/\bGau \cong \Rtp^{\dsh E}/\wGau,
\end{equation*}
 keeping in mind that until \cref{rmk:monodromy} is shown, this space also depends on the choice of $G$.

Recall that $\Measop_G$ yields a homeomorphism $\Rtp^E/\Gau\xrasim \Ptp_f$, while $\Measop_{\dsh G}$ yields a homeomorphism $\Rtp^{\dsh E}/\Gau\xrasim \Ptp_{\dsh f}$. For the next result, we consider $\Measop_G$ and $\Measop_{\dsh G}$ to be defined on $\Rtp^E/\bGau$ and $\Rtp^{\dsh E}/\wGau$, respectively.
\begin{proposition}\label{prop:MeMe}
For any loopless $f\in\Bkn$ and $G\in\BTV(f)$, we have a homeomorphism
\begin{equation}\label{eq:Meas_times}
  (\Measop_G, \Measop_{\dsh G}): \Pmid_f\xrasim \Ptp_f\times \Ptp_{\dsh f}.
\end{equation}
\end{proposition}
\begin{proof}
Consider the group $\Rtp^{\Vwh(G)}$ of gauge transformations at white interior vertices of $G$. Thus two elements $\wt_1,\wt_2\in \Rtp^E/\bGau$ satisfy $\Meas(G,\wt_1)=\Meas(G,\wt_2)$ if and only if they are related by the action of $\Rtp^{\Vwh(G)}$. Rescaling all edge weights by the same constant clearly yields the same element of $\Rtp^E/\bGau$, and the quotient group $\Rtp^{\Vwh(G)}/\Rtp$ acts simply transitively on the preimage of any point under the surjective map $\Measop_G:\Rtp^E/\bGau\to\Ptp_f$. 

Let us now transfer this action via~\eqref{eq:bgau_wgau} into the action of $\Rtp^{\Vwh(G)}$ on $\Rtp^{\dsh E}/\wGau$. 
It follows from the proof of \cref{prop:dsh} that the white vertices of $G$ are in bijection with the faces of $\dsh G$. Namely, for a white vertex $w\in \Vwh(G)$ of $G$, let $S_1,S_2,\dots,S_m$ be the labels of the faces adjacent to $w$, then $S_1\cap S_2\cap \dots \cap S_m$ is a face label $\lambda(F)\in\Fcal(\dsh G)$ of some face $F$ of $\dsh G$. See also the proof of~\cite[Lemma~4.2]{chord_sep}. 
 Suppose that $\wt_1,\wt_2\in\Rtp^E/\bGau$ are related via a gauge transformation at $w$:
\begin{equation*}%
  \wt_2(e):=
  \begin{cases}
    t\wt_1(e), &\text{if $e$ is incident to $w$;}\\
    \wt_1(e),&\text{otherwise}
  \end{cases}
\end{equation*}
for some $t\in\Rtp$. Let $X_1:=\Meas(\dsh G,\dsh \wt_1)$ and $X_2:=\Meas(\dsh G,\dsh \wt_2)$. We claim that the twisted minors (cf. \cref{sec:twist}) of $X_1$ and $X_2$ are related as follows: for each $I\in\Fcal(\dsh G)$, we have
\begin{equation}\label{eq:gauge_twist}
   \Delta_I(\cev\tau(X_2))=
   \begin{cases}
     t\Delta_I(\cev\tau(X_1)), &\text{if $I=\lambda(F)$,}\\
     \Delta_I(\cev\tau(X_1)), &\text{otherwise.}
   \end{cases}
\end{equation}
To prove~\eqref{eq:gauge_twist}, suppose that $w$ is incident to $e_1,e_2,\dots,e_r\in E$ and let $\dsh e_1,\dsh e_2,\dots,\dsh e_r\in \dsh E$ be the corresponding edges of $\dsh G$. First, assume that none of $e_1,e_2,\dots,e_r$ are boundary edges. It is then easy to check that
\begin{equation}\label{eq:Mleft_cap_gauge}
  \Mleft(F)\cap \{\dsh e_1,\dsh e_2,\dots,\dsh e_r\}=\emptyset \quad\text{and}\quad |  \Mleft(F')\cap \{\dsh e_1,\dsh e_2,\dots,\dsh e_r\}|=1
\end{equation}
for each face $F'\neq F$ of $\dsh G$. Indeed, the tails of any two intersecting strands cannot intersect in a reduced graph, thus the upstream wedges of $\dsh e_1,\dsh e_2,\dots,\dsh e_r$ partition the set of faces of $\dsh G$ not equal to $F$; see \cref{fig:dsh-DW}.

\begin{figure}
  \includegraphics{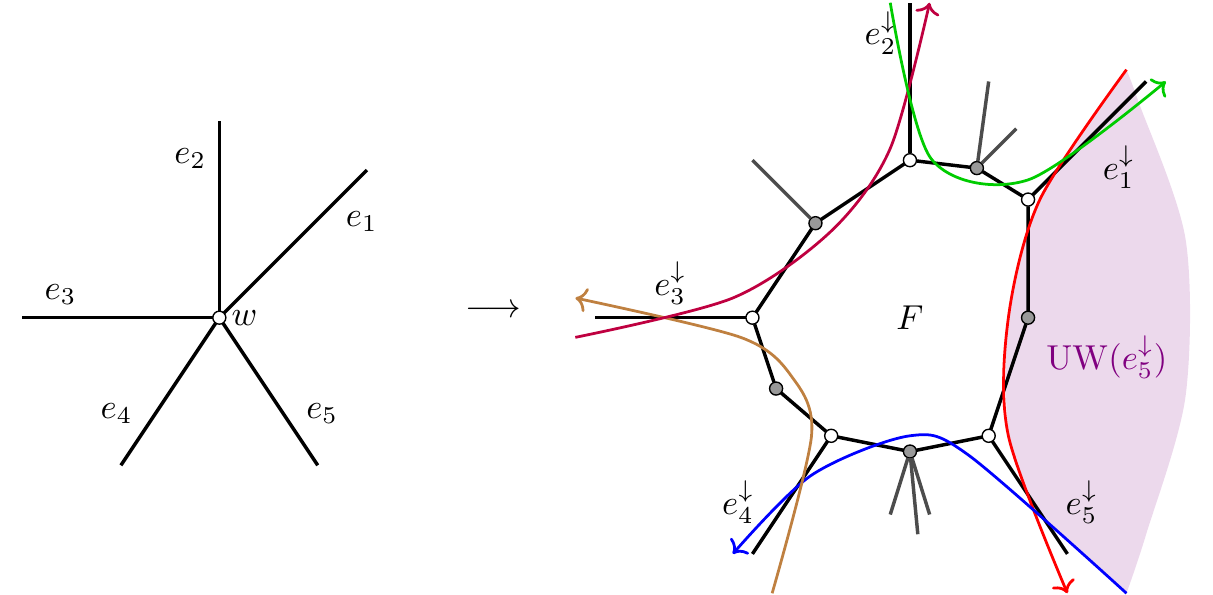}
  \caption{\label{fig:dsh-DW} A white vertex $w$ of $G$ corresponding to a face $F$ of $\dsh G$. The upstream wedges of $\dsh e_1,\dsh e_2,\dsh e_3,\dsh e_4,\dsh e_5$ cover all faces of $\dsh G$ except for $F$.}
\end{figure}

Assume now that we have a partition $\{\dsh e_1,\dsh e_2,\dots,\dsh e_r\}=A\sqcup B$ where $A$ consists of interior edges and $B$ consists of boundary edges. Then~\eqref{eq:Mleft_cap_gauge} no longer holds, but instead we have 
\begin{equation*}%
  (\Mleft(F)\cap A)\triangle B=\emptyset \quad\text{and}\quad |  (\Mleft(F')\cap A)\triangle B|=1
\end{equation*} 
for all $F'\neq F$, where $\triangle$ denotes symmetric difference. (By definition, the \emph{upstream wedge} of a boundary edge $e$ adjacent to $b_p$ consists of all faces to the right of the strand terminating at $b_p$.) 
 Since the weights of boundary edges are inverted in the definition of $\dsh\wt$, and since the twisted minor $\Delta_I(\cev\tau(X_2))$ for $I:=\lambda(F')$ equals the product of the edge weights $\dsh \wt_2(e)$ over $e\in \Mleft(F')$, \eqref{eq:gauge_twist} follows. %

We have shown that the $\Rtp^{\Vwh(G)}$-action on $\Rtp^E/\bGau$ translates into the action of the group $\Rtp^{\Fcal(\dsh G)}$ on $\Ptp_{\dsh f}$ by rescaling the twisted minors associated to the face labels of $\dsh G$. Rescaling all edges by the same constant transforms into rescaling all twisted minors by the same constant, and thus the $\Rtp^{\Vwh(G)}/\Rtp$-action on $\Rtp^E/\bGau$ translates into the $\Rtp^{\Fcal(\dsh G)}/\Rtp$-action on $\Ptp_{\dsh f}$. The latter action is well known~\cite{MuSp} to be simply transitive.

Let us go back to studying the map $(\Measop_G,\Measop_{\dsh G})$ in~\eqref{eq:Meas_times}. First, observe that this map is bijective. Indeed, $\Rtp^{\Fcal(\dsh G)}/\Rtp$ acts simply transitively on the preimages of points under $\Measop_G$ while at the same time it acts simply transitively on the image of $\Measop_{\dsh G}$, since it is identified under~\eqref{eq:bgau_wgau}  with the action of $\Rtp^{\Fcal(\dsh G)}/\Rtp$ on the twisted minors. Second, the map is clearly continuous, and since $\Pmid_f$ and $\Ptp_f\times\Ptp_{\dsh f}$ are both homeomorphic to open balls of the same dimension, the map $(\Measop_G,\Measop_{\dsh G})$ is a homeomorphism by the invariance of domain theorem.
\end{proof}

\begin{example}
Let $f=f_{2,4}$, thus $\dsh f=f_{1,4}$ as in \cref{fig:dsh-24}. The boundary measurements of $X:=\Meas(\dsh G,\dsh \wt)$ computed from \figref{fig:dsh-24}(right) are given by
\begin{equation*}%
  \Delta_1(X)=\frac{ci}{djh},\quad   \Delta_2(X)=\frac{fi}{ajh},\quad   \Delta_3(X)=\frac{bg}{adh},\quad   \Delta_4(X)=\frac{be}{adj}.
\end{equation*}
(Recall that since the boundary vertices of the graph in \figref{fig:dsh-24}(right) are white, the boundary of an almost perfect matching $\Acal$ consists of the boundary vertices \emph{not} used in $\Acal$.) 
By the definition of the twist map, the Pl\"ucker coordinates of $\cev\tau(X)$ are just the inverses of the Pl\"ucker coordinates of $X$: we have $\Delta_1(\cev\tau(X))=\frac{djh}{ci}$, $\Delta_2(\cev\tau(X))=\frac{ajh}{fi}$, etc. The graph $G$ in \figref{fig:dsh-24}(left) has four white vertices, which are in bijection with the four faces of the graph $\dsh G$ in \figref{fig:dsh-24}(right). We see that applying a gauge transformation to all edges incident to a given white vertex of $G$ corresponds to rescaling $\Delta_I(\cev\tau(X))$ by the same constant (where $I\in{[4]\choose1}$ labels the corresponding face) while leaving the other Pl\"ucker coordinates of $\cev\tau(X)$ unchanged. For example, multiplying $a$ and $b$ by $t$ corresponds to dividing $\Delta_2(X)$ by $t$ and preserving $\Delta_1(X),\Delta_3(X),\Delta_4(X)$. This corresponds to multiplying $\Delta_2(\cev\tau(X))$ by $t$. Similarly, multiplying each of $f,g,h$ by $t$ corresponds to multiplying $\Delta_1(\cev\tau(X))$ by $t$. This agrees with~\eqref{eq:gauge_twist}.
\end{example}

\subsection{Moves}\label{sec:shift-moves}

\begin{figure}
  \includegraphics[width=1.0\textwidth]{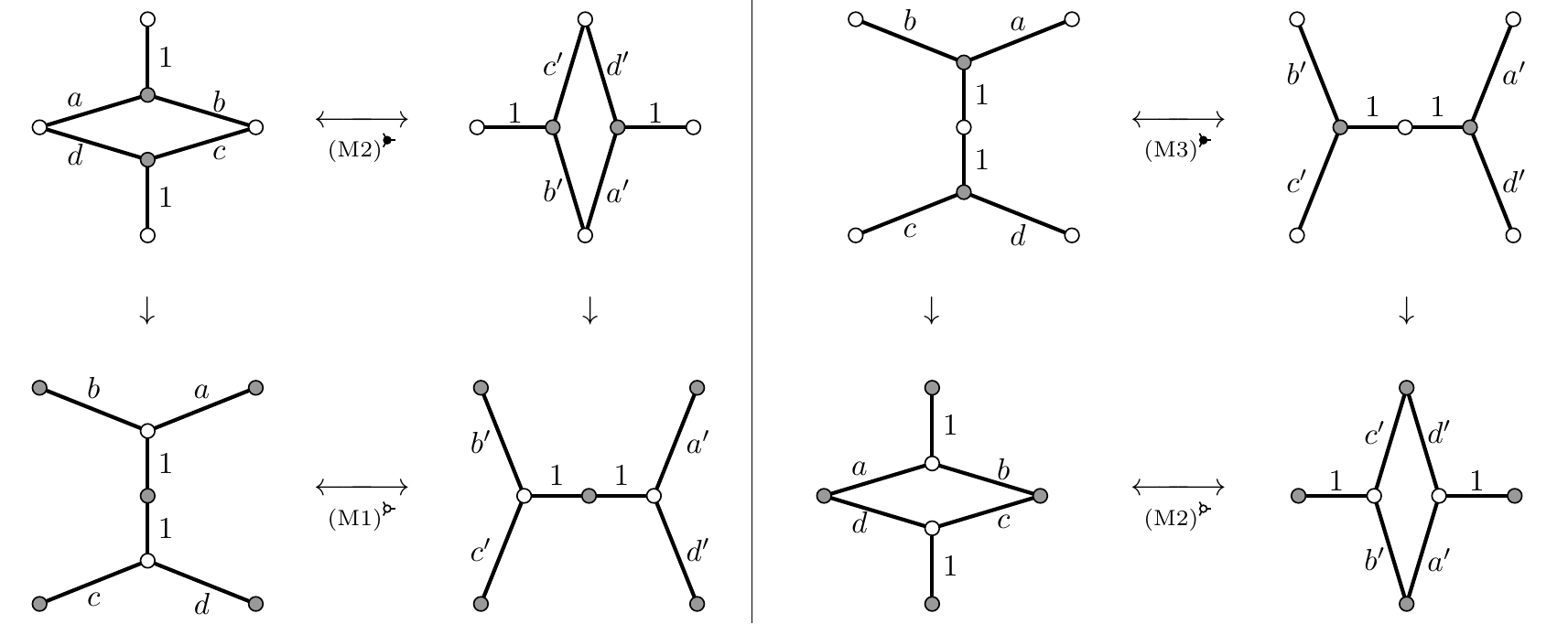}
  \caption{\label{fig:sq-mv-1} The moves for white- and black-trivalent graphs. The vertical maps are given by $(G,\wt)\mapsto(\dsh G,\dsh \wt)$. Here $a':=\frac a{ac+bd}$, etc. as in \cref{fig:sq-mv}.}
\end{figure}

Next, we introduce moves for weighted white- and black-trivalent graphs, as shown in \cref{fig:sq-mv-1}. Here we have used the allowed gauge transformations at black (resp., white) trivalent vertices to make the weights of certain edges to be equal to $1$. It follows from the results of~\cite{Pos} that for any $G_1,G_2\in\BTV(f)$, $G_1$ and $G_2$ are connected by moves\footnote{The names (M1), (M2), (M3) for the moves in \cref{fig:sq-mv-1} are taken from~\cite{chord_sep}: these moves are obtained from a single transformation of $3$-dimensional zonotopal tilings by taking horizontal sections by planes $z=k$ for $k=1,2,3$, respectively; see~\cite[Figure~8]{chord_sep}.} \Mb and \Mcb, and $\dsh G_1$ and $\dsh G_2$ are connected by moves \Maw and \Mw. Moreover, we see from \cref{fig:sq-mv-1} that the shift map~\eqref{eq:bgau_wgau} transforms \Mcb into \Mw and \Mb into \Maw. We therefore obtain the following consequence of \cref{prop:MeMe}.

\begin{corollary}\label{cor:monodromy}
Let $f\in\Bkn$ be loopless, and suppose that $G_1,G_2\in\BTV(f)$ are related by either \Mb or \Mcb. Then we have commutative diagrams
\begin{equation*}%
    \begin{tikzcd}[column sep=5ex,row sep=3ex]
\Rtp^{E(G_1)}/\bGau \arrow[d,"\sim"']\arrow[r,"\sim","{\text{\normalfont\Mbsmall}}"'] &\Rtp^{E(G_2)}/\bGau \arrow[d,"\sim"]\\
\Ptp_f\times\Ptp_{\dsh f} \arrow[r,"\sim","{\id}"'] & \Ptp_f\times\Ptp_{\dsh f}\\
\Rtp^{E(\dsh G_1)}/\wGau \arrow[u,"\sim"]\arrow[r,"\sim","{\text{\normalfont\Mawsmall}}"'] &\Rtp^{E(\dsh G_2)}/\wGau, \arrow[u,"\sim"']
\end{tikzcd} \quad\text{resp.,}\quad  \begin{tikzcd}[column sep=5ex,row sep=3ex]
\Rtp^{E(G_1)}/\bGau \arrow[d,"\sim"']\arrow[r,"\sim","{\text{\normalfont\Mcbsmall}}"'] &\Rtp^{E(G_2)}/\bGau \arrow[d,"\sim"]\\
\Ptp_f\times\Ptp_{\dsh f} \arrow[r,"\sim","{\id}"'] & \Ptp_f\times\Ptp_{\dsh f}\\
\Rtp^{E(\dsh G_1)}/\wGau \arrow[u,"\sim"]\arrow[r,"\sim","{\text{\normalfont\Mwsmall}}"'] &\Rtp^{E(\dsh G_2)}/\wGau \arrow[u,"\sim"'],
\end{tikzcd} 
\end{equation*}
where all vertical maps are given by~\eqref{eq:Meas_times}.%
\end{corollary}
\begin{remark}\label{rmk:monodromy}
It follows that if a sequence of moves relates $G\in\BTV(f)$ to itself, the induced map  $\Rtp^{E(G)}/\bGau\xrasim\Rtp^{E(G)}/\bGau$ must be the identity map since after applying the homeomorphism~\eqref{eq:Meas_times}, we get the identity map $\Ptp_f\times\Ptp_{\dsh f}\xrasim \Ptp_f\times\Ptp_{\dsh f}$ by \cref{cor:monodromy}. This confirms that the space $\Pmid_f$ depends canonically only on $f$ and not on the choice of $G\in\BTV(f)$.
\end{remark}

\subsection{Shift for critical varieties}\label{sec:shift_commut_diag}

Even though the above constructions follow naturally from the geometry of zonotopal tilings, we discovered them in relation to critical varieties, as we now explain. Let $f\in\Bkn$ be loopless. Observe first that $\bth\in\R^n$ is $f$-admissible if and only if it is $\dsh f$-admissible, since the reduced strand diagram of $f$ coincides with the dual reduced strand diagram of $\dsh f$; compare e.g. \figref{fig:pis-pel}(b) and \figref{fig:pis-pel}(d). Thus we have a homeomorphism $\THtp_f\xrasim \THtpd_{\dsh f}$ given by the identity map $\bth\mapsto\bth$.

\begin{proposition}\label{prop:commut_diag}
For each loopless $f\in\Bkn$,  we have the commutative diagram in~\figref{fig:cube}(a). Moreover, if the injectivity conjecture (\cref{conj:inj}) holds for $f$, we have the commutative diagram in~\figref{fig:cube}(b), where the dashed arrow is the composition of the three homeomorphisms in the square on the left hand side.
\begin{figure}[h]
\begin{tabular}{c|c}%
$\begin{tikzcd}[column sep=3.1ex,row sep=0ex]
 & & \Rtp^E/\bGau\arrow[dr,twoheadrightarrow] \arrow[ddd,pos=0.2,swap,"\sim"] & \\
 \THtp_f\arrow[urr]\arrow[ddd,"\sim"]\arrow[dr,twoheadrightarrow, swap,near end] & & & \Ptp_f \arrow[from=3-2,hookrightarrow,crossing over]\\
 & \Ctp_f  &  & \\
 & & \Rtp^{\dsh E}/\wGau\arrow[dr,twoheadrightarrow] & \\
 \THtpd_{\dsh f}\arrow[urr]\arrow[dr,swap,near end,twoheadrightarrow,] & & & \Ptp_{\dsh f}\\
 & \Ctpd_{\dsh f} \arrow[urr,hookrightarrow]  &  & 
  \end{tikzcd}$ &
  $\begin{tikzcd}[column sep=3.1ex,row sep=0ex]
 & & \Rtp^E/\bGau\arrow[dr,twoheadrightarrow] \arrow[ddd,pos=0.2,swap,"\sim"] & \\
 \THtp_f\arrow[urr]\arrow[ddd,"\sim"]\arrow[dr,swap,"\sim",near end] & & & \Ptp_f \arrow[from=3-2,hookrightarrow,crossing over]\\
 & \Ctp_f  &  & \\
 & & \Rtp^{\dsh E}/\wGau\arrow[dr,twoheadrightarrow] & \\
 \THtpd_{\dsh f}\arrow[urr]\arrow[dr,"\sim",swap,near end] & & & \Ptp_{\dsh f}\\
 & \Ctpd_{\dsh f} \arrow[urr,hookrightarrow] \arrow[from=3-2,dashed,crossing over,near start,swap,"\sim"] &  & 
  \end{tikzcd}$ 
\\%
(a) arbitrary loopless $f\in\Bkn$ & (b) assuming the injectivity conjecture for $f$ %
\end{tabular}
  \caption{\label{fig:cube} The shift map for critical varieties.}
\end{figure}
\end{proposition}
\noindent Note that $\Ptp_f\not\cong\Ptp_{\dsh f}$ since these spaces usually have different dimension. Instead, they have the same \emph{codimension} $\ell(f)=\ell(\dsh f)$ inside $\Gr(k,n)$, resp., $\Gr(k-1,n)$.

\begin{remark}
By \cref{thm:inj_fkn}, the dashed arrow in~\figref{fig:cube}(b) yields a homeomorphism
\begin{equation*}%
  \Ctp_{k,n}\xrasim \Ctpd_{k-1,n}
\end{equation*}
for each $1\leq k\leq n$.
\end{remark}

\begin{proof}[Proof of \cref{prop:commut_diag}]
  First, let us explain the maps in \cref{fig:cube}. The map $\THtp_f\xrasim \THtpd_{\dsh f}$ is the identity, and  $\THtp_f\to \Ctp_f$ and $\THtpd_{\dsh f}\to\Ctpd_{\dsh f}$ are given by $\Measop_f$ and $\Measdop_{\dsh f}$, respectively. The maps $\Ctp_f\hookrightarrow \Ptp_f$ and $\Ctpd_f\hookrightarrow \Ptp_{\dsh f}$ are inclusions of subsets, while  $\THtp_f\to\Rtp^E/\bGau$ and $\THtpd_{\dsh f}\to\Rtp^{\dsh E}/\wGau$ are given by $\bth\mapsto \wt_\bth$, resp., $\bth\mapsto \wtd_\bth$. We have already explained the dashed arrow in~\figref{fig:cube}(b), and the remaining map $\Rtp^E/\bGau\xrasim \Rtp^{\dsh E}/\wGau$ is given by~\eqref{eq:bgau_wgau}. 

The commutativity of the top and bottom squares is obvious. The commutativity of 
\begin{equation*}%
  \begin{tikzcd}%
 \THtp_f\arrow[r]\arrow[d,swap,"\sim"] &  \Rtp^E/\bGau \arrow[d,swap,"\sim"]\\
 \THtpd_f\arrow[r] &  \Rtp^{\dsh E}/\wGau
  \end{tikzcd}
\end{equation*}
follows from the observation in \cref{sec:shift-positr-cells} that whenever an interior edge $e\in E$ is labeled by $\{p,q\}$, the edge $\dsh e\in \dsh E$ is also interior and labeled by $\{p,q\}$. The left square in~\figref{fig:cube}(b) commutes by construction.
\end{proof}

\begin{question}
Which of the above results extend to the complex algebraic setting (with $\Pio_f$ replacing $\Ptp_f$ and $\Cio_f$ replacing $\Ctp_f$)?
\end{question}

\section{Proofs of the results for the top cell}\label{sec:proofs_top_cell}
Our goal is to prove the injectivity conjecture (\cref{thm:inj_fkn}) and the description (\cref{thm:real_fkn}) of the  real points of an open critical variety in the case $f=\fkn$. We start by stating several auxiliary lemmas.

\begin{lemma}%
Let $(t_1,t_2,t_3)\in(\Cast)^3$ be a generic triple. 
\begin{theoremlist}
\item \label{lemma:triang:R} If $\frac{\br[t_2,t_1]}{\br[t_3,t_1]},\frac{\br[t_3,t_2]}{\br[t_3,t_1]}\in\R$ then either $|t_1|=|t_2|=|t_3|$ or $\frac{t_1}{t_2},\frac{t_2}{t_3}\in\R$.
\vspace{0.05in}
\item \label{lemma:triang:iR} If $\frac{\br[t_2,t_1]}{\br[t_3,t_1]},\frac{\br[t_3,t_2]}{\br[t_3,t_1]}\in i\R$ then $\frac{t_1}{t_2},\frac{t_2}{t_3}\in i\R$.%
\end{theoremlist}
\end{lemma}
\begin{proof}
Since $\br[t_q,t_p]$ is invariant under rescaling $t_p$ and $t_q$ by the same nonzero constant, we may assume that $t_2=1$. In this case, the resulting systems of equations on the real and imaginary parts of $t_1$ and $t_3$ may be solved explicitly, and both statements are then verified directly in a straightforward fashion.
\end{proof}
 
\begin{lemma} \label{lemma:quad}%
Let $(t_1,t_2,t_3,t_4)\in(\Cast)^4$ be a generic quadruple satisfying
\begin{equation*}%
  \frac{\br[t_2,t_1]}{\br[t_3,t_2]},  \frac{\br[t_3,t_2]}{\br[t_4,t_3]},   \frac{\br[t_4,t_3]}{\br[t_4,t_1]},  \frac{\br[t_4,t_1]}{\br[t_2,t_1]}\in\R.
\end{equation*}
Then either
\begin{equation}\label{eq:quad}
  |t_1|=|t_2|=|t_3|=|t_4|, \quad \frac{t_1}{t_2},\frac{t_2}{t_3},\frac{t_3}{t_4},\frac{t_4}{t_1}\in \R, \quad\text{or}\quad \frac{t_1}{t_2},\frac{t_2}{t_3},\frac{t_3}{t_4},\frac{t_4}{t_1}\in i\R.
\end{equation}
\end{lemma}
\begin{proof}
Denote
\begin{equation*}%
  z_1:=\br[t_2,t_1],\quad z_2:=\br[t_3,t_2],\quad z_3:=\br[t_4,t_3],\quad z_4:=\br[t_4,t_1],\quad x:=\br[t_3,t_1],\quad y:=\br[t_4,t_2].
\end{equation*}
Then we have relations
\begin{equation*}%
  x^2=\frac{(z_1z_3+z_2z_4)(z_1z_4+z_2z_3)}{z_1z_2+z_3z_4} \quad\text{and}\quad   y^2=\frac{(z_1z_3+z_2z_4)(z_1z_2+z_3z_4)}{z_1z_4+z_2z_3}.
\end{equation*}
Let $\eps\in\Cast$, $|\eps|=1$ be such that $\eps z_j\in\R$ for $j=1,2,3,4$. Then $(\eps x)^2,(\eps y)^2\in\R$, thus $\eps x,\eps y\in \R\cup i\R$. Additionally, by~\eqref{eq:Pluck_t}, we have $xy=z_1z_3+z_2z_4$, so $\eps x\cdot \eps y\in\R$. Thus either $\eps x,\eps y\in \R$ or $\eps x, \eps y\in i\R$. In the former case, by \cref{lemma:triang:R}, we get either $|t_1|=|t_2|=|t_3|=|t_4|$ or $\frac{t_1}{t_2},\frac{t_2}{t_3},\frac{t_3}{t_4},\frac{t_4}{t_1}\in \R$. In the latter case, by \cref{lemma:triang:iR}, we get $\frac{t_1}{t_2},\frac{t_2}{t_3},\frac{t_3}{t_4},\frac{t_4}{t_1}\in i\R$.
\end{proof}

The next well-known result states that knowing the ratios of side lengths of an inscribed convex polygon is sufficient to reconstruct its angles.
\begin{lemma}\label{lemma:inscribed_polygon}
Let $(a_1,a_2,\dots,a_m)\in\Rtp^m$ be such that no $a_p$ is greater than $\sum_{q\neq p} a_q$. Then there exists a unique tuple $(\th_1,\th_2,\dots,\th_m)$ such that $0=\th_1<\th_2<\dots<\th_m<\pi$, and 
\begin{equation}\label{eq:inscribed_polygon}
  \frac{\sin(\th_{p+1}-\th_p)}{\sin(\th_{q+1}-\th_{q})}=\frac{a_p}{a_q} \quad\text{for all $p, q\in [m]$,}
\end{equation}
 where we set $\th_{m+1}:=\pi$. \qed
\end{lemma}

\begin{proof}[Proof of \cref{thm:inj_fkn,thm:real_fkn}]
Let $f=\fkn$. Recall that \cref{thm:inj_fkn} is stated for $1\leq k\leq n$ while \cref{thm:real_fkn} only deals with the case $2\leq k\leq n-2$. First, \cref{thm:inj_fkn} is trivial to see when $k=1$ or $k=n$ because each of $\THtp_f$ and $\Ctp_{f}$ consists of a single point in these cases. When $k=n-1$, $\Meas(f,\bth)$ records the ratios of sines on the left hand side of~\eqref{eq:inscribed_polygon}, and thus the injectivity conjecture in this case follows from \cref{lemma:inscribed_polygon}. From now on, we restrict to the case $2\leq k\leq n-2$ and prove \cref{thm:inj_fkn,thm:real_fkn} simultaneously. (Note in particular that each $\bt$ corresponding to $\bth\in \THtp_f$ via~\eqref{eq:th_to_t} is automatically generic.)

By the results of~\cite{OPS}, for all quadruples $1\leq a<b<c<d\leq n$, there exists a graph $G\in\Gred(f)$ containing a square face $F$ whose boundary edges are labeled by $\{a,b\}$, $\{b,c\}$, $\{c,d\}$, and $\{a,d\}$. Denote $\v_p:=t_p^2$ for $p\in[n]$ as above. Consider a \emph{cross-ratio} 
\begin{equation}\label{eq:cross_rat}
\crat(a,b;c,d):=\frac{\H ca \H db}{\H cb\H da}=\frac{(\v_c-\v_a)(\v_d-\v_b)}{(\v_c-\v_b)(\v_d-\v_a)}.
\end{equation}
We have $\crat(a,b;c,d)=1-\crat(a,c;b,d)$, and by~\cite[Corollary~5.11]{MuSp}, $\crat(a,c;b,d)$ may be written as a ratio of Pl\"ucker coordinates of the left twist $\cev\tau(\Meas(f,\bt))$ corresponding to the four faces of $G$ adjacent to $F$. All such Pl\"ucker coordinates are monomials in the edge weights which are nonzero since $\bt$ is generic. In particular, $\crat(a,b;c,d)\in \R$ when $\Meas(f,\bt)\in \Cio_f(\R)$, and moreover, $\crat(a,b;c,d)$ may be reconstructed from $\Meas(f,\bt)$. We observe that in the setting of \cref{thm:real_fkn}, all cross-ratios $\crat(a,b;c,d)$ are real, and therefore the (pairwise distinct) points $(\v_p)_{p\in[n]}$ all belong to the same circle or to the same line. 

Next, we see that $f$ has a bridge at $r$ for all $r\in[n]$. By~\cite[Lemma 7.6]{LamCDM}, we find that for each $p\in[n]$, $\H p{p-k}/\H p{p-1}$ may be written as a ratio of two Pl\"ucker coordinates of $\Meas(f,\bt)$, where the indices are taken modulo $n$. We thus see that
\begin{equation}\label{eq:bridge_R}
  \H p{p-k}/\H p{p-1}\in\R \quad\text{for $p\in[n]$}
\end{equation}
and $\H p{p-k}/\H p{p-1}$ may be reconstructed from $\Meas(f,\bt)$.

The cross-ratio $\crat(a,b;c,d)\in\R$ changes predictably under permuting the indices. Substituting $\{a,b,c,d\}=\{1,2,k,k+1\}$ or $\{a,b,c,d\}=\{1,2,k+1,k+2\}$ into~\eqref{eq:cross_rat} and $p=k+1$ or $p=k+2$ into~\eqref{eq:bridge_R}, we find 
\begin{equation}\label{eq:quad_proof}
  \frac{\br[t_2,t_1]}{\br[t_k,t_2]},  \frac{\br[t_k,t_2]}{\br[t_{k+1},t_k]},   \frac{\br[t_{k+1},t_k]}{\br[t_{k+1},t_1]},  \frac{\br[t_{k+1},t_1]}{\br[t_2,t_1]}\in\R
\end{equation}
and that all these ratios may be reconstructed from $\Meas(f,\bt)$. Applying \cref{lemma:quad}, we find that the points $t_1,t_2,t_k,t_{k+1}$ satisfy the conditions in~\eqref{eq:quad}. Recall that we have shown above that all points $(\v_p)_{p\in[n]}$ belong to either a common circle or a common line. By~\eqref{eq:quad}, if it is a circle then its center must pass through $0$, and if it is a line then it must pass through the origin. In the case of the line, we note additionally that~\eqref{eq:quad} implies $\frac{t_1}{t_k}\in\R$, and by cyclic symmetry, we must have $\frac{t_p}{t_{p+k-1}}\in\R$ for all $p\in[n]$ (taken modulo $n$). By the definition of $\Tspace_f$, we have, say, $t_1=1$, and therefore we either have $|t_p|=1$ for all $p\in[n]$ or $t_p\in \R\cup i\R$ for all $p\in[n]$. This finishes the proof of \cref{thm:real_fkn}.

We now focus on \cref{thm:inj_fkn}. Recall that the ratios in~\eqref{eq:quad_proof} may be reconstructed from $\Meas(f,\bt)$. Combining this with \cref{lemma:inscribed_polygon}, we can uniquely reconstruct (modulo shift) the quadruple $(\th_1,\th_2,\th_k,\th_{k+1})$ from $\Meas(f,\bt)$. For each $q\notin \{1,k,k+1\}$, we may plug in $\{a,b,c,d\}=\{1,q,k,k+1\}$ into~\eqref{eq:cross_rat} and $p=k+1$ into~\eqref{eq:bridge_R} to see that $\H q1/\H qk$ can be uniquely reconstructed from $\Meas(f,\bt)$, and therefore having already recovered $\th_1$ and $\th_k$, we can now also recover $\th_q$. We have shown that $\Measop_f:\THtp_f\to \Ctp_f$ is injective (and thus bijective). It is clearly continuous and its inverse is also continuous since the map in \cref{lemma:inscribed_polygon} is continuous. Thus $\Measop_f:\THtp_f\xrasim \Ctp_f$ is a homeomorphism, and it remains to note that $\THtp_f$ is homeomorphic to the interior of an $(n-1)$-dimensional simplex. We are done with the proof of \cref{thm:inj_fkn}.
\end{proof}

\section{Further directions}
In addition to the multiple conjectures mentioned in the body of the text, we list a few other questions which arise in relation to critical varieties. The general philosophy is that critical varieties should be considered as \emph{critical parts} of positroid varieties, and thus for each existing result for positroid varieties (resp., open positroid varieties, their totally positive and totally nonnegative parts) an immediate direction is to investigate analogs of that result for critical varieties and critical cells.

\parag{The totally nonnegative part} Taking the closure of a positroid cell $\Ptp_f$ inside $\Grtnn(k,n)$ gives rise to a remarkable topological space denoted $\Ptnn_f$. For instance, it was recently shown in~\cite{GKL3} that $\Ptnn_f$ is a regular CW complex homeomorphic to a closed ball.
\begin{definition}
Let $f\in\Bkn$ be loopless. The \emph{totally nonnegative part} $\Ctnn_f$ of $\Crit_f$ is defined as the closure of $\Ctp_f$ in the usual (Hausdorff) topology on $\Grtnn(k,n)$.
\end{definition}
The topology and combinatorics of the cell structure of $\Ctnn_f$ is explored in a separate paper~\cite{crit_tnn}.
Among other things, we establish the following result.
\begin{theorem}[\cite{crit_tnn}]
  For $2\leq k\leq n-1$, $\Ctnn_{k,n}$ is homeomorphic (via a stratification-preserving homeomorphism) to the hypersimplex $\Delta_{2,n}$.
\end{theorem}
Thus, unlike $\Ptnn_{k,n}$, the space $\Ctnn_{k,n}$ is homeomorphic to a polytope (which does not depend on $k$). What happens in the case of arbitrary loopless $f\in \Bkn$ remains an open question which depends, among other things, on the injectivity conjecture (\cref{conj:inj}). The properties of the associated polytopes are developed in~\cite{crit_polyt}. In particular, the structure of $\Ctnn_{k,n}$ is intimately tied with the \emph{cyclohedron}~\cite{BoTa,Simion}, which is a natural compactification of the space $\THtp_{k,n}$; see~\cite[Theorem~1.3]{crit_tnn}.

\parag{Algebraic geometry of critical varieties} Besides the topology of $\Ctnn_f$, it is also interesting to study $\Crit_f$ and $\Cio_f$ as algebraic varieties, since their positroid counterparts possess many nice properties (normal, Cohen--Macaulay, etc.); see~\cite{KLS}. One particularly optimistic direction is to investigate the cohomology of $\Crit_f$ and $\Cio_f$. The analogous questions for $\Pio_f$ have recently attracted some interest~\cite{GL2}. 

\parag{Peterson variety} On the surface, critical varieties look very similar to the \emph{Peterson variety} $\Pet_n$, which is a certain subvariety of the flag variety introduced by D.~Peterson in the 1990s; see e.g.~\cite{Kostant,Rietsch03}. For example, both $\Pet_n$ and $\Critkn$ have dimension $n-1$. The coordinate ring of $\Pet_n$ is of great importance since it computes the quantum cohomology ring of the flag variety~\cite{Rietsch03}. It would be interesting to see whether there is some projection relating $\Pet_n$ to $\Crit_{k,n}$, and whether the coordinate ring of $\Crit_{k,n}$ (or, more generally, of $\Crit_f$) also recovers some well-studied cohomology ring.

\parag{Chow quotient}
The \emph{torus} $T$ consisting of diagonal $n\times n$ matrices acts on $\Gr(k,n)$ by right multiplication. One can consider the corresponding \emph{Chow quotient} and study its totally nonnegative part, see e.g. the recent results of~\cite{AHLS,LPW}. Observe that the critical varieties and the dual critical varieties only differ by the action of $T$, thus the resulting quotients coincide. The cross-ratios~\eqref{eq:cross_rat} that we used above yield natural coordinates on the Chow quotient, thus it seems plausible that the questions we studied for critical varieties have simpler Chow quotient analogs.

\parag{Boundary measurements for reduced strand diagrams}
Our results (especially \cref{prop:edges_conn_cpts,prop:Meas_exists}) suggest that there could be a formula for $\Measop_f$ directly in terms of the reduced strand diagram of $f$. (For instance, it could be some ``oriented strand version'' of the \emph{six-vertex model}.) An optimistic hope would be that such a formula would potentially provide a positive answer to~\cite[Question~7.6]{scsv}.

\parag{More general edge weights}
Given a reduced graph $G\in\Gred(f)$ and an edge $e\in E(G)$ labeled by $\{p,q\}$ with $1\leq p<q\leq n$, we have set the edge weight to either $\sin(\th_q-\th_p)$ or to $\H qp$. One can consider the following more general assignment of edge weights: choose an element $\BT\in \Gr(2,n)$, and then set $\wt_\BT(e):=\Delta_{p,q}(\BT)$. This weight assignment includes the ones we have considered as special cases, and the resulting dimer model is still invariant under square moves. It appears that most of our constructions extend to this more general set up in a straightforward fashion, in particular, we have the corresponding versions of critical cells ($\BT\in\Grtp(2,n)$) and (open) critical varieties, where the condition $t_p\neq\pm t_q$ is replaced with $\Delta_{p,q}(\BT)\neq0$. The underlying combinatorics is still dictated by reduced strand diagrams. It would be interesting to expand this further, for instance, to extend the boundary measurement formula or the injectivity conjecture to such weights.

\parag{Plabic tilings as isoradial embeddings}
The starting point for this work was the observation that the plabic tilings of~\cite{OPS} may be naturally considered as special cases of isoradial embeddings of planar bipartite graphs. Thus one can study the asymptotic properties of this special family of isoradial embeddings, and potentially apply the vast literature on such embeddings to study convergence and conformal invariance questions for families of plabic tilings. For instance, an interesting question appears to be whether the results of~\cite{CLR} apply in this case, since plabic tilings are easily seen to satisfy the \emph{small origami} property introduced in~\cite{CLR}. We thank Marianna Russkikh for  discussions related to these questions.

\bibliographystyle{alpha_tweaked} 
\bibliography{crit}

\end{document}